\documentclass[]{interact}

\usepackage{epstopdf}
\usepackage[caption=false]{subfig}
\usepackage{epstopdf}

\usepackage{graphicx,epsfig}
\usepackage{enumerate}
\usepackage{amssymb}
\usepackage{amsthm}
\usepackage{amsmath}
\usepackage{centernot}
\usepackage{xcolor}
\usepackage{todonotes}
\usepackage{mathrsfs}
\usepackage[utf8]{inputenc}
\usepackage{multirow}
\usepackage{xcolor}
\usepackage{tikz}
\usetikzlibrary{positioning}
\usepackage{pgfplots}
\usepackage{pst-plot}
\usepackage{amsmath}
\usepackage{amsfonts}
\usepackage{amssymb}
\usepackage{multirow}
\usepackage{multicol}
\usepackage{framed}
\usetikzlibrary{angles,quotes}  
 
\usepackage{mathtools}
\usepackage{appendix}
\usepackage{wrapfig}
\usepackage{amsmath}
\usepackage{setspace}
\usepackage{booktabs}
\usepackage{lscape}
\usepackage{algorithm,algorithmicx}
\usepackage{algcompatible}
\usepackage{bm}
\usepackage{graphicx,epsfig}
\usepackage{xfrac}
\usepackage{url}
\usepackage{multirow}
\usepackage{longtable}
\usepackage[linkcolor=blue, urlcolor=blue, citecolor=blue,
colorlinks, bookmarks]{hyperref}
\usepackage[numbers,sort&compress]{natbib}

\hypersetup{colorlinks,%
citecolor=blue,%
filecolor=blue,%
linkcolor=blue,%
urlcolor=blue
}

\bibpunct[, ]{[}{]}{,}{n}{,}{,}
\renewcommand\bibfont{\fontsize{10}{12}\selectfont}
\makeatletter
\def\NAT@def@citea{\def\@citea{\NAT@separator}}
\makeatother

\theoremstyle{plain}
\newtheorem{theorem}{Theorem}[section]
\newtheorem{lemma}[theorem]{Lemma}

\newtheorem{proposition}[theorem]{Proposition}
\newtheorem{assumption}[theorem]{Assumption}

\theoremstyle{definition}
\newtheorem{definition}[theorem]{Definition}
\newtheorem{example}[theorem]{Example}

\theoremstyle{remark}
\newtheorem{remark}{Remark}

\begin{document}

\articletype{ARTICLE TEMPLATE}

\title{Quasi-Newton Method for Set Optimization Problems with Set-Valued Mapping Given by Finitely Many Vector-Valued Functions}

\author{
\name{Debdas Ghosh\textsuperscript{a}\thanks{Corresponding Author: D. Ghosh (debdas.mat@iitbhu.ac.in)}, Anshika\textsuperscript{a}, Jen-Chih Yao\textsuperscript{b,}\textsuperscript{c}, Xiaopeng Zhao\textsuperscript{d}}
\affil{\textsuperscript{a}Department of Mathematical Sciences, Indian Institute of Technology (BHU), Varanasi, Uttar Pradesh---221005, India}
\affil{\textsuperscript{b}Center for General Education, China Medical University, Taichung, Taiwan}
\affil{\textsuperscript{c}Academy of Romanian Scientists, 50044 Bucharest, Romania}
\affil{\textsuperscript{d}School of Mathematical Sciences, Tiangong University, Tianjin, 300387, China}
}

\maketitle

\begin{abstract}
In this article, we propose a quasi-Newton method for unconstrained set optimization problems to find its weakly minimal solutions with respect to lower set-less ordering. The set-valued objective mapping under consideration is given by a finite number of vector-valued functions that are twice continuously differentiable. At first, we derive a necessary optimality condition for identifying weakly minimal solutions for the considered set optimization problem with the help of a family of vector optimization problems and the Gerstewitz scalarizing function. To find the necessary optimality condition for weak minimal points with the help of the proposed quasi-Newton method, we use the concept of partition and formulate a family of vector optimization problems. The evaluation of necessary optimality condition for finding the weakly minimal points involves the computation of the approximate Hessian of every objective function, which is done by a quasi-Newton scheme for vector optimization problems. In the proposed quasi-Newton method, we derive a sequence of iterative points that exhibits convergence to a point which satisfies the derived necessary optimality condition for weakly minimal points. After that, we find a descent direction for a suitably chosen vector optimization problem from this family of vector optimization problems and update from the current iterate to the next iterate. The proposed quasi-Newton method for set optimization problems is not a direct extension of that for vector optimization problems, as the selected vector optimization problem varies across the iterates. The well-definedness and convergence of the proposed method are analyzed. The convergence of the proposed algorithm under some regularity condition of the stationary points, a condition on nonstationary points, the boundedness of the norm of quasi-Newton direction, and the existence of step length that satisfies the Armijo condition are derived. We obtain a local superlinear convergence of the proposed method under uniform continuity of the Hessian approximation function. Lastly, some numerical examples are given to exhibit the performance of the proposed method. Also, the performance of the proposed method with the existing steepest-descent method is given.  
\end{abstract}


\section{Introduction}
A general set-valued optimization problem is given by 
\begin{equation}\tag{$A$}\label{general_eqn}
 \text{Minimize }F(x)~\text{ subject to }~x\in X,
\end{equation}
where $F$ is a set-valued map from a nonempty subset $X$ of $\mathbb{R}^n$ to $\mathbb{R}^m$, where the ordering of the elements in the image of $F$ is given by a convex cone $K\subset \mathbb{R}^m$. The solution set for the problem (\ref{general_eqn}) provides an important generalization and unification for scalar as well as vector optimization problems using different approaches (see \cite{alonso2008optimality,jahn2004some,kuroiwa1997some}). Optimization problems with set-valued objective functions or set-valued constraints have a number of applications in various areas of mathematical economics \cite{bao2010set}, finance \cite{feinstein2015comparison}, game theory, optimal control, and many others \cite{chen1998optimality,pilecka2016set}. These problems have significant applications in uncertain optimization problems, as discussed in \cite{ide2014concepts,mutapcic2009cutting}.
 
Kuroiwa \cite{kuroiwa1997some,kuroiwa2001set,kuroiwa1997cone} was the first to define solutions of set optimization problems by preorder relations of sets. In \cite{kuroiwa1997some}, solution concepts have been defined based on the approach of comparing the sets that are values of the objective function. Motivated by this work, Khan \cite{khan2016set} gave a detailed discussion on set-valued optimization maps. The existing methods for solving set optimization problems in the literature fall into one of the following groups.  
 
The first group consists of algorithms based on scalarization (see \cite{ehrgott2014minmax,eichfelder2020algorithmic,ide2014concepts}). The methods in this group are characterized by robust counterparts of vector optimization problems. In \cite{ehrgott2014minmax,ide2014concepts}, a linear scalarization technique was used to obtain the optimistic solution of the set optimization problems and extended the $\varepsilon$-constraint method for the ordering cones with nonnegative orthant to deal with the set optimization problems. The main drawback of these methods is that not all the solutions of the problem can be captured by them.

The next group consists of algorithms of sorting type as given in \cite{gunther2019computing,kobis2016treatment,kobis2018numerical}. These methods deal with set optimization problems that have a finite feasible set. In \cite{kobis2016treatment,kobis2018numerical}, K\"obis incorporated the forward and backward reduction to the algorithms of \cite{jahn2006multiobjective,jahn2011new} and proposed extended version of the algorithms. Further, G\"unther and Popovici \cite{gunther2018new}  evaluated the images of set-valued mappings whose values are in an increasing manner by scalarizing via a strongly monotone functional and employed a forward iteration procedure. The sorting-type algorithms cannot be applied to a continuous feasible set and are applicable only to a finite feasible set. Also, an additional computational cost is required for comparing sets. 

{Recently, Bouza et al.\ \cite{bouza2021steepest} pioneered generalizing classical gradient-based algorithm (started with the steepest descent method) for set optimization problems. Ghosh et al. \cite{ghosh2024newton} reported Newton method for set-valued optimization problems. The conventional steepest descent method is known to have a linear convergence rate. Moreover, a Newton method commonly does not converge if the chosen initial point is not close to the optimal solution; in addition, there is a need to compute the inverse of the Hessian at every iteration, which is costly. In quasi-Newton methods, instead of computing the Hessian inverse, an approximation of the Hessian inverse is considered. These methods use only the first-order derivative information to compute such approximation, which comparatively keeps the computational cost very low. The Broyden-Fletcher-Goldfarb-Shanno (BFGS) method is a type of quasi-Newton method in classical optimization that is designed to solve unconstrained nonlinear optimization problems. This is originated from the independent work of Broyden \cite{broyden1969new}, Fletcher \cite{fletcher1970new}, Goldfarb \cite{goldfarb1970family}, and Shanno \cite{shanno1970conditioning}. Several other extensions (or modifications) of this method are given in \cite{dai2002convergence,powell1971convergence,salim2018quasi,byrd1987global,yuan1991modified,povalej2014quasi,kumar2023quasi,mahato2023quasi,singh2024globally,upadhayay2024nonmonotone,ghosh2023infeasible} and references therein.}    

Motivated by the literature, we aim to develop the quasi-Newton method for set-valued optimization problems. The set-valued objective that we consider here is defined by a finite number of twice continuously differentiable vector-valued functions. We use the ideas from \cite{kumar2023quasi,prudente2024global,povalej2014quasi} for vector optimization problems. We approximate the Hessian matrices with the help of BFGS approximation techniques. The proposed method in this study exhibits a superlinear convergence rate and works well for highly nonlinear objective functions. 

The paper is arranged in the following order. In Section \ref{section2}, we discuss the preliminaries, basic notations, some fundamental definitions, and results that will be used throughout the paper. Section \ref{section3} consists of results on optimality conditions for weakly minimal solutions of set optimization problems. We discuss the interrelation of stationary points with weakly minimal solutions. Next, in Section \ref{section4}, we propose the quasi-Newton method for the considered set optimization problem. We report the well-definedness of the proposed algorithm with the existence result of Armijo step length and the boundedness of the descent direction. After that, we analyze the convergence of the proposed quasi-Newton method in Section \ref{section5}. Further, in Section \ref{section6}, we show the numerical implementation of our method with the help of suitable examples. Subsequently, we compare the results of the proposed algorithm with the results of the steepest descent method presented in \cite{bouza2021steepest}. Finally, we conclude the paper in Section \ref{section7} by summarizing the results and provide some ideas for further research in this direction. 
 
\section{\textbf{Preliminaries and Definitions}}\label{section2}
Throughout the paper, we use the following notations. 
\begin{itemize}
\item $\mathbb{R},~\mathbb{R}_{+},~\mathbb{R}_{++}$ denote the set of real numbers, nonnegative real numbers, and positive real numbers, respectively.
    \item $\mathbb{R}^m=\mathbb{R}\times\mathbb{R}\times\cdots\times\mathbb{R}$ ($m$-times), $\mathbb{R}_+^m=\mathbb{R}_+\times\mathbb{R}_+\times\cdots\times\mathbb{R}_+$, and $\mathbb{R}_{++}^m=\mathbb{R}_{++}\times\mathbb{R}_{++}\times\cdots\times\mathbb{R}_{++}$.
    \item $\mathcal{P}(\mathbb{R}^m)$ denotes the class of all nonempty subsets of $\mathbb{R}^m$.
    \item For any nonempty set $A\in\mathcal{P}(\mathbb{R}^m)$, the notations int$(A)$, cl($A$),  $|A|$, bd($A$), and conv($A$) denote the interior, closure, cardinality, boundary, and convex hull, respectively, of the set $A$.
    \item $\top$ denotes the transpose operator, and all the elements in $\mathbb{R}^n$ are column vectors.
    \item $\lVert\cdot\rVert$ stands for the standard Euclidean norm of a vector or spectral norm of a matrix.
    \item For a given $k\in\mathbb{N}$, $[k]$ represents the set $\{1,2,\ldots,k\}$.
    \item A cone $K \subseteq \mathbb{R}^m$ is said to be convex if $K+K=K$, solid if int$(K)\not=\emptyset$, and pointed if $K\cap(-K)=\{0\}$.
    \item Throughout, the notation $K\in\mathcal{P}(\mathbb{R}^m)$ represents a closed, convex, pointed, and solid cone. 
    \item The set $K^*=\{y\in\mathbb{R}^m~:~y^\top z\geq0 \text{ for all }z\in K\}$ represents the dual cone of $K$.
\item $\nabla f^i(x)$ or $J f^i(x)$ denotes the Jacobian of a vector-valued function $f^i: \mathbb{R}^n \to \mathbb{R}^m$ at $x$. 
\item $\mathcal{R}(J f^i(x))$ denotes the range or image space of $J f^i(x)$.
    \item For a vector-valued function $f^i: \mathbb{R}^n \to \mathbb{R}^m$, given by $$f^i(x) : = \left(f^{i, 1}(x), f^{i, 2}(x), \ldots, f^{i, m}(x)\right)^\top,$$ the notation $\nabla^2 f^i(x)$ denotes the follow matrix: 
    $$\nabla^2 f^i(x) := \left(\nabla^2 f^{i, 1}(x), \nabla^2 f^{i, 2}(x), \ldots, \nabla^2 f^{i, m}(x)\right)^\top.$$
\end{itemize}

Next, we discuss some basic definitions and results from set optimization that are used throughout the paper.
\begin{definition}(Partial ordering on $\mathbb{R}^m$ \cite{gopfert2003variational}).\label{partial} For any $y,z\in\mathbb{R}^m$, the cone $K$ generates a partial order $(\preceq)$ and a strict order $(\prec)$ on $\mathbb{R}^m$ defined as follows: 
\[
y\preceq z\iff z-y\in K,\text{ and }y\prec z\iff z-y\in \text{ int}(K).
\] 
\end{definition}

\begin{definition}(Minimal and weakly minimal elements of a set \cite{jahn2009vector}). The set of minimal and weakly minimal elements of $A\in\mathcal{P}(\mathbb{R}^m)$ with respect to $K$ is defined by 
    \begin{eqnarray*} 
    &&\text{Min}(A,K)=\{z\in A:(z-K)\cap A=\{z\}\}\text{ and }\\
    &&\text{WMin}(A,K)=\{z\in A:(z-\text{ int}(K))\cap A=\emptyset\},\text{ respectively}.
     \end{eqnarray*}
\end{definition}

\begin{proposition}\label{min_min} \emph{\cite{jahn2009vector}} Let $A\in \mathcal{P}(\mathbb{R}^m)$ be any compact set. Then, $A$ satisfies the domination property with respect to $K$, i.e., $A+K=\emph{\text{Min}}(A, K) + K$. 
\end{proposition}

Below, we discuss the Gerstewitz scalarizing function, which is an important part of the main results of the paper.
\begin{definition}(Gerstewitz function \cite{gerstewitz1983nichtkonvexe}). For an element $e\in \text{int}(K)$ and $y\in\mathbb{R}^m$, the Gerstewitz function $\mathcal{G}_e:\mathbb{R}^m\to \mathbb{R}$ associated with $e$ and $K$ is  defined by 
\[
\mathcal{G}_e(y)= \min\{t\in\mathbb{R}~:~te\in y+K\}.
\]
\end{definition}

\begin{proposition}\label{gerstewitz}\emph{(See \cite{khan2016set})}.
For a given element $e\in\emph{\text{int}}(K$), the function $\mathcal{G}_e$ has the following properties:
 \begin{enumerate}
     \item[(i)] $\mathcal{G}_e$ is sublinear on $\mathbb{R}^m$.\label{ge_1} 
     
     \item[(ii)] $\mathcal{G}_e$ is positive homogenous of degree 1 on $\mathbb{R}^m$.\label{ge_2} 

     \item[(iii)] $\mathcal{G}_e$ is Lipschitz continuous on $\mathbb{R}^m$.\label{ge_3}
     \item[(iv)] $\mathcal{G}_e$ is monotone: for any $x,y\in \mathbb{R}^m$,
     \[
      x\preceq y\implies \mathcal{G}_e(x)\leq\mathcal{G}_e(y)
     ~~~\text{ and }~~ x \prec y\implies \mathcal{G}_e(x)<\mathcal{G}_e(y).
     \]  
    
    \item[(v)]\label{ge_5} $\mathcal{G}_e$ satisfies the representability property, i.e., 
     \[
     -K=\{x\in\mathbb{R}^m:\mathcal{G}_e(x)\leq0\}\text{ and } -\emph{\text{int}}(K)=\{x\in\mathbb{R}^m:\mathcal{G}_e(x)<0\}.
     \]
     \item[(vi)] \label{ge_6} $\mathcal{G}_e$ has the translativity property, i.e., 
     \[
     \mathcal{G}_e(x + te) = \mathcal{G}_e(x) + t ~\text{ for all }~ x \in \mathbb{R}^m.  
     \]
\end{enumerate}
\end{proposition}

Next, to deal with set-valued functions, we discuss the set order relations between the nonempty subsets of $\mathbb{R}^m$. 

\begin{definition}(Lower set less relations \cite{kuroiwa1997cone}).\label{set_less}
For $A$ and $B$ in $\mathcal{P}(\mathbb{R})^m$, the lower set less $(\preceq^l)$ and strict lower set less $(\prec^l)$ relations with respect to a given cone $K$ are defined by
\[
A\preceq^l B\iff B\subseteq A+K ~\text{ and }~ A\prec^l B\iff B\subseteq A+\text{ int}(K),\text{ respectively}.
\]
\end{definition}

\begin{framed}
\noindent
In this paper, we aim to derive a quasi-Newton method to identify weakly minimal solutions to the following unconstrained set optimization problem. Let $F:\mathbb{R}^n \rightrightarrows \mathbb{R}^m$ be a nonempty set-valued mapping. The unconstrained set optimization problem that we study is defined as follows: 
\[
{\preceq^l}\text{--}\underset{x\in\mathbb{R}^n}{\min}~F(x), \label{sp_equation} \tag{SOP}
\] 
where the solution concept is given by the set of weakly minimal solutions with respect to a given ordering cone $K$ in $\mathbb{R}^m$ as given below in Definition \ref{weakly_minimal_solution}.
\end{framed}

\begin{definition}\label{weakly_minimal_solution}(Weakly minimal solution of (\ref{sp_equation}) \cite{bouza2021steepest}). A point $\bar{x}\in\mathbb{R}^n$ is called a local weakly minimal solution of (\ref{sp_equation}) if there exists a neighbourhood $U\subset \mathbb{R}^n$ of $\bar{x}$ such that there does not exist any $x\in U$ with $F(x)\prec^l F(\bar{x}).$ If $U=\mathbb{R}^n$, then $\bar{x}$ is a weakly minimal solution of (\ref{sp_equation}). 
\end{definition}

\begin{framed}
\begin{assumption} \label{assumption} 
In order to deal with set optimization problem \eqref{sp_equation}, we assume the following structure of $F$ throughout the paper. The function $F:\mathbb{R}^n\rightrightarrows\mathbb{R}^m$ in \eqref{sp_equation} is given by finitely many functions as follows: 
\[
F(x)=\left\{f^1(x),f^2(x),\ldots,f^p(x)\right\},~ x\in\mathbb{R}^n, 
\]
where $f^1,f^2,\ldots,f^p:\mathbb{R}^n\to\mathbb{R}^m$ are twice continuously differentiable on $\mathbb{R}^n$.  
\end{assumption}
\end{framed}

\section{Optimality Conditions for Set-Valued Mappings}\label{section3}
In this section, we discuss results on optimality conditions for weakly minimal solutions of (\ref{sp_equation}) under Assumption \ref{assumption}. These notions are the foundation for constructing the proposed quasi-Newton method to capture weakly minimal solutions of \eqref{sp_equation}. The contribution of this section is divided into two subsections. 

\begin{enumerate}[(i)] 
    \item To identify the sequence of iterates in the proposed quasi-Newton method, we figure out a family of vector optimization problems using the concept of partition set at a point. Thereafter, we discuss the concept of stationary point for \eqref{sp_equation} and interrelate stationary points with weakly minimal solutions using the defined family of vector optimization problems. 
    \item We derive a necessary optimality condition for weakly minimal points of \eqref{sp_equation}. In the process of evaluating these points, we approximate the Hessian corresponding to each objective function $f^1,f^2,\ldots,f^p$ with the help of the BFGS methods for vector optimization problems {(see \cite{povalej2014quasi,prudente2024global,kumar2023quasi})}. 
\end{enumerate}

\subsection{\texorpdfstring{Family of Vector Optimization Problems to Find Weakly Minimal Solutions of \eqref{sp_equation}}{Family of Vector Optimization Problems to Find Weakly Minimal Solutions}}
We begin by discussing some index-related set-valued mappings as given in \cite{bouza2021steepest}. 

\begin{definition}(Active indices for set-valued maps \cite{bouza2021steepest}).
\begin{enumerate}
\item[(i)] The active index of minimal elements associated with the set-valued mapping $F$ of (\ref{sp_equation}) is $I:\mathbb{R}^n\rightrightarrows[p]$, defined as
\[
I(x)=\{i\in[p]: f^i(x)\in\text{ Min}(F(x),K)\}.
\]
\item[(ii)] The active index of weakly minimal elements associated with the set-valued mapping $F$ is $I_W:\mathbb{R}^n\rightrightarrows[p]$, defined by
\[
I_W(x)=\{i\in[p]: f^i(x)\in\text{ WMin}(F(x),K)\}.
\]
\item[(iii)] For a given $r\in\mathbb{R}^m$, the set-valued mapping $I_r:\mathbb{R}^n\rightrightarrows[p]$ is given by 
\[ 
I_r(x)=\{i\in I(x): f^i(x)=r\}.
\] 
It is to notice that for any $u\in\mathbb{R}^m$, $I_u(x)= \emptyset $ for $p \notin \text{Min}(F(x),K)$;~ $I_u(x) \cap I_v(x) = \emptyset$ for any $u \neq v \in \mathbb{R}^m$ and $I(x)=\bigcup\limits_{u\in\text{Min}(F(x),K)} I_u(x)$.
\end{enumerate}
\end{definition}

\begin{definition}\label{cardinality_w}(Cardinality of a set of minimal elements \cite{bouza2021steepest}).
The map $w:\mathbb{R}^n\to \mathbb{N}\cup\{0\}$, which is defined by 
\[w(x)= |\text{Min}(F(x),K)|\]
is called the cardinality of the set of minimal elements of $F(x)$ with respect to $K$. Further for simplicity, we use $\bar{w}=w(\bar{x})$, where $\bar{x}\in\mathbb{R}^n$. 
\end{definition}

\begin{definition}(Partition set at a point \cite{bouza2021steepest}).
Let us consider an element $x\in\mathbb{R}^n$ and an enumeration 
$\{r_1^x,r_2^x,\ldots,r^x_{w(x)}\}$ of the set $\text{Min}(F(x),K)$. The partition set at $x$ is defined by $P_x=\prod_{j=1}^{w(x)} I_{r_j^x}(x).$
\end{definition}

Throughout the paper, for a given iterative point $x_k \in \mathbb{R}^n$, a generic element of the partition set $P_{x_k}$ is denoted by $a^k$. For every $j\in[w(x_k)]$, we denote the $j$-th component of $a^k$ by $a^k_j$, where $k = 1, 2, 3, \ldots$. Specifically, if $\lvert P_{x_k}\rvert=p_k$ and Min$(F(x_k),K) = \{r_1^{x_k},r_2^{x_k},\ldots,r^{x_k}_{w(x_{k})}\}$, then 
\[ P_{x_k}=\left\{a^1,a^2,\ldots,a^{p_k}\right\},\]
where for each $k=1,2,3,\ldots,p_k$, 
\[ a^k=\left(a^k_1,a^k_2,\ldots,a^k_{w(x_k)}\right),~a_j^k \in I_{r_{j}^{x_k}},~j\in [w(x_k)]. \]

Now, we discuss a family of vector optimization problems that will help to find the weakly minimal solutions of (\ref{sp_equation}). We start with the following result from \cite{bouza2021steepest}.

\begin{theorem}\emph{(See \cite{bouza2021steepest}).}\label{equivalence_relation}
Let $P_{\bar{x}}$ be the partition set at $\bar{x}$ and $\bar{w} = w(\bar{x})$. For every $a=(a_1,a_2,\ldots,a_{\bar{w}})\in P_{\bar{x}}$, define a vector-valued function  $\widetilde{f}^a:\mathbb{R}^n\to\prod_{j=1}^{\bar{w}} \mathbb{R}^m$ by 
\[ 
\widetilde{f}^a(x)= \left(
f^{a_1}(x), f^{a_2}(x), \ldots, f^{a_{\bar{w}}}(x)\right)^\top.  
\] 
Let $\widetilde{K}\in\mathcal{P}(\mathbb{R}^{m\bar{w}})$ be the cone given by $\widetilde{K}=\prod_{j=1}^{\bar{w}} K,$ and $\preceq_{\widetilde{K}}$ denote  the partial order in $\mathbb{R}^{m\bar{w}}$ induced by $\widetilde{K}$. Then, $\bar{x}$ is a local weakly minimal solution of (\ref{sp_equation}) if and only if for every $a\in P_{\bar{x}}$, $\bar{x}$ is a local weakly minimal solution of the vector optimization problem 
 \[
{\preceq_{\widetilde{K}}}\text{--}\underset{x\in\mathbb{R}^n}{\min}~\widetilde{f}^a(x).\tag{VOP}\label{vp_equation}
 \]
\end{theorem}

Next, to find a necessary optimality condition for weakly minimal solutions of (\ref{sp_equation}),  we provide the concept of a stationary point for  (\ref{sp_equation}).

\begin{definition}(Stationary points of \eqref{sp_equation}  \cite{bouza2021steepest}).\label{stationary_dfn} 
A point $\bar{x}$ is called a stationary point of (\ref{sp_equation}) if for every $a=(a_1,a_2,\ldots,a_{\bar{w}})\in P_{\bar{x}}$ and $u\in\mathbb{R}^n$, there exists $j\in[\bar{w}]$ such that 
\begin{eqnarray}
\nabla f^{a_j}(\bar{x})^\top u \not\in -\text{ int}(K), \text{ i.e., }\mathcal{G}_e(\nabla f^{a_j}(\bar{x})^\top u) \geq 0.\label{stationary_1}
\end{eqnarray}   
\end{definition}

{
\begin{definition}(Stationary point for \eqref{vp_equation} \cite{drummond2005steepest}).
A point $\bar{x}$ is called a stationary (or critical) point of \eqref{vp_equation} if for every $a=(a_1,a_2,\ldots a_{\bar{w}})\in P_{\bar{x}}$ and $u\in\mathbb{R}^n$, there exists $j\in[\bar{w}]$ such that 
\[
\mathcal{R}(\mathcal{J} f^{a_j}(\bar{x})\cap (-\text{int}(K))=\emptyset,
\]
Therefore, $\bar{x}$ is stationary if and only if for every $a=(a_1,a_2,\ldots a_{\bar{w}})\in P_{\bar{x}}$ and for all $u\in\mathbb{R}^n$, we have
\[
\nabla f^{a_j}(\bar{x})^{\top} u\not\in -\text{int}(K)\text{ for all }j\in[\bar{w}]. 
\]  
\end{definition}}

\begin{lemma}\label{rtyrrsv}
A point $\bar{x} \in \mathbb{R}^n$ is a stationary point of \eqref{sp_equation}  if and only if  for every $a \in {P}_{\bar x}$, $\bar x$ is a stationary point of \eqref{vp_equation} for every $a \in {P}_{\bar{x}}$. 
\end{lemma}

\begin{proof}
Let $\bar{x}$ be a stationary point of  (\ref{sp_equation}). We prove that $\bar{x}$ is a stationary point of \eqref{vp_equation}. Let us assume contrarily that for some $a \in {P}_{\bar{x}}$, the point $\bar{x}$ is not a stationary point of (\ref{vp_equation}). Then, there exists $\bar u \in \mathbb R^{n}$ such that
\begin{align}\label{rsexi}
       & \nabla f^{a_{j}}(\bar{x})^{\top} \bar u \in -\text{int}(K)  \text{ for all }j \in [\bar{w}]
    \end{align}
which is contradictory to the assumption that $\bar x$ is a stationary point (see \eqref{stationary_1}) of (\ref{sp_equation}). Hence, $\bar{x}$ must be a stationary point of (\ref{vp_equation}) for every $a \in {P}_{\bar{x}}$. \\

Conversely, assume that $\bar{x}$ is a stationary point of (\ref{vp_equation}) for every $a \in {P}_{\bar{x}}$. We prove that $\bar{x}$ is a stationary point of \eqref{sp_equation}. Let us assume contrarily that $\bar x$ is not a stationary point of (\ref{sp_equation}). Then, there exists $ \bar{a}\in P_{\bar{x}}$ and $\bar u \in \mathbb R^{n}$ such that  
\allowdisplaybreaks
\begin{align*} 
 & \nabla f^{\bar{a}_{j}}(\bar x)^{\top} \bar u \in -\text{int}(K) ~\forall ~j \in [\bar{w}], 
\end{align*}
which is contradictory to $\bar{x}$ a stationary point of (\ref{vp_equation}) for every $a \in P_{\bar x}$. Hence, we conclude that $\bar{x}$ must be a stationary point of (\ref{sp_equation}) for every $a \in {P}_{\bar{x}}$.
\end{proof}

\begin{lemma}\label{weak_min_iff_stationary}
If a point $\bar x$ is a local weakly minimal point of \eqref{sp_equation}, then  $\bar x$ is a stationary point of \eqref{sp_equation}.    
\end{lemma}

\begin{proof}
Let $\bar x$ be a local weakly minimal solution of \eqref{sp_equation}. Assume contrarily that $\bar x$ is not a stationary point of (\ref{sp_equation}). Then, in view of Lemma \ref{rtyrrsv}, there exists at least one $\bar{a} = (\bar{a}_{1}, \bar{a}_{2}, \ldots, \bar{a}_{\bar{w}})^{\top} \in P_{\bar x}$ such that $\bar{x}$ is not a stationary point of \eqref{vp_equation}. 
That is, there exists $\bar u \in \mathbb R^{n}$ such that 
\begin{equation}\label{tt_K}
\nabla f^{\bar a_{j}}(\bar x)^{\top} \bar u \in -\text{int}(K) \text{ for all }j \in [ \bar{w}].
\end{equation}

\noindent
Since $f^{a_{j}}: \mathbb R^{n} \rightarrow \mathbb R^{m}$ is continuously differentiable (Assumption \ref{assumption}) for all  $j \in [\bar{w}]$, we have 
\begin{align}\label{tylor_series}
f^{\bar a_{j}}(x) = f^{\bar a_{j}}(\bar x) + \nabla f^{\bar a_{j}}(\bar x)^{\top}(x-\bar x)+ o(\lVert x-\bar x\rVert), \text{ where }\lim\limits_{x \to \bar x} \frac{o(\lVert x -\bar x\rVert)}{\lVert x-\bar{x}\rVert}=0.
\end{align}
Note that $\bar x$ is a local weakly minimal solution of \eqref{sp_equation}. Therefore, by Theorem \ref{equivalence_relation}, $\bar x$ is a local weakly minimal point of \eqref{vp_equation} for all $a \in P_{\bar x}$. So, $\bar x$ is a local weakly minimal point of \eqref{vp_equation} for $\bar a \in P_{\bar x}$. Thus, there exists a neighborhood $U$ of $\bar x$ such that 
 $$\nexists ~x \in U \text{ with } {{f}^{\bar a_j}}(x) - {{f}^{\bar a_{j}}}(\bar x) \in -\text{int}(K) \text{ for all } j \in [\bar w ].$$ 
From (\ref{tylor_series}), there exists a neighborhood $B \subseteq U$ of $\bar x$  such that for all $j \in [\bar w]$, 
      \begin{align}\label{tyu}
      & \nabla f^{\bar a_{j}}(\bar x)^{\top}(x-\bar x) \notin -\text{int} (K)\text{ for all } x \in B.
      \end{align}
 As $B$ is a neighborhood of $\bar x$, there exists $\bar t > 0$ such that $x' = \bar x + \bar t \bar{u} \in B$. The relation \eqref{tyu} with $x = x'$ yields 
 \[\nabla f^{\bar a_{j}}(\bar x)^{\top} \bar u \notin -\text{int} (K), 
 \] 
 which is contradictory to \eqref{tt_K}. Therefore, $\bar x$ is a stationary point for (\ref{sp_equation}).
\end{proof}

\subsection{\texorpdfstring{Hessian Approximation and Necessary Condition for Weakly Minimal Solutions of \eqref{sp_equation}}{Hessian Approximation and Necessary Condition for Weakly Minimal Solutions}}

To find a quadratic approximation of the functions $f^1,f^2,\ldots,f^p$, we use a positive definite approximation of their Hessian by the Broyden, Fletcher, Goldfarb, and Shanno (BFGS) \cite{broyden1969new,fletcher1970new,goldfarb1970family,shanno1970conditioning} methods.  

Corresponding to a given initial point $x_0$, for the function $f^i$, $i\in[p]$, we generate a sequence of symmetric positive definite matrices $\{B^i(x_k)\}$ starting with an initial symmetric positive definite matrix $B^i(x_0)$. For each $i\in[p]$, the quadratic approximation of each $f^i$ about $x_k$ is given by 
\begin{align*}
  &Q^i(x) = f^i(x_k)+\nabla f^i(x_k)^\top (x-x_k)+\tfrac{1}{2}(x-x_k)^\top B^i(x_k)(x-x_k). 
\end{align*}
We choose $\{B^i(x_k)\}$ that satisfies the quasi-Newton equation 
\[B^i(x_{k + 1}) ~ s_k = y^i_k,~ k = 0, 1, 2, \ldots,   \]
where $s_{k}=x_{k+1}-x_{k}$ and $y^i_{k} = \nabla f^i(x_{k+1})-\nabla f^i(x_{k})$. To ensure symmetric and positive definiteness of all the terms in the sequence $\{B^i(x_k)\}$, for a given symmetric and positive definite $B^i(x_0)$, we use the BFGS update formula at any point $x_k$ by 
\begin{eqnarray}\label{bfgs_formula}
B^i(x_{k+1})=B^i(x_{k})-\frac{B^i(x_{k})s_{k} s_{k}^\top B^i(x_{k})}{s^\top_{k}B^i(x_{k})s_{k}}+ {\frac{y^i_{k} y^{i^\top}_{k}}{s^\top_{k}y^i_{k}}}.
\end{eqnarray} 
It can be observed from \cite[Section 6.1]{wright2006numerical} that for every $k\in\mathbb{N}\cup\{0\}$, if $B^i(x_k)$ is positive definite, then $B^i(x_{k+1})$ remains positive definite.

\begin{remark}\label{remark_eigen_value_bigger_than_rho}(See \cite{wright2006numerical})
 {The BFGS method satisfies the curvature condition $s_{k}^{\top} y_k^i >0$. If each $f^i$, $i=1,2,\ldots,p$, is strongly convex, then the curvature condition is satisfied by any two points $x_k$ and $x_{k+1}$.}
\end{remark}

Next, we derive a necessary condition for weakly minimal points of \eqref{sp_equation}. We start with the following lemma.

\begin{lemma}\label{g_is_strong_convex}
For any given $x \in \mathbb{R}^n$, $a \in P_x$ and $j \in [w(x)]$, the function $g: \mathbb{R}^n \to \mathbb{R}$ given by 
\[g(u) = \mathcal{G}_{e}\left(\nabla f^{a_j}(x)^{\top}u + \tfrac{1}{2}u^\top B^{a_j}(x)u\right),
\]
where $B^{a_j}(x)\in\mathbb{R}^{m\times m}$ is a symmetric positive definite matrix, is strongly convex on $\mathbb{R}^n$. 
\end{lemma}

\begin{proof}
For any $u\in\mathbb{R}^n$, we define a function $h_j: \mathbb{R}^n \rightarrow \mathbb{R}^m$ by
\[
h_j(u) = \nabla f^{a_j}(x)^{\top}u + \tfrac{1}{2}u^\top B^{a_j}(x)u. 
\]
Since $B^{a_j}$ is a symmetric positive definite matrix, there exists positive constant $\rho_j$ such that 
\begin{eqnarray}\label{strong_convexity_eq1}
u^\top \nabla^2 h_j(x) u \succeq \rho_j \|u\|^2 e. 
\end{eqnarray}
So, by Corollary 2.2 in \cite{drummond2014quadratically}, the function $h_j$ is strongly convex. Hence, there exists $\mu > 0$ such that for any $u_1, u_2 \in \mathbb{R}^n$ and $\lambda \in [0, 1]$,  
\begin{equation}\label{hj_ineq}
h_j(\lambda u_1 + (1 - \lambda)u_2) \preceq \lambda h_j(u_1) + (1 - \lambda)h_j(u_2) - \tfrac{\mu}{2}\lambda (1 - \lambda) \|u_1 - u_2\|^2e. 
\end{equation}
Therefore, in view of Proposition \ref{gerstewitz}, for any $u_1, u_2 \in \mathbb{R}^n$ and $\lambda \in [0, 1]$, we have 
\begin{align*}
~&~ g(\lambda u_1 + (1 - \lambda)u_2) \\ 
= ~&~ \mathcal{G}_e(h_j(\lambda u_1 + (1 - \lambda)u_2)) \\ 
\overset{\eqref{hj_ineq}}{\le} ~&~ \lambda \mathcal{G}_e(h_j(u_1)) + (1 - \lambda) \mathcal{G}_e(h_j(u_2)) - \tfrac{\mu}{2}\lambda (1 - \lambda) \|u_1 - u_2\|^2  \\
= ~&~ \lambda g(u_1) + (1 - \lambda) g(u_2) - \tfrac{\mu}{2}\lambda (1 - \lambda) \|u_1 - u_2\|^2. 
\end{align*} 
Hence, $g$ is strongly convex on $\mathbb{R}^n$.  
\end{proof}

\begin{remark}
In view of Lemma \ref{weak_min_iff_stationary}, we note that every weakly minimal solution of \eqref{sp_equation} is a stationary point of \eqref{sp_equation}. Also, from Definition \ref{stationary_dfn}, we see that   
\begin{align*}
&~~ \text{A point } \bar{x} \text{ is a stationary point of (\ref{sp_equation})} \\ 
\Longleftrightarrow &\text{ for every } a \in P_{\bar x} \text{ and } u \in \mathbb{R}^n,  ~\exists~ a_j \text{ with } \mathcal{G}_e(\nabla f^{a_j}(\bar x)^\top u) \ge 0. 
\end{align*}
Thus, by Proposition \ref{ge_6}(v) \& (vi) and \eqref{strong_convexity_eq1}, for a weakly minimal point $\bar x$, for any $a \in P_{\bar x}$ and $u \in \mathbb{R}^n$, there exists $a_j$ with 
\begin{equation}\label{psi_e_ge_0}
\mathcal{G}_e\left(\nabla f^{a_{j}}(\bar x)^\top u + u^\top B^{a_{j}}(\bar x) u \right)   
\ge \mathcal{G}_e\left(\nabla f^{a_{j}}(\bar x)^\top u\right) + \tfrac{1}{2}\rho \|u\|^2 \ge 0, 
\end{equation}
where $\rho = \min\{\rho_1, \rho_2, \ldots, \rho_p\}$.  
\end{remark}

Next, we discuss a necessary condition for weakly minimal solutions of (\ref{sp_equation}).  For this, for any $x \in \mathbb{R}^n$, we define a function $\xi_x:P_x \times \mathbb{R}^n \to\mathbb{R}$ by 
\begin{align}
\xi_x(a,u) = \underset{j\in [w(x)]}{\max}\left\{\mathcal{G}_{e}(\nabla f^{a_j}(x)^{\top}u + \tfrac{1}{2}u^\top~ B^{a_j}(x)u)\right\}, a\in P_x,~u\in\mathbb{R}^n. \label{g_function}   
\end{align}
Then, by \eqref{psi_e_ge_0}, at a stationary point $\bar x$, we have 
\begin{align}\label{xi_bar_x_ge_0} 
~&~ \xi_{\bar{x}}(a, u) \ge 0 ~~\forall~ a \in P_{\bar x} \text{ and } u \in \mathbb{R}^n \notag \\ 
\Longrightarrow ~&~ \min_{u \in \mathbb{R}^n} \xi_{\bar{x}}(a, u) \ge 0 ~~\forall~ a \in P_{\bar x} \notag \\ 
\Longrightarrow ~&~ \forall a \in P_{\bar x}: 0 \le \min_{u \in \mathbb{R}^n} \xi_{\bar{x}}(a, u) \le \xi_{\bar{x}}(a, 0) = 0 \notag \\ 
\Longrightarrow ~&~ \forall a \in P_{\bar x}: \min_{u \in \mathbb{R}^n} \xi_{\bar{x}}(a, u) = 0. 
\end{align}

\noindent
Moreover, as for any $x \in \mathbb{R}^n$, $P_x$ is finite, we note from Lemma \ref{g_is_strong_convex} that for any $a\in P_x$, the function $\xi_x(a,\cdot)$ is strongly convex in $\mathbb{R}^n$. Hence, the function $\xi_{\bar x}(a,\cdot)$ has a unique minimum over $\mathbb{R}^n$. If for $a\in P_{\bar x}$, $\bar{u}_{a, \bar{x}}\in\mathbb{R}^n$ be such that $\xi_{\bar x}(a,\bar{u}_{a, \bar{x}})=\underset{u\in\mathbb{R}^n}{\min}~\xi_{\bar x}(a,u)$, then from \eqref{xi_bar_x_ge_0}, we have 
\begin{eqnarray}
\xi_{\bar x}(a,\bar{u}_{a, \bar{x}}) = 0 \text{ if and only if }\bar{u}_{a, \bar{x}} = 0. \label{function_4}
\end{eqnarray}
As for any $x \in \mathbb{R}^n$, the partition set $P_{x}$ is finite, $\xi_{x}$ attains its minimum over the set $P_{x}\times\mathbb{R}^n$. Let us define a function $\Phi:\mathbb{R}^n\to\mathbb{R}$ by	
\begin{eqnarray}
&&\Phi(x)=\underset{(a,u)\in P_{x} \times \mathbb{R}^n}{\min}~\xi_x(a,u)\label{phi_function}.
\end{eqnarray}
Then, in view of (\ref{function_4}) and \eqref{xi_bar_x_ge_0}, if for $(a,\bar{u}) \in P_{\bar x}\times\mathbb{R}^n$ we have $\Phi(\bar x)=\xi_{\bar x}(\bar{a},\bar{u})$, then   
\begin{eqnarray}\label{function_2}
\Phi(\bar x)= 0 \text{ and } \bar{u}=0. 
\end{eqnarray}
Accumulating all, we obtain the following result.

\begin{proposition}\emph{(Necessary condition for weakly minimal points)}.\label{prop_minimal}
Let $\bar{x}$ be a weakly minimal point of (\ref{sp_equation}) and  $\bar{a}\in P_{\bar{x}}$ and $\bar{u}\in\mathbb{R}^n$ be such that $\Phi(\bar{x})=\xi_{\bar{x}}(\bar{a},\bar{u})$, where $\xi_{\bar{x}}$ and $\Phi$ are as defined in (\ref{g_function}) and (\ref{phi_function}), respectively. Then, $\bar{u}=0$.  
\end{proposition}


\section{Quasi-Newton Method for \eqref{sp_equation}}\label{section4}

The whole section is described in the following two subsections.
\begin{enumerate}[(i)]
    \item At first, we discuss a few properties of $\Phi$. 
    \item Then, we propose the quasi-Newton method (Algorithm \ref{algo1}) and discuss its well-definedness. After that, we characterize the boundedness of the norm of the used descent direction for \eqref{sp_equation}. Further, the existence of step length that satisfies the Armijo condition is derived.   
\end{enumerate}
\subsection{Properties of $\Phi$}
In this subsection, we discuss a few properties of $\Phi$, which play an important role in the convergence analysis of the proposed quasi-Newton method for \eqref{sp_equation}. 

\begin{proposition}\label{continuity}
The function $\Phi$ as given in (\ref{phi_function}) is continuous at any  $\bar{x}\in \mathbb{R}^n$.    
\end{proposition}

\begin{proof}
Let $\{x_k\}$ be a sequence in $\mathbb{R}^n$ that converges to $\bar{x}\in\mathbb{R}^n$. We show that \[\lim_{k\to\infty}\Phi(x_k) = \Phi(\bar{x}).\] 
Since the set $P_{\bar x}$ is finite and $\xi_{\bar x}$ attains its minimum over the set $P_{\bar x}\times\mathbb{R}^n$, there exists $(\bar a, \bar u) \in P_{\bar x} \times \mathbb{R}^n$ such that $\Phi(\bar x) = \xi_{\bar x} (\bar a, \bar u)$. \\ \\ 
Let $(a^k, u_k)$ be an element in $P_{x_k} \times \mathbb{R}^n$ such that $\Phi(x_k) = \xi_{x_k}(a^k,u_k)$. Such an element $(a^k, u_k)$ exists since the set $P_{x_k}$ is finite and $\xi_{x_k}$ attains its minimum over the set $P_{x_k}\times\mathbb{R}^n$. 
Since $\mathcal{G}_e$ is Lipschitz continuous on $\mathbb{R}^n$ (Proposition \ref{gerstewitz} (iii)) and $f^{a_j}$ is twice continuously differentiable for each $j\in[w(x_k)]$, the function $\xi_{x}$ is continuous on $P_{x_k} \times \mathbb{R}^n$. Thus, we get
\begin{eqnarray}\label{continuity_1} 
\underset{k\to \infty}{\limsup}~\Phi(x_k) = \underset{k\to \infty}{\limsup}~\xi_{x_k}(a^k,u_k) \le \underset{k\to \infty}{\limsup}~\xi_{x_k}(\bar a, \bar u) =\xi_{\bar{x}}(\bar{a},\bar{u})
=\Phi(\bar{x}).
\allowdisplaybreaks
\end{eqnarray}
Let $L>0$ be a Lipschitz constant of $\mathcal{G}_e$. Then, from the definition (\ref{phi_function}) of $\Phi$ at $\bar{x}$, we observe that 
 \allowdisplaybreaks
\begin{align}\label{cont_dg}
&~ \Phi(\bar{x}) \nonumber \\ 
=&~  \underset{(a,u)\in P_{\bar{x}}\times\mathbb{R}^n}{\min}\xi_{\bar{x}}(a,u)\nonumber\\
\leq &~\xi_{\bar{x}}(a^k,u_k)\nonumber\\
=&~\underset{k\to \infty}{\liminf}~\xi_{\bar{x}}(a^k,u_k)\text{ since }\xi_{\bar{x}}\text{ is continuous}\nonumber \\
=&~\underset{k\to \infty}{\liminf}~\left\{\underset{j\in[w(\bar{x})]}{\max}\left(\mathcal{G}_e(\nabla f^{a^k_j}(\bar{x})^\top u_k+\tfrac{1}{2}u_k^\top B^{a^k_j}(\bar{x})u_k)\right)\right\}\nonumber\\
=&~\underset{k\to \infty}{\liminf}~\left\{\underset{j\in[w(\bar{x})]}{\max}\left(\mathcal{G}_e(\nabla f^{a^k_j}(\bar{x})^\top u_k+\tfrac{1}{2}u_k^\top B^{a^k_j}(\bar{x})u_k+\nabla f^{a^k_j}(x_k)^\top u_k\right.\right.\nonumber\\
&~\quad \quad \quad \quad \left.\left. +\tfrac{1}{2}u_k^\top B^{a^k_j}(x_k)u_k-\nabla f^{a^k_j}(x_k)^\top u_k-\tfrac{1}{2}u_k^\top B^{a^k_j}(x_k)u_k)
\right) \right\} \nonumber \\ 
\overset{\ref{gerstewitz}\text{(i)}}{\le}&~\underset{k\to \infty}{\liminf}~\left\{\underset{j\in[w(\bar{x})]}{\max}\left\{\mathcal{G}_e\left(\nabla f^{a^k_j}(x_k)^\top u_k + \tfrac{1}{2} u_k^\top B^{a^k_j}(x_k)u_k\right) \right. \right.  \nonumber\\
&~\left. \left. + \mathcal{G}_e\left(\nabla f^{a^k_j}(\bar{x})^\top u_k +\tfrac{1}{2}u_k^\top B^{a^k_j}(\bar{x})u_k-\nabla f^{a^k_j}(x_k)^\top u_k-\tfrac{1}{2}u_k^\top B^{a^k_j}(x_k)u_k
\right)\right\}\right\}\nonumber\\
=&~\underset{k\to \infty}{\liminf}~\left\{\xi_{x_k}(a^k,u_k) +\underset{j\in[w(\bar{x})]}{\max}\left\{\mathcal{G}_e\left(\nabla f^{a^k_j}(\bar{x})^\top u_k \right.\right.\right.\nonumber\\
&~ \left. \left. \left. \qquad~~ + \tfrac{1}{2}u_k^\top B^{a^k_j}(\bar{x})u_k-\nabla f^{a^k_j}(x_k)^\top u_k - \tfrac{1}{2}u_k^\top B^{a^k_j}(x_k)u_k\right)\right\}\right\}\nonumber\\
\overset{\ref{gerstewitz}\text{(iii)}}{\le} & \underset{k\to \infty}{\liminf}~\left\{\xi_{x_k}(a^k,u_k) + L\underset{j\in[w(\bar{x})]}{\max}\left\{\lVert\nabla f^{a^k_j}(\bar{x})^\top u_k+\tfrac{1}{2}u_k^\top B^{a^k_j}(\bar{x})u_k \right. \right.\nonumber\\
&~\quad \quad \quad \left.\left.-\nabla f^{a^k_j}(x_k)^\top u_k -\tfrac{1}{2}u_k^\top B^{a^k_j}(x_k)u_k\rVert\right\}\right\} \nonumber\\
\leq&~\underset{k\to \infty}{\liminf}\left\{~\xi_{x_k}(a^k,u_k) + L~\underset{j\in[w(\bar{x})]}{\max}\{\lVert\nabla f^{a^k_j}(\bar{x})-\nabla f^{a^k_j}(x_k)\rVert\lVert u_k\rVert\}\right. \nonumber \\
&~\left.+\tfrac{L}{2}\underset{j\in[w(\bar{x})]}{\max}\left\{\lVert u_k^\top\left( B^{a^k_j}(\bar{x})- B^{a^k_j}(x_k)\right)u_k\rVert\right\}\right\}. 
\end{align}
Note that for $j\in[w]$, each $f^{a^k_j}$ is a twice continuously differentiable and the sequence $\{x_k\}$ converges to $\bar{x}$. Also, note that there is no loss of generality if $\{u_k\}$ is assumed to be in $\{u \in \mathbb{R}^n: \|u\| \le 1\}$. Thus, we obtain from \eqref{cont_dg} that 
\begin{eqnarray}\label{continuity_3}
\Phi(\bar{x})\leq \underset{k\to \infty}{\liminf}~\xi_{x_k}(a^k,u_k)=\underset{k\to \infty}{\liminf}~\Phi(x_k).
\end{eqnarray}
Finally, in view of (\ref{continuity_1}) and (\ref{continuity_3}), we conclude that
\[
\underset{k\to\infty}{\lim}\Phi(x_k)=\Phi(\bar{x}).
\]
Thus, the function $\Phi$ is continuous at $\bar{x}$.
\end{proof}

\begin{proposition}\label{bounded}
Let $U$ be a nonempty subset of $\mathbb{R}^n$. Suppose there exists $p,q\in\mathbb{R}_{++}$ such that for any $x\in U$ and $a\in P_x$, $B^{a_j}(x)\leq qI$ for all $j\in[w(x)]$. Then, for any $a\in P_{x}$, there exist
 $\lambda_j\geq0$, $j\in [w(x)]$ with $\sum_{j=1}^{[w(x)]}\lambda_j=1$ such that
\[
\lvert\Phi(x)\rvert\leq\tfrac{3L}{2q}\left\lVert\sum_{j=1}^{[w(x)]}\lambda_j\nabla f^{a_j}(x)\right\rVert^2 ,
\]
 where $L$ is the Lipschitz constant of $\mathcal{G}_e$.
\end{proposition}

\begin{proof}
Let $x\in U$ and $P_x$ be the partition set of (\ref{sp_equation}) at $x$. Note that for any $b_1,b_2,\ldots,b_{w(x)}\in\mathbb{R}$, the identity $\max\{b_1,b_2,\ldots,b_{w(x)}\}=\underset{\lambda\in \Delta_{w(x)}}{\max}\sum_{i=1}^{w(x)} \lambda_i b_i $ holds, where $\Delta_{w(x)} = \{(\lambda_1, \lambda_2, \ldots, \lambda_{w(x)}) \in \mathbb{R}^{w(x)}_+: \sum_{i = 1}^{w(x)} \lambda_i = 1\}$. Thus, in view of the definition (\ref{phi_function}) of $\Phi$, we have for any $a\in P_{x}$ that
\allowdisplaybreaks 
\begin{eqnarray} 
|\Phi(x)|&=& \left|\underset{(a,u)\in P_x\times\mathbb{R}^n}{\min}\xi_{x}(a,u) \right|\nonumber \\
&=&\left| \underset{(a,u)\in P_x\times\mathbb{R}^n}{\min}~\left\{\underset{j\in[w(x)]}{\max}\mathcal{G}_e\left(\nabla f^{a_j}(x)^\top u+\tfrac{1}{2}u^\top B^{a_j}(x)u\right)\right\}\right|\nonumber\\ 
&=&\underset{(a,u)\in P_x\times\mathbb{R}^n}{\min}~\left|\underset{\lambda\in \Delta_{w(x)}}{\max}\sum_{j=1}^{[w(x)]}\lambda_j\mathcal{G}_e\left(\nabla f^{a_j}(x)^\top u+\tfrac{1}{2}u^\top B^{a_j}(x)u\right)\right|\nonumber\\ 
&\leq&\underset{(a,u)\in P_x\times\mathbb{R}^n}{\min}~\sum_{j=1}^{w(x)} \left|\lambda_j\mathcal{G}_e\left(\nabla f^{a_j}(x)^\top u+\tfrac{1}{2}u^\top B^{a_j}(x)u\right)\right| \notag \\ 
&& \text{ for any } \lambda \in \Delta_{w(x)} \nonumber \\ 
&\overset{\ref{gerstewitz}\text{(iii)}}{\le} & \underset{(a,u)\in P_x\times\mathbb{R}^n}{\min}~\sum_{j=1}^{w(x)} \lambda_j L \left\lVert\nabla f^{a_j}(x)^\top u+\tfrac{1}{2}u^\top B^{a_j}(x)u\right\rVert \notag \\ 
&\leq& L \underset{(a,u)\in P_x\times\mathbb{R}^n}{\min}~\left\{\sum_{j=1}^{w(x)} \lambda_j\left\lVert\nabla f^{a_j}(x)^\top u\right\rVert + \sum_{j=1}^{w(x)} \lambda_j \left\lVert\tfrac{1}{2}u^\top B^{a_j}(x)u\right\rVert \right\} \nonumber \\ 
&\leq& L\underset{(a,u)\in P_x\times\mathbb{R}^n}{\min}~\left\{\sum_{j=1}^{w(x)} \lambda_j \lVert\nabla f^{a_j}(x)^\top u\rVert + \tfrac{\gamma}{2}\lVert u\rVert^2\right\}\text{ as }B^{a_j}(x) \leq qI.  \label{dg_aux1}
\end{eqnarray} 
Note that the function $u \mapsto \sum_{j=1}^{w(x)}\lambda_j\lVert\nabla f^{a_j}(x)^\top u\rVert+\tfrac{q}{2}\lVert u\rVert^2$ is a strongly convex function on $\mathbb{R}^n$. Therefore, the first-order optimality condition implies that its minimum is obtained at $u=-\tfrac{1}{q}\sum_{j=1}^{w(x)}\lambda_j\nabla f^{a_j}(x)$. Thus, \eqref{dg_aux1} gives that for all $\lambda_j\geq0$ and $j\in w(x)$ with $\sum_{j=1}^{w(x)}\lambda_j=1$, we have 
\[
\lvert\Phi(x)\rvert\leq\tfrac{3L}{2q}\left\lVert\sum_{j=1}^{[w(x)]}\lambda_j\nabla f^{a_j}(x)\right\rVert^2. 
\]
\end{proof}


\subsection{Quasi-Newton Method for \eqref{sp_equation}}
\begin{algorithm}[!htb]
\caption{Quasi-Newton Method for the Set Optimization Problem (\ref{sp_equation})  \label{algo1}}
\begin{enumerate}[\textit{Step} 1]
\item \label{step0}  \textbf{Inputs}\\ Provide the objective function $F$ with $f^1,f^2,\ldots,f^p$ being twice continuously differentiable.\\
\item \label{step1}\textbf{Initialization}\\
Choose an initial point $x_0\in\mathbb{R}^n$, a trial step length $\beta\in(0,1)$, and a positive $\nu\in(0,1)$.\\
Provide an initial symmetric positive definite matrix $B^i(x_0) \in\mathbb{R}^{n\times n}$, for each $i\in[p]$.\\
Set the iteration number $k=0$.\\
Provide a value of the precision level $\varepsilon>0$ for termination. \\
\item \label{step2}
 \textbf{Calculate the minimal set and the partition set at the $k$-th iteration}\\
 Compute $M_k=$ Min $(F(x_k),K)=\{r_1,r_2,\ldots,r_{w_k}\}$ and $w_k=\lvert$ Min $F(x_k),K \rvert$.\\
 Find 
 $P_k=P_{x_k}=I_{r_1}\times I_{r_2}\times\cdots\times  I_{r_{w_k}}$,~$p_k=\lvert P_{x_{k}}\rvert$, and $P_{x_k}=\{a_1,a_2,\ldots,a_{p_k}\}$,\\
 and for each $i\in [p_k],~a_i=(a^{1}_i,a^{2}_i,\ldots,a^{w_{k}}_i)\in P_{x_{k}},~a^j_{i}\in I_{r^{x_k}_j},~j\in[w_k].$\\
\item \label{step3}
     \textbf{Computation of a descent direction}\\
     Find $
     (a^k,u_k)\in \underset{(a,u)\in P_k\times\mathbb{R}^n}{\text{ argmin}} \xi_{x_k}(a,u)$, where $\xi_{x_k}:P_{x_k}\times\mathbb{R}^n\to \mathbb{R}$ is given by 
     $\xi_{x_k}(a,u)=\underset{j\in [w_k]}{\max}\{\mathcal{G}_e(\nabla f^{a_{j}}(x_k)^\top u_k+\tfrac{1}{2}u_k^\top B^{a_j}(x_k) u_k)\}$.\\
     
\item \label{step4} \textbf{Stopping criterion}\\
  If $\lVert u_k\rVert<\varepsilon$, stop. Otherwise, go to Step \ref{step5}.\\
\item \label{step5} \textbf{Compute step length}\\
Evaluate the smallest value of $q$ such that $\nu^q$ estimates the step length $t_k$ by 
\[
t_k= \underset{q\in\mathbb{N}\cup\{0\}}{\max}\left\{\nu^q~:f^{a_{j}^{k}}(x_k+\nu^q u_k)\preceq f^{a_{j}^{k}}(x_k)+\beta \nu^q \nabla f^{a_{j}^{k}}(x_k)^\top u_k ~\forall~j\in[w_k]\right\}.
\]
\item \label{step6} \textbf{Update the iterate and approximation} \\
Update $x_{k+1}\leftarrow x_k+t_k u_k$ and $k\leftarrow k+1$. \\ 
Compute $s_{k}=x_{k+1}-x_{k}$ and $y^i_{k} = \nabla f^i(x_{k+1})-\nabla f^i(x_{k})$, $i \in [p]$. \\ 
Evaluate the next symmetric positive definite matrix using the formula \eqref{bfgs_formula}: 
\[ 
B^i(x_{k+1})=B^i(x_{k})-\frac{B^i(x_{k})s_{k} s_{k}^\top B^i(x_{k})}{s^\top_{k}B^i(x_{k})s_{k}}+\frac{y^i_{k} y^{i^\top}_{k}}{s^\top_{k}y^i_{k}}.
\]
Go to Step \ref{step2}. 
\end{enumerate}
\end{algorithm}

In this subsection, we propose a quasi-Newton method (Algorithm \ref{algo1}) for set optimization problems (\ref{sp_equation}) with an $F$ as given in Assumption \ref{assumption}. We start the algorithm by selecting an arbitrary initial point. If this point does not satisfy the necessary condition for a weakly minimal point as stated in Proposition \ref{prop_minimal}, then we proceed to update this point as discussed in Algorithm \ref{algo1}. At each iteration, we select an element $a$ from the partition set, and then a descent direction for (\ref{vp_equation}) is evaluated by following the ideas of \cite{povalej2014quasi,kumar2023quasi}. Once a descent direction is found, we employ a backtracking procedure similar to the classical Armijo-type method to find an appropriate step size and then update the iterate. We keep updating the iterate until the necessary condition in Proposition \ref{prop_minimal} for a weakly minimal point is met. The entire method is given in Algorithm \ref{algo1}.

\begin{remark}
It is to be noted that for $p=1$ in Algorithm \ref{algo1}, that is, for $F(x)=\{f^1(x)\}$, the Step \ref{step3} of Algorithm \ref{algo1} reduces to finding $u_k$ such that
\[
u_k= \underset{u\in\mathbb{R}^n}{\text{ argmin}} ~~\mathcal{G}_e\left(\nabla f^{1}(x_k)^\top u_k+\tfrac{1}{2}u_k^\top B^{1}(x_k) u_k\right).
\] 
In this case, the proposed Algorithm \ref{algo1} reduces to the method given in \cite{povalej2014quasi} and \cite{kumar2023quasi}.   
\end{remark}

Next, we show that Algorithm \ref{algo1} is well-defined. The well-definedness of Algorithm \ref{algo1} is based on the following two points:
\begin{enumerate}
    \item[(i)] Existence of $(a^k,u_k)$ in Step \ref{step3}, which is assured by the discussion in the paragraph after Definition \ref{stationary_dfn}. 
    \item[(ii)] Existence of step length $t_k$ in Step \ref{step5}, which is assured by the result in Proposition \ref{armijo}. 
\end{enumerate} 
Therefore, Algorithm \ref{algo1} is well-defined.\\

Next, we characterize the stationary points of (\ref{sp_equation}) in terms of the functions $\xi_x$ and $\Phi$ as defined in (\ref{g_function}) and (\ref{phi_function}), respectively.


\begin{theorem}\label{critical}
Let us consider the functions $\xi_x$ and $\Phi$ given in (\ref{g_function}) and (\ref{phi_function}), respectively. Let $(\bar{a},\bar{u})\in P_x\times\mathbb{R}^n$ be such that $\Phi(\bar{x})=\xi_{\bar{x}}(\bar{a},\bar{u})$. Then, the following conditions are equivalent:
\begin{enumerate}
    \item[(i)] \label{c_0}The point $\bar{x}$ is a nonstationary point of (\ref{sp_equation}).
    \item[(ii)] \label{c_1} $\Phi(\bar{x})<0$. 
    \item[(iii)] \label{c_2} $\bar{u}\not=0.$
\end{enumerate}
\end{theorem}

\begin{proof}
 (i)$\implies$(ii). Let us assume that the point $\bar{x}$ is a nonstationary point of (\ref{sp_equation}). Then, in view of (\ref{stationary_1}), there exists an $\widetilde{a}=(\widetilde{a}_1,\widetilde{a}_2,\ldots,\widetilde{a}_{\bar{w}})\in P_{\bar{x}}$ and $\widetilde{u} \in\mathbb{R}^n$ for which   
\[
\mathcal{G}_e(\nabla f^{\widetilde{a}_j}(\bar{x})^\top \widetilde{u}) < 0 \text{ for all } j \in [\bar w].
\]
Thus, in view of the above relation, we conclude that
\begin{eqnarray}\label{critical_0}
\Phi(\bar{x})&=&\underset{(a,u)\in P_{\bar{x}}\times\mathbb{R}^n}{\min}~\xi_{\bar{x}}(a,u)\nonumber\\
&\leq& \xi_{\bar{x}}(\widetilde a, t\widetilde{u})\text{ for any }t>0  \nonumber \\
&=&\underset{j\in[w(\bar{x})]}{\max} \Psi_{e}\left(\nabla f^{\widetilde{a}_j}(\bar{x})^{\top}t\widetilde{u}+\tfrac{1}{2}t\widetilde{u}^\top B^{\widetilde{a}_j}(\bar{x})t\widetilde{u}\right)  \nonumber\\
&=&t\underset{j\in[w(\bar{x})]}{\max} \Psi_{e}\left(\nabla f^{\widetilde{a}_j}(\bar{x})^{\top}\widetilde{u} + \tfrac{t}{2}\widetilde{u}^\top B^{\widetilde{a}_j}(\bar{x})\widetilde{u}\right)  \text{ by Proposition \ref{gerstewitz}(ii)  }\nonumber\\
&\leq&t\underset{j\in[w(\bar{x})]}{\max}\left\{\Psi_{e}\left(\nabla f^{\widetilde{a}_j}(\bar{x})^{\top}\widetilde{u}\right)+\tfrac{t}{2}\mathcal{G}_e\left(\widetilde{u}^\top B^{\widetilde{a}_j}(\bar{x})\widetilde{u}\right)\right\} \text{ by Proposition  \ref{gerstewitz} (i)\&(ii)}\nonumber\\
&\leq&t\left\{\underset{j\in[w(\bar{x})]}{\max}\{\Psi_{e}(\nabla f^{\widetilde{a}_j}(\bar{x})^{\top}\widetilde{u})\}+\tfrac{t}{2}\underset{j\in[w(\bar{x})]}{\max}\{\mathcal{G}_e(\widetilde{u}^\top B^{\widetilde{a}_j}(\bar{x})\widetilde{u})\}\right\}. 
\end{eqnarray} 
Choosing any $t$ such that $0<t<\left(\frac{-2}{\underset{j\in [w(x)]}{\max}\{\mathcal{G}_e(\widetilde{u}^\top B^{\widetilde{a}_j}(\bar{x})\widetilde{u})\}}\right)\left(\underset{j\in [w(x)]}{\max}\{\mathcal{G}_e(\nabla f^{\widetilde{a}_j}(\bar{x})^{\top}\widetilde{u})\}\right)$, we obtain from (\ref{critical_0}) that 
\[
\Phi(\bar{x})<t\left\{\underset{j\in[w(\bar{x})]}{\max}\{\Psi_{e}(\nabla f^{\widetilde{a}_j}(\bar{x})^{\top}\widetilde{u})\}-\underset{j\in[w(\bar{x})]}{\max}\{\Psi_{e}(\nabla f^{\widetilde{a}_j}(\bar{x})^{\top}\widetilde{u})\}\right\}=0.
\]
 \\
    (ii)$\implies$(iii).
   It trivially follows from  (\ref{function_2}). 
   \\
    (iii)$\implies$(i).
   Let us assume contrarily that $\bar{x}$ is a stationary point of (\ref{sp_equation}) and $\bar{u}\not=0$. Then, in view of (\ref{stationary_1}), for $\bar{a}\in P_{\bar{x}}$, there exists  $\widetilde{j}\in[\bar{w}]$  such that
   \begin{eqnarray}\label{critical_1}
    \mathcal{G}_e(\nabla f^{{\bar{a}}_{\widetilde{j}}}(\bar{x})^\top \bar{u})\geq0. 
   \end{eqnarray}
   Note from Assumption \ref{assumption} that for any $a\in P_x$ and $x\in\mathbb{R}^n$, we have $\bar{u}^\top B^{{\bar{a}}_{\widetilde{j}}}(\bar{x})\bar{u}>0$. Therefore, from (\ref{critical_1}) with the help of  Proposition \ref{gerstewitz} (iv), we get 
   \allowdisplaybreaks
     \begin{eqnarray*}
     && \mathcal{G}_e\left(\nabla f^{{\bar{a}}_{\widetilde{j}}}(\bar{x})^\top \bar{u} +\tfrac{1}{2}\bar{u} ^\top B^{{\bar{a}}_{\widetilde{j}}}(\bar{x})\bar{u} \right)\geq0\\
     &\text{or,}& \underset{j\in[w(\bar{x})]}{\max}\left\{\mathcal{G}_e(\nabla f^{{\bar{a}}_{j}}(\bar{x})^\top \bar{u} +\tfrac{1}{2}\bar{u} ^\top B^{{\bar{a}}_{j}}(\bar{x})\bar{u})\right\}\geq0\\
     &\text{or,}&\xi_{\bar{x}}(\bar{a},\bar{u})\geq0\\
     &\text{or,}&\Phi(\bar{x})=0 \text{ from (\ref{xi_bar_x_ge_0})}\\
     &\text{or,}&\bar{u}=0 \text{ from (\ref{function_2})}, 
    \end{eqnarray*}  
which is a contradiction to the considered assumption. Thus, $\bar{x}$ is a nonstationary point of (\ref{sp_equation}). 
\end{proof}
\begin{remark}
    In view of (\ref{function_2}) and statements (i)--(iii) of Theorem \ref{critical}, we obtain that $\bar{x}$ is a stationary point of (\ref{sp_equation}) if and only if $\Phi(\bar{x})=0$ or $\bar{u}=0.$
\end{remark}

Next, we characterize an upper bound for the norm of quasi-Newton's direction $u_k$ generated by Algorithm \ref{algo1} for (\ref{sp_equation}). After that, we provide  convergence analysis of Algorithm \ref{algo1}.


\begin{theorem} \label{uk_bounded}
Let $\{x_k\}$ be the sequence of nonstationary points, $\{u_k\}$ be a sequence of descent directions generated by Algorithm \ref{algo1}, and  $\{x_k\}$ be convergent. Then, the sequence $\{u_k\}$ is bounded.    
\end{theorem}

\begin{proof}
 Let $P_{x_k}$ be the partition set at $\{x_k\}$  {and $\{x_k\}$ be a sequence of nonstationary points that converges to $\bar{x}$ (say). Then, in view of Theorem \ref{critical}, there exists $a^k\in P_{x_k}$ such that $\Phi(x_k)<0$, i.e.,}
 \begin{eqnarray}\label{bounded_eq1}
  && \underset{j\in [w(x_k)]}{\max}\left\{\mathcal{G}_{e}(\nabla f^{a^k_j}(x_k)^{\top}u_k+\tfrac{1}{2}u_k^\top B^{a^k_j}(x_k)u_k)\right\}< 0\nonumber\\
  &\implies&\mathcal{G}_{e}(\nabla f^{a^k_j}(x_k)^{\top}u_k+\tfrac{1}{2}u_k^\top B^{a^k_j}(x_k)u_k)<0 \text{ for all }~j\in[w(x_k)] \nonumber\\ 
  &\overset{\text{\ref{gerstewitz}(vi)}}{\implies}& \mathcal{G}_{e}(\nabla f^{a^k_j}(x_k)^{\top}u_k)+\tfrac{1}{2}u_k^\top B^{a^k_j}(x_k)u_k<0 \text{ for all }~j\in[w(x_k)].
  \end{eqnarray}
Then, note that $W=\{x_k\}\cup\{\bar{x}\}$ is compact. Moreover, $B^1,B^2,\ldots,B^p$ are continuous functions, for each $i\in[p]$. Then, we have
\begin{eqnarray}\label{bfgs_eigenvalues_bound}
    \sigma_i=\underset{\lVert u\rVert=1,~x\in W}{\min}~u^\top B^i(x)u>0.
\end{eqnarray}
On taking $\rho=\min\{\sigma_1,\sigma_2,\ldots,\sigma_p\}$, we have $u^\top B^i(x)u\geq \rho \lVert u\rVert^2$. Thus, in view of \eqref{bounded_eq1} and \eqref{bfgs_eigenvalues_bound}, we get
 \begin{align}\label{bounded_u_k1}
&\mathcal{G}_{e}(\nabla f^{a^k_j}(x_k)^{\top}u_k)+\tfrac{1}{2} \rho\lVert u_k\rVert^2<0 \text{ for all }~j\in[w(x_k)]\nonumber\\
\implies& \tfrac{1}{2} \rho\lVert u_k\rVert^2\leq - \mathcal{G}_e (\nabla f^{a^k_j}(x_k)^{\top}u_k)\text{ for all }~j\in[w(x_k)]\nonumber \\
 \implies&\tfrac{1}{2} \rho\lVert u_k\rVert^2\leq \underset{j\in [w(x_k)]}{\max}\{\lvert\mathcal{G}_e(\nabla f^{a^k_j}(x_k)^{\top}u_k)\rvert\}\nonumber\\
\overset{\text{\ref{gerstewitz}(iii)}}{\implies}&\tfrac{1}{2} \rho\lVert u_k\rVert^2 \leq L \underset{j\in [w(x_k)]}{\max}\lVert\nabla f^{a^k_j}(x_k)^{\top}u_k\rVert,~L\text{ is a Lipschitz constant of } \mathcal{G}_e\nonumber\\
\implies& \tfrac{1}{2} \rho\lVert u_k\rVert^2 \leq L \lVert u_k\rVert \underset{j\in [w(x_k)]}{\max}\lVert\nabla f^{a^k_j}(x_k)\rVert\nonumber\\
\implies& \tfrac{1}{2} \rho\lVert u_k\rVert^2 \leq L \lVert u_k\rVert ~{\max}\{\lVert\nabla f^{1}(x_k)\rVert,\lVert\nabla f^{2}(x_k)\rVert,\ldots,\lVert\nabla f^{p}(x_k)\rVert\}. 
 \end{align}
 Note that each $f^i,~i\in[p]$, is twice continuously differentiable and the sequence $\{x_k\}$ is convergent. Therefore, ${\max}\{\lVert\nabla f^{1}(x_k)\rVert,\lVert\nabla f^{2}(x_k)\rVert,\ldots,\lVert\nabla f^{p}(x_k)\rVert\}$ is a convergent sequence, and hence bounded. Let $C$ be an upper bound of ${\max}\{\lVert\nabla f^{1}(x_k)\rVert,\lVert\nabla f^{2}(x_k)\rVert,\ldots,\lVert\nabla f^{p}(x_k)\rVert\}$. Then, from (\ref{bounded_u_k1}), we observe that 
 \begin{eqnarray*}
 &&\rho\lVert u_k\rVert^2\leq 2CL\lVert u_k\rVert\implies \lVert u_k\rVert\leq \frac{2CL}{\rho}.
 \end{eqnarray*}
 Thus, we conclude that the sequence $\{u_k\}$ is bounded.
 \end{proof}

Next, we give a proposition on the existence of a step size along the chosen (descent) direction of $F$ at the Step \ref{step3} of Algorithm \ref{algo1} for the set optimization problem (\ref{sp_equation}).

\begin{proposition}\label{armijo}
Let $\beta\in(0,1)$ and $(\bar{a},\bar{u})\in P_{\bar{x}}\times\mathbb{R}^n$ be such that $\Phi(\bar{x})=\xi_{\bar{x}}(\bar{a},\bar{u})$ and assume that the point $\bar{x}$ is not a stationary point of (\ref{sp_equation}). Then,  there exists $\widetilde{t}>0$ such that for all $t\in(0,\widetilde{t}]$ and $j\in[\bar{w}]$, 
    \begin{equation}\label{armijo_1}
    f^{\bar{a}_j}(\bar{x}+t\bar{u})\preceq f^{\bar{a}_j}(\bar{x})+\beta t\nabla f^{\bar{a}_j} (\bar{x})^\top\bar{u}.
    \end{equation}
    Additionally, for all $t\in(0,\widetilde{t}]$ and $j\in[\bar{w}]$, we have  
    \begin{equation}\label{armijo_2}
     F(\bar{x}+t\bar{u})\preceq^{l} \{f^{\bar{a}_j}(\bar{x})+\beta t\nabla f^{\bar{a}_j} (\bar{x})^\top\bar{u}\}_{j\in [\bar{w}]}\prec^l F(\bar{x}).   
    \end{equation}
\end{proposition}

\begin{proof} 
Let us assume that (\ref{armijo_1}) does not hold. Therefore, there exists a sequence $\{t_k\} \searrow 0$ and ${j'}\in[\bar{w}]$ such that
\begin{eqnarray}\label{armijo_3}
&&f^{\bar{a}_{j'}}(\bar{x}+t_k\bar{u})-f^{\bar{a}_{j'}}(\bar{x})-\beta t_k\nabla f^{\bar{a}_{j'}} (\bar{x})^\top\bar{u}\not\in -K\nonumber\\
&\implies&\underset{k\to0}{\lim}\frac{f^{\bar{a}_{j'}}(\bar{x}+t_k\bar{u})-f^{\bar{a}_{j'}}(\bar{x})}{t_k}-\beta\nabla f^{\bar{a}_{j'}} (\bar{x})^\top\bar{u}\not\in -K\nonumber\\
&\implies&(1-\beta)\nabla f^{\bar{a}_{j'}} (\bar{x})^\top\bar{u}\not\in -\text{int}(K)\nonumber\\
&\implies&\nabla f^{\bar{a}_{j'}} (\bar{x})^\top\bar{u}\not\in -\text{int}(K)\text{ since }\beta\in(0,1).
\end{eqnarray}
Note that $\bar{x}$ is not a stationary point of (\ref{sp_equation}) and $(\bar{a},\bar{u})\in P_{\bar{x}}\times\mathbb{R}^n$ is such that $\Phi(\bar{x})=\xi_{\bar{x}}(\bar{a},\bar{u})$. Therefore, in view of Theorem \ref{critical}, we have 
\begin{eqnarray}\label{armijo_4}
&&\xi_{\bar{x}}(\bar{a},\bar{u}) = \Phi(\bar{x}) <0\nonumber\\
&\implies&\underset{j\in[\bar{w}]}{\max}~\mathcal{G}_e(\nabla f^{\bar{a}_j}(\bar{x})^\top \bar{u}+\tfrac{1}{2}\bar{u}^\top B^{\bar{a}_j}(\bar{x})\bar{u})<0\text{ since }\bar{u}\not=0\nonumber\\
&\implies&\mathcal{G}_e(\nabla f^{\bar{a}_j}(\bar{x})^\top \bar{u}+\tfrac{1}{2}\bar{u}^\top B^{\bar{a}_j}(\bar{x})\bar{u})<0\text{ for all }j\in[\bar{w}]\nonumber\\
&\overset{\eqref{strong_convexity_eq1}}{\implies}& \mathcal{G}_e(\nabla f^{\bar{a}_j}(\bar{x})^\top \bar{u})+\tfrac{1}{2}\rho_{\bar{a}_j}\lVert\bar{u}\rVert^2<0\text{ for all }j\in[\bar{w}]\nonumber\\
&\implies&\mathcal{G}_e(\nabla f^{\bar{a}_j}(\bar{x})^\top \bar{u})<0\nonumber\\
&\implies& \nabla f^{\bar{a}_j}(\bar{x})^\top \bar{u}\in -\text{int}(K),
\end{eqnarray}
which is a contradiction to (\ref{armijo_3}). Therefore, we conclude that for every $j\in[\bar{w}]$, the relation (\ref{armijo_1}) holds.\\

Now, for a nonstationary point $\bar{x}$, from (\ref{armijo_1}), we observe that
\allowdisplaybreaks
\begin{eqnarray}\label{armijo_5}
&&f^{\bar{a}_j}(\bar{x})+\beta t\nabla f^{\bar{a}_j} (\bar{x})^\top\bar{u}\prec f^{\bar{a}_j}(\bar{x})\text{ for all }j\in[\bar{w}]\text{ and } t\in(0,\widetilde{t}].
\end{eqnarray}
From Proposition \ref{min_min}, we observe that for every $t\in(0,\bar{t}]$,
\begin{eqnarray*}
 F(\bar{x})&\subseteq& \{f^{\bar{a}_j}(x)\}_{j\in[\bar{w}]}+K\\
&\subseteq& \{f^{\bar{a}_j}(\bar{x})+\beta t\nabla f^{\bar{a}_j} (\bar{x})^\top\bar{u}\}_{j\in[\bar{w}]}+K+\text{ int}(K)\\
&\subseteq& \{f^{\bar{a}_j}(\bar{x}+t \bar{u})\}_{j\in[\bar{w}]}+K+K+\text{ int}(K)\text{ from  Definition \ref{partial}}\\
&\subseteq& F(\bar{x}+t \bar{u})+\text{ int}(K),
\end{eqnarray*}
which implies that for every $t\in(0,\bar{t}]$, \eqref{armijo_2} holds. 
\end{proof}
\section{Convergence Analysis}\label{section5}
Below, we define the notion of the regularity of a point with an essential property for a set-valued mapping, which has a significant role in the convergence of the proposed Algorithm \ref{algo1}.

\begin{definition}(Regular point \cite{bouza2021steepest}).\label{regular_def}
A point $\bar{x}$ is said to be a regular point of $F$ if it satisfies the following conditions:
\begin{enumerate}
    \item[(i)] Min$(F(\bar{x}),K)=$ WMin$(F(\bar{x}),K)$, and 
    \item[(ii)] the cardinality function $w$ in Definition \ref{cardinality_w} is constant in a neighbourhood of $\bar{x}$.
\end{enumerate}
\end{definition}

\begin{lemma}\emph{(See \cite{bouza2021steepest})}.\label{regular_lemma}
Let us assume that $\bar{x}\in \mathbb{R}^n$ is a regular point of $F$. Then, there exists a neighbourhood $U$ of $\bar{x}$ such that for every $x\in U$, $w(x)=\bar{w}$, and $P_x\subseteq P_{\bar{x}}.$
\end{lemma}


Now, we present the main theorem of the paper that proves the global convergence of the proposed method given in Algorithm \ref{algo1}.

\begin{theorem}\label{convergence}
Let $\{x_k\}$ is an infinite sequence generated by Algorithm \ref{algo1} and $\bar{x}$ is an accumulation point for the sequence $\{x_k\}$. Additionally, assume that $\bar{x}$ is a regular point of $F$. Then, $\bar{x}$ is a stationary point of (\ref{sp_equation}).     
\end{theorem}

\begin{proof}
Without loss of generality, let $\{x_k\}$ be a subsequence of $\{x_k\}$ which converge to an accumulation point $\bar{x}$. We prove that $\bar{x}$ is stationary. Towards this, define the functional $\varsigma:\mathcal{P}(\mathbb{R}^m)\to \mathbb{R}\cup\{-\infty\}$ given by
 \[
 \varsigma(A)=\underset{z\in A}{\inf}~\mathcal{G}_e(z)\text{ for all }A\in\mathcal{P}(\mathbb{R}^m).
 \]
From Proposition \ref{gerstewitz}(iv), the function $\mathcal{G}_e$ is monotonic. Therefore, the functional $\varsigma$ is monotone with respect to the preorder $\preceq^l$, i.e., for all $A,B\in\mathcal{P}(\mathbb{R}^m)$, we have 
     \[
     A\preceq^l B\implies \varsigma(A)\leq \varsigma(B).
     \]
     Now in view of (\ref{armijo_2}) of Proposition \ref{armijo}, for every $k=0,1,2,\ldots$, we obtain
     \allowdisplaybreaks
     \begin{eqnarray}\label{convergence_1}
      &&\varsigma(F(x_{k+1}))\nonumber\\
      =&&\varsigma(F(x_k+t_ku_k))\nonumber\\
       \leq&& \underset{j\in[w_k]}{\min}\left\{ \mathcal{G}_e\left(f^{a^k_j}(x_k)+\beta t_k\nabla f^{a^k_j} (x_k)^\top u_k\right)\right\}  \nonumber\\
       \leq&& \underset{j\in[w_k]}{\min}\left\{ \mathcal{G}_e\left(f^{a^k_j}(x_k)+\beta t_k(\nabla f^{a^k_j} (x_k)^\top u_k+\tfrac{1}{2}u_{k}^\top B^{a^{k}_j}(x_k)u_k)\right)\right\} \nonumber\\
       &&\text{ from Proposition \ref{gerstewitz}(iv) and } 
       u_{k}^\top B^{a^k_j}(x_k) u_k>0 \nonumber\\
       \leq&& \underset{j\in[w_k]}{\min}\left\{ \mathcal{G}_e\left(f^{a^k_j}(x_k)\right)+\beta t_k\mathcal{G}_e\left(\nabla f^{a^k_j} (x_k)^\top u_k+\tfrac{1}{2}u_{k}^\top B^{a^k_j}(x_k)u_k\right)\right\}\nonumber\\
       &&\text{ from Proposition \ref{gerstewitz}(i)}\nonumber\\
       \leq&& \underset{j\in[w_k]}{\min}\left\{ \mathcal{G}_e\left(f^{a^k_j}(x_k)\right)+\beta t_k\underset{j'\in[w_k]}{\max}\mathcal{G}_e\left(\nabla f^{a^k_{j'}} (x_k)^\top u_k+\tfrac{1}{2}u_{k}^\top B^{a^k_{j'}}(x_k)u_k\right)\right\}\nonumber\\
       \leq&& \underset{j\in[w_k]}{\min}\mathcal{G}_e\left(f^{a^k_j}(x_k)\right)+\beta t_k\underset{j\in[w_k]}{\max}\left\{\mathcal{G}_e\left(\nabla f^{a^k_{j}} (x_k)^\top u_k+\tfrac{1}{2}u_{k}^\top B^{a^k_{j}}(x_k)u_k\right)\right\}\nonumber\\
       =&&\varsigma(F(x_k))+\beta t_k \Phi(x_k).
     \end{eqnarray}
     Therefore, for a fixed $k\in\mathbb{N}\cup\{0\}$, we get 
     \[
     -\beta t_k\underset{j\in [w_k]}{\max} \left\{\mathcal{G}_e(\nabla f^{a^k_j} (x_k)^\top u_k+\tfrac{1}{2}u_{k}^\top\nabla^2f^{a^k_j}(x_k)u_k\right\}\leq\varsigma(F(x_k))-\varsigma(F(x_{k+1})).
     \]
     On adding the above relation for $k=0,1,\ldots,\kappa$, we obtain
     \begin{align}\label{convergence_2}
         -\beta\sum_{k=0}^{\kappa} t_k\underset{j\in [w_k]}{\max} \left\{\mathcal{G}_e(\nabla f^{a^k_j}(x_k)^\top u_k+\tfrac{1}{2}u_{k}^\top B^{a^k_j}(x_k)u_k\right\}\leq\varsigma(F(x_0))-\varsigma(F(x_{\kappa +1})).
     \end{align}
     Since $\{x_k\}$ is a convergent sequence and $\mathcal{G}_e$ is monotonic,  from (\ref{convergence_2}), we have
     \begin{eqnarray}
     -\beta\underset{\kappa\to\infty}{\lim}\sum_{k=0}^{\kappa} t_k\underset{j\in [w_k]}{\max} \left\{\mathcal{G}_e(\nabla f^{a^k_j}(x_k)^\top u_k+\tfrac{1}{2}u_{k}^\top B^{a^k_j}(x_k)u_k\right\}< +\infty.\label{convergence_6}
     \end{eqnarray}
     Given that $\{x_k\}$ is a sequence of nonstationary points. Therefore, in view of (\ref{stationary_1}),
     for every $a^k\in P_{x_k},u_k\in\mathbb{R}^n,j\in[w(x_k)]$, we get $\nabla f^{a^k_j}(x_k)^\top u_k\in-\text{ int}(K)$. On proceeding in similar manner to (\ref{critical_0}), we can find $t_k>0$ such that  
     \[0<t_k<\left(\frac{-2}{\underset{j\in [w_k]}{\max}\{\mathcal{G}_e({u}_{k}^\top B^{a^k_j}(x_k){u}_{k})\}}\right)\left(\underset{j\in [w_k]}{\max}\{\mathcal{G}_e(\nabla f^{a^k_j}(x_k)^{\top}{u}_{k})\}\right).\]
     In view of the above chosen $t_k$, we conclude that 
     \begin{eqnarray*}
     &&t_k\underset{j\in [w_k]}{\max} \left\{\mathcal{G}_e\left(\nabla f^{a^k_j} (x_k)^\top u_k+\tfrac{1}{2}u_{k}^\top B^{a^k_j}(x_k)u_k\right)\right\}\\
     &\leq& t_k\left\{\underset{j\in[w(\bar{x})]}{\max}\{\mathcal{G}_{e}(\nabla f^{a^k_j}(\bar{x})^{\top}\widetilde{u})\}-\underset{j\in[w(\bar{x})]}{\max}\{\mathcal{G}_{e}(\nabla f^{a^k_j}(\bar{x})^{\top}\widetilde{u})\}\right\}\\
     &=&0.
      \end{eqnarray*}
      Therefore, we get
       \begin{eqnarray}\label{convergence_7}
     &&-t_k\underset{j\in [w_k]}{\max} \left\{\mathcal{G}_e(\nabla f^{a^k_j} (x_k)^\top u_k+\tfrac{1}{2}u_{k}^\top B^{a^k_j}(x_k)u_k\right\}\geq0.
     \end{eqnarray}
     On combining (\ref{convergence_6}) and (\ref{convergence_7}), and taking limit $k\to\infty$, we have
     \begin{eqnarray}\label{convergence_3}
     &&\underset{k\to\infty}{\lim}~t_k\underset{j\in [w_k]}{\max} \left\{\mathcal{G}_e(\nabla f^{a^k_j} (x_k)^\top u_k+\tfrac{1}{2}u_{k}^\top B^{a^k_j}(x_k)u_k\right\}=0.
      \end{eqnarray}
      Since $\bar{x}$ is an accumulation point of the sequence $\{x_k\}$ and the sequences $\{t_k\}$ and $\{u_k\}$ are bounded (Theorem \ref{bounded}), we can find $\bar{t}\in\mathbb{R}_+$, $\bar{u}\in\mathbb{R}^n$ and a subsequence $\mathcal{K}\in \mathbb{N}$ such that
      \[ 
      t_k\overset{k\in\mathcal{K}}{\longrightarrow} \bar{t}\text{ and } u_k\overset{k\in\mathcal{K}}{\longrightarrow} \bar{u}.
      \]
      Note that the number of points in $[p]$ is finite, and $\bar{x}$ is a regular point of $F$. Thus, in view of Lemma \ref{regular_lemma}, for all $k\in\mathcal{K},u\in\mathbb{R}^n$, we have $~w_k=\bar{w},P_{x_k}=\bar{P},a^k=\bar{a}$ and
      \begin{eqnarray}\label{convergence_4}
        &&\Phi(x_k)=\xi_{x_k}(\bar{a},u_k)\leq \xi_{x_k}(a,u)\nonumber\\
        &\text{and}&\xi_{\bar{x}}(\bar{a},\bar{u})\leq \xi_{\bar{x}}(a,u)\text{ on taking limit }k\overset{\mathcal{K}}{\to} \infty.
      \end{eqnarray}
      Now, we analyze the following two cases:
      \begin{enumerate}
          \item[(i)] Let $\bar{t}>0$. In view of (\ref{convergence_3}) and for all $k\in\mathcal{K}$ such that $w_k=\bar{w},P_{x_k}=\bar{P},a^k=\bar{a}$, we have
          \begin{eqnarray}\label{convergence_5}
           &&\underset{k\overset{\mathcal{K}}{\to} \infty}{\lim}\underset{j\in [\bar{w}]}{\max} \left\{\mathcal{G}_e(\nabla f^{\bar{a}_{j}} (x_k)^\top u_k+\tfrac{1}{2}u_{k}^\top B^{\bar{a}_{j}}(x_k)u_k\right\}=0\nonumber\\
    &\implies&\underset{k\overset{\mathcal{K}}{\to} \infty}{\lim}\Phi(x_k)=0.   
          \end{eqnarray}
          Thus, by Theorem \ref{critical}, we conclude that $\bar{u}=0$. Hence, $\bar{x}$ is a stationary point of (\ref{sp_equation}).
          \item[(ii)] Let $\bar{t}=0$. Fix any $\kappa \in \mathbb{N}$. Since $t_k\overset{\mathcal{K}}{\to} \bar t = 0$, large enough $\nu^\kappa$ does not satisfy Armijo condition in Step \ref{step5} of Algorithm \ref{algo1}. Therefore, for all $k\in\mathcal{K}$ such that $w_k=\bar{w},~P_{x_k}=\bar{P}$, and $a^k=\bar{a}$, there exists $\bar{j}\in[\bar{w}]$ such that 
          \begin{eqnarray*}
          &&f^{\bar{a}_{\bar{j}}}(x_k+\nu^\kappa u_k)\npreceq f^{\bar{a}_{\bar{j}}}(x_k)+\beta \nu^\kappa \nabla f^{\bar{a}_{\bar{j}}} (x_k)^\top u_k\\
          &\implies&\frac{f^{\bar{a}_{\bar{j}}}(x_k+\nu^\kappa u_k)-f^{\bar{a}_{\bar{j}}}(x_k)}{\nu^\kappa}-\beta \nabla f^{\bar{a}_{\bar{j}}} (x_k)^\top u_k\not\in -K\\
          &\implies&\frac{f^{\bar{a}_{\bar{j}}}(\bar{x}+\nu^\kappa \bar{u})-f^{\bar{a}_{\bar{j}}}(\bar{x})}{\nu^\kappa}-\beta \nabla f^{\bar{a}_{\bar{j}}} (\bar{x})^\top \bar{u}\not\in -\text{ int}(K)~\text{taking }k\overset{\mathcal{K}}{\to}+\infty\\
          &\implies& (1-\beta)\nabla f^{\bar{a}_{\bar{j}}}(\bar{x})^\top\bar{u}\not\in-\text{ int}(K)~\text{taking limit }k\to+\infty\\
          &\implies&\nabla f^{\bar{a}_{\bar{j}}}(\bar{x})^\top\bar{u}\not\in-\text{ int}(K)~\text{since }(1-\beta)\in(0,1).
          \end{eqnarray*}
          Therefore, from (v) of Proposition \ref{gerstewitz}, we have 
          \begin{eqnarray*}
              &&\mathcal{G}_e(\nabla f^{\bar{a}_{\bar{j}}}(\bar{x})^\top\bar{u})\geq0\\
              &\text{or, }&\mathcal{G}_e(\nabla f^{\bar{a}_{\bar{j}}}(\bar{x})^\top\bar{u}+\tfrac{1}{2}\bar{u}^\top B^{\bar{a}_{\bar{j}}}(\bar{x})\bar{u})\geq0 \notag \\ 
              && \text{ by Proposition \ref{gerstewitz}(iv) and } \bar{u}^\top B^{\bar{a}_{\bar{j}}}(\bar{x})\bar{u} \succ 0\\
              &\text{or, }&0\leq \mathcal{G}_e(\nabla f^{\bar{a}_{\bar{j}}}(\bar{x})^\top\bar{u}+\tfrac{1}{2}\bar{u}^\top B^{\bar{a}_{\bar{j}}}(\bar{x})\bar{u})\\
              &\text{or, }&0\leq \xi_{\bar{x}}(\bar{a},\bar{u})=\underset{(a,u)\in P_{x}\times \mathbb{R}^n}{\min}\xi_x(a,u)=\Phi(\bar{x})\leq0\text{ from (\ref{phi_function})}.
          \end{eqnarray*}
          Thus, from Theorem \ref{critical}, we conclude that $\bar{x}$ is a stationary point of (\ref{sp_equation}). 
      \end{enumerate}
\end{proof}
 {Now, we recall some results that help in analyzing the convergence properties of the proposed quasi-Newton Algorithm \ref{algo1}. The first result follows from \cite[Theorem 3.1]{dennis1977quasi}, which says the following.  
Let $\{x_k\}$ be a sequence of nonstationary points converging to $\bar{x}$. Then, we have the relation given by  
\begin{eqnarray*}
\underset{k\to\infty}{\lim}\frac{\lVert(B(x_k)-\nabla^2f(\bar{x}))s_k\rVert}{\lVert s_k\rVert}=0,
\end{eqnarray*}
where $s_k=x_{k+1}-x_k=t_ku_k$, $B(x_k)$ is the BFGS approximation of the Hessian $\nabla^2 f(x_k)$, and $f$ is twice continuously differentiable such that $\nabla f(\bar{x})=0$ and $\nabla^2 f(\bar{x})$ is positive definite. Now, from \cite{povalej2014quasi}, and in view of above result, for any $\varepsilon>0$, there exists $k_0\in\mathbb{N}$ such that for all $k\geq k_0$, the following holds
\begin{eqnarray*}
\underset{k\to\infty}{\lim}\frac{\lVert(B^{a^k_j}(x_k)-\nabla^2f^{a^k_j}(\bar{x}))u_k\rVert}{\lVert u_k\rVert}<\varepsilon\text{ for every }j\in[w(x)].    
\end{eqnarray*}
\\
Next, we recall the assumptions from \cite{dennis1996numerical}, where for each $j\in[w(x)]$, the authors have estimated the error of approximating $\nabla f^{a^k_j}$ and $f^{a^k_j}$ by its linear and quadratic models, respectively. }

\begin{lemma}\emph{(See \cite{dennis1996numerical}).}\label{lipschitz}
Let $U$ be a nonempty subset of $\mathbb{R}^n$ and $\varepsilon,\delta>0$ be such that for any $x,y\in U$ with $\lVert y-x\rVert<\delta$, the following conditions hold:
\begin{enumerate}
    \item[(i)] For every $j\in [w(x)]$, we have
    \begin{eqnarray}\label{li_1}
    \lVert\nabla^2 f^{a_j}(y)-\nabla^2 f^{a_j}(x)\rVert<\tfrac{\varepsilon}{2}.
    \end{eqnarray}
    \item[(ii)] For every $j\in [w(x)]$ and $x,y\in U$ such that $\lVert y-x\rVert<\delta$, we have
    \begin{eqnarray}\label{li_2}
    \lVert \nabla f^{a_j}(y)-[\nabla f^{a_j}(x)+\nabla^2 f^{a_j}(x) (y-x)]\rVert<\tfrac{\varepsilon}{2}\lVert y-x\rVert.
    \end{eqnarray}
    \item[(iii)] Also, for every $j\in [w(x)]$ and $x,y\in U$ such that $\lVert y-x\rVert<\delta$, we have
    \begin{align}\label{li_3}
    \lvert  f^{a_j}(y)-(f^{a_j}(x)+\nabla f^{a_j}(x)^\top(y-x)+\tfrac{1}{2}(y-x)^\top\nabla^2 f^{a_j}(x) (y-x))\rvert<\tfrac{\varepsilon}{4}\lVert y-x\rVert^2.
    \end{align}
\end{enumerate}
\end{lemma}
Now, we use Lemma \ref{lipschitz} to estimate the error of approximations, where we use the BFGS approximation of the second-order derivative term Hessian.

\begin{lemma}\label{bfgs_assumption}
Let $U$ be a nonempty subset of $\mathbb{R}^n$. Let $V\subset U$ be a convex subset and $\delta\in \mathbb{R}_+$ be constants such that $x, y\in V$ with $\lVert y-x\rVert<\delta$. Let $\{x_k\}$ be a sequence generated by Algorithm \ref{algo1}. Assume that for every $\varepsilon>0$ and for every $j\in[w(x_k)]$, there exists $k_0\in\mathbb{N}$ such that for all $k\geq k_0$, we have 
\begin{eqnarray}\label{bfgs_eq1}
\frac{\left\lVert (\nabla^2 f^{a_j}(x_k)-B^{a_j}(x_k))(y-x_k)\right\rVert}{\lVert y-x_k\rVert}<\tfrac{\varepsilon}{2}.
\end{eqnarray}
Then, for any $x_k$ and $k\geq k_0$, and any $x,y\in V$ such that $\lVert y-x_k\rVert<\delta$, we have   
\begin{align}\label{assump_1} 
\lVert \nabla f^{a_j}(y)-(\nabla f^{a_j}(x_k)+B^{a_j}(x_k) (y-x_k))\rVert<\varepsilon\lVert y-x_k\rVert
\end{align}
and
\begin{align}\label{assump_2}
    \left\lvert  f^{a_j}(y)-\left(f^{a_j}(x_k)+\nabla f^{a_j}(x_k)^\top(y-x_k)+\tfrac{1}{2}(y-x_k)^\top B^{a_j}(x_k) (y-x_k)\right)\right\rvert<\tfrac{\varepsilon}{2}\lVert y-x_k\rVert^2,
\end{align}
for every $j\in [w(x_k)]$.
\end{lemma}

\begin{proof}
In view of (\ref{li_2}) of Lemma \ref{lipschitz}, for every $j\in [w(x_k)]$ and $x_k,y\in V$ such that $\lVert y-x_k\rVert<\delta$, we have
    \begin{eqnarray*}
    &&\lVert \nabla f^{a_j}(y)-[\nabla f^{a_j}(x_k)+B^{a_j}(x_k) (y-x_k)]\rVert\\
    &\leq&\lVert \nabla f^{a_j}(y)-\nabla f^{a_j}(x_k)-\nabla^2 f^{a_j}(x_k)(y-x_k)]\rVert+\lVert (\nabla^2 f^{a_j}(x_k)-B^{a_j}(x_k)) (y-x_k)]\rVert\\
    &<&\frac{\varepsilon}{2}\lVert y-x_k\rVert+\frac{\varepsilon}{2}\lVert y-x_k\rVert\text{ from (\ref{bfgs_eq1})}\\
    &<&\varepsilon\lVert y-x_k\rVert.
    \end{eqnarray*}
In the similar manner in (\ref{li_3}) of Lemma \ref{lipschitz}, for every $j\in [w(x_k)]$ and $x_k,y\in U$ such that $\lVert y-x_k\rVert<\delta$, we have
\allowdisplaybreaks
    \begin{eqnarray*}
    &&\left\lVert f^{a_j}(y)-\left(f^{a_j}(x_k)+\nabla f^{a_j}(x_k)^\top(y-x_k)+\tfrac{1}{2}(y-x_k)^\top B^{a_j}(x_k) (y-x_k)\right)\right\rVert\\
    &\leq&\left\lVert f^{a_j}(y)-f^{a_j}(x_k)-\nabla f^{a_j}(x_k)^\top(y-x_k)-\tfrac{1}{2}(y-x_k)^\top\nabla^2 f^{a_j}(x_k) (y-x_k)\right\rVert \\ && +\left\lVert \tfrac{1}{2}(y-x_k)^\top(\nabla^2 f^{a_j}(x_k)-B^{a_j}(x_k)) (y-x_k)\right\rVert\\
    &<&\frac{\varepsilon}{4}\lVert y-x_k\rVert^2+\frac{\varepsilon}{4}\lVert y-x_k\rVert^2\text{ from (\ref{bfgs_eq1})}\\
    &<&\frac{\varepsilon}{2}\lVert y-x_k\rVert^2.
    \end{eqnarray*}    
\end{proof}

\begin{theorem}\emph{(Superlinear convergence).}\label{superlinear}
Let $\{x_k\}$ be a sequence of nonstationary points generated by Algorithm \ref{algo1} and $\bar{x}$ be one of its accumulation points. Additionally, assume that $\bar{x}$ is a regular point of $F$, and  there exists a nonempty convex set $V \subseteq \mathbb{R}^n$ and $p>0,q>0,\delta>0,\epsilon>0$ for which the following conditions hold: 
\begin{enumerate}
    \item[(i)]\label{supe_1} $pI\leq B^{a_j}(x)\leq qI$ \text{ for all }$j\in [w(x)]$,
    \item[(ii)]\label{supe_2}  {$\lVert\nabla^2 f^{a_j}(x)-\nabla^2f^{a_j}(y) \rVert<\frac{\varepsilon}{2} $\text{ for all } $x,y\in V$\text{ with }$\lVert x-y \rVert<\delta$,}
    \item[(iii)]  {$\lVert (\nabla^2 f^{a_j}(x_k)-B^{a_j}(x_k)) (y-x_k)]\rVert<\frac{\varepsilon}{2}$ for all $x,y\in V$\text{ with }$\lVert x-y \rVert<\delta$}, \text{ and }  
    \item[(iv)] \label{supe_3} $\tfrac{\varepsilon}{q}\leq 1-2\beta$.
\end{enumerate}
Then, for sufficiently large $k\in\mathbb{N}$, we have $t_k=1$ and the sequence $\{x_k\}$ converges superlinearly to $\bar{x}\in\mathbb{R}^n$.
\end{theorem}

\begin{proof}
From Theorem \ref{convergence}, it can be observed that the sequence $\{x_k\}$ converges to a stationary point $\bar{x}$ of (\ref{sp_equation}). Moreover, each $f^{a_j}$ is twice continuously differentiable. Therefore, for any $\varepsilon>0$, there exists $\delta_{\varepsilon}>0$ such that for all $x,y\in \mathcal{B}(\bar{x},\delta_\varepsilon)$, we have
\[
\mathcal{B}(\bar{x},\delta_\varepsilon)\subset V, ~ \lVert B^{a_j}(x)-B^{a_j}(y)\rVert<\varepsilon\text{ with }\lVert x-y \rVert<\delta_{\varepsilon}. 
\]
 {Now, for $x\in U$, and $\Delta_{w(x)} = \{(\lambda_1, \lambda_2, \ldots, \lambda_{w(x)}) \in \mathbb{R}^{w(x)}_+\text{ with }\sum_{i = 1}^{w(x)} \lambda_i = 1\}$, we define a function $\Theta_{\lambda}:V\times \mathbb{R}^n\to \mathbb{R}^m$ such that 
 \[
 \Theta_{\lambda}(x,u)=\sum_{j=1}^{[w(x)]}\lambda_j\nabla f^{a_j}(x)^\top u+\tfrac{1}{2}\sum_{j=1}^{[w(x)]}\lambda_j u^\top B^{a_j}(x)^\top u,~\lambda_{w(x)}\in\Delta_{w(x)}.
 \]}
Note that for any $a^k\in P_{x_k}$ and $j\in[w(x_k)]$, each $f^{a_j^k}$ is twice continuously differentiable and strongly convex function. Moreover, for any $x_k\in\mathbb{R}^n$, and $a^k\in P_{x_k}$, the set $P_{x_k}$ is finite. Therefore, the function $\Theta_{\lambda}(a,\cdot)$ is strongly convex in $\mathbb{R}^n$ and hence the function $\Theta_{\lambda}(a,\cdot)$ attains its minimum. Then, using Danskin's theorem (see Proposition 4.5.1, pp. 245--247 in \cite{bertsekas2003convex}) and the first order necessary condition for the existence of a minimizer, we conclude that
\begin{eqnarray}
&&\sum_{j=1}^{[w_k]}\lambda_j\nabla f^{a^k_j}(x_{k})+\sum_{j=1}^{[w_k]}\lambda_jB^{a^k_j}(x_{k}) u_k=0\label{super_2}\\
\implies&& u_k= -\left[\sum_{j=1}^{[w_k]}\lambda_jB^{a^k_j}(x_{k})\right]^{-1}\sum_{j=1}^{[w_k]}\lambda_j\nabla f^{a^k_j}(x_{k})\nonumber\\
\implies&& u_k \leq -\tfrac{1}{q} \sum_{j=1}^{[w_k]}\lambda_j\nabla f^{a^k_j}(x_{k})\text{ since }B^{a^k_j}(x_k)\leq qI\nonumber\\
\implies&& u_k\leq-\tfrac{1}{q} \underset{\lambda\in \Delta_k}{\max}\sum_{j=1}^{[w_k]}\lambda_j\nabla f^{a^k_j}(x_{k})\label{super_1}.
\end{eqnarray}
As the sequence $\{x_k\}$ converges to $\bar{x}$, there exists $k_{\varepsilon}\in\mathbb{N}$ such that for all $k\geq k_\varepsilon$, we have $
x_k,x_k+u_k\in \mathcal{B}(\bar{x},\delta_{\varepsilon})$. Now, using the second-order Taylor expansion at $x_k$ of $f^{a^k_j}$, we have
\begin{eqnarray*}
 f^{a^k_j}(x_k+u_k)&\leq& f^{a^k_j}(x_k) + \nabla f^{a^k_j}(x_k)^\top u_k+\tfrac{1}{2}u_k^\top B^{a^k_j}(x_k)u_k +\tfrac{\varepsilon}{2}\lVert u_k\rVert^2.
 \end{eqnarray*}
 Note that $\max\{b_1,b_2,\ldots,b_{w(x)}\}=\underset{\lambda\in \Delta_{w(x)}}{\max}\sum_{i=1}^{w(x)} \lambda_i b_i $. Therefore, applying this identity in the above relation, we get
 \begin{eqnarray*}
&&f^{a^k_j}(x_k+u_k)- f^{a^k_j}(x_k)\\
&\leq& \nabla f^{a^k_j}(x_k)^\top u_k+\tfrac{1}{2}u_k^\top B^{a^k_j}(x_k)u_k +\tfrac{\varepsilon}{2}\lVert u_k\rVert^2\\
 &\leq& \beta \nabla f^{a^k_j}(x_k)^\top u_k+(1-\beta)\nabla f^{a^k_j}(x_k)^\top u_k+\tfrac{(q+\varepsilon)}{2}\lVert u_k\rVert^2\text{ since }B^{a^k_j}(x)\leq qI\\
  &\leq& \beta \nabla f^{a^k_j}(x_k)^\top u_k+(1-\beta)\underset{j\in[w(x_k)]}{\max}\{\nabla f^{a_j}(x_k)^\top u_k\}+\tfrac{(q+\varepsilon)}{2}\lVert u_k\rVert^2\\
  &\leq& \beta \nabla f^{a^k_j}(x_k)^\top u_k+(1-\beta)\underset{\lambda\in\Delta_k}{\max}\left\{\sum_{j=1}^{[w(x_k)]}\lambda_j\nabla f^{a^k_j}(x_k)^\top u_k\right\}+\tfrac{(q+\varepsilon)}{2}\lVert u_k\rVert^2 \\
 &\leq& \beta \nabla f^{a^k_j}(x_k)^\top u_k-q(1-\beta)\lVert u_k\rVert^2 + \tfrac{(q+\varepsilon)}{2}\lVert u_k\rVert^2\text{ from (\ref{super_1})}\\ 
 &\leq& \beta \nabla f^{a^k_j}(x_k)^\top u_k+\frac{\varepsilon-q(1-2\beta)}{2}\lVert u_k\rVert^2,
\end{eqnarray*}
where from the assumption (iii), we conclude that $\varepsilon-q(1-2\beta)\leq 0$ and $t_k=1$ hold in the above relation. Now, for $k\geq k_\varepsilon$, $\lambda\in \Delta_k$, and $j\in[w(x_k)]$, we have
\allowdisplaybreaks
\begin{align}\label{super_3}
&\left\lVert \sum_{j=1}^{[w(x_{k+1})]}\lambda_j\nabla f^{a^k_j}(x_{k+1}) \right\rVert\nonumber\\
=&\left\lVert  \sum_{j=1}^{[w(x_{k+1})]}\lambda_j\nabla f^{a^k_j}(x_{k}+u_k) \right\rVert\nonumber\\
=& \left\lVert  \sum_{j=1}^{[w(x_{k+1})]}\lambda_j\nabla f^{a^k_j}(x_{k}+u_k)-\left[\sum_{j=1}^{[w(x_k)]}\lambda_j\nabla f^{a^k_j}(x_{k})+\sum_{j=1}^{[w(x_k)]}\lambda_jB^{a^k_j}(x_{k}) u_k\right]\right\rVert\text{ from (\ref{super_2})}\nonumber\\
\leq& \varepsilon \lVert u_k\rVert \text{ from (\ref{assump_1}) of Lemma \ref{bfgs_assumption}}. 
\end{align}
Now, combining the assumption (i) and boundedness of $\{u_{k+1}\}$ (Theorem \ref{uk_bounded}), we observe that
\begin{eqnarray}
\tfrac{1}{2}u_k^\top B^{a^k_j}(x_k)u_k\leq \tfrac{q}{2}\lVert u_k\rVert^2 \text{ for any }j\in[w(x_k)].    
\end{eqnarray}
Therefore, incorporating the above relation in (\ref{bounded_eq1}), we get
\begin{align*}
&\lVert u_{k+1}\rVert\\
\leq&\tfrac{2L}{q} \underset{j\in [w(x_k)]}{\max}\left\lVert\nabla f^{a^k_j}(x_{k+1})\right\lVert\\
\leq &\tfrac{2L}{q}\left \{\underset{\lambda\in\Delta_k}{\max}\left\lVert\sum_{j=1}^{[w(x_k)]}\nabla f^{a^k_j}(x_{k+1})\right\lVert\right\}\text{using }\max\{b_1,b_2,\ldots,b_{w(x)}\}=\underset{\lambda\in \Delta_{w(x)}}{\max}\sum_{i=1}^{w(x)} \lambda_i b_i\\
\leq&\tfrac{2L\varepsilon}{q}\lVert u_k\rVert\text{ from (\ref{super_3})}.    
\end{align*}
In view of the above relation, we have
\begin{align*}
\lVert x^{k+1}-x^{k+2}\rVert= \lVert u^{k+1}\rVert  \leq \tfrac{2L\varepsilon}{q} \lVert u_k\rVert =\tfrac{2L\varepsilon}{q} \lVert x^{k}-x^{k+1}\rVert, 
\end{align*}
and for any $k\geq1$ and $m\geq1$, we obtain
\begin{align}\label{superlinear_1}
\lVert x^{k+m}-x^{k+m+1}\rVert\leq& \left(\tfrac{2L\varepsilon}{q}\right)
\lVert x^{k+m-2}-x^{k+m-1}\rVert\nonumber\\
\leq& \left(\tfrac{2L\varepsilon}{q}\right)^2\lVert x^{k+m-1}-x^{k+m}\rVert\nonumber\\
\leq&\cdots\leq \left(\tfrac{2L\varepsilon}{q}\right)^m\lVert x^{k}-x^{k+1}\rVert. 
\end{align}
Now, we assume $0<\tau<1$ and define 
\[
\bar{\varepsilon}=\min\left\{q(1-2\beta),\tfrac{\tau}{1+2\tau}\left(\tfrac{q}{2L\varepsilon}\right)\right\}.
\]
If we take $\varepsilon<\bar{\varepsilon}$ and $k\geq k_\varepsilon$, then by convergence of $\{x_k\}$ and the relation (\ref{superlinear_1}), we have
\begin{eqnarray*}
\lVert \bar{x}-x^{k+1}\rVert\leq \sum_{m=1}^{\infty}\lVert x^{k+m}-x^{k+m+1}\rVert &\leq& \sum_{m=1}^{\infty}\left(\tfrac{\tau}{1+2\tau}\right)^m\lVert x^{k}-x^{k+1}\rVert\\
&=&\tfrac{\tau}{1+\tau}\lVert x^{k}-x^{k+1}\rVert. 
\end{eqnarray*}
Therefore, we obtain
\begin{align*}
&\lVert \bar{x}-x^{k}\rVert \geq \lVert x^{k}-x^{k+1}\rVert-\lVert x^{k+1}-\bar{x}\rVert \geq \tfrac{1}{1+\tau}\lVert x^{k}-x^{k+1}\rVert.
\end{align*}
Hence, we can conclude that if $\varepsilon<\bar{\varepsilon}$ and $k\geq k_\varepsilon$, then 
$\frac{\lVert \bar{x}-x^{k+1}\rVert}{\lVert \bar{x}-x^{k}\rVert}\leq \tau.$
\end{proof}

\section{Numerical Execution} \label{section6}
In this section, we execute the proposed quasi-Newton Algorithm \ref{algo1} on some numerical examples. We analyze the experimentation of Algorithm \ref{algo1} in MATLAB R2023b software. This MATLAB software is installed in an IOS machine equipped with a 12-core CPU and 8 GB RAM. In the numerical  implementation of the algorithm, we have considered the following parameter values:
\begin{itemize}
    \item We considered the cone $K$ to be a standard ordering cone, i.e.,  {$K=\mathbb{R}^2_+$} for all test instances except Example \ref{example5} and Example \ref{example6}, and the parameter  {$e=(1,1)^\top\in$ int($K$)} for the scalarizing function $\mathcal{G}_e$.
    \item The parameters $\beta$ and $\nu$ in Step \ref{step5} for the line search of the Algorithm \ref{algo1} is chosen as $\beta=0.5$ and $\nu=0.6$.
    \item The employed stopping criterion is $\lVert u_k\rVert <0.001$, or a maximum number of 100 iterations is reached.
    \item To figure out the set Min ($F(x_k),K$) at the $k$-th iteration in Step \ref{step2} of Algorithm \ref{algo1}, we adopt the common method of comparing the elements in $F(x_k)$. 
    \item At the $k$-th iteration in Step \ref{step2} of Algorithm \ref{algo1}, for every $a^k\in P_k$, we compute the unique solution $u_a$ of the function given by
    \[
     \underset{u\in \mathbb{R}^n}{\min}~\xi_{x_k}(a,u).
    \]
    In the conclusion, we find $ (a^k,u_k)=\underset{(a,u)\in P_k\times\mathbb{R}^n}{\text{argmin}}\xi_{x_k}(a,u)$ with the help of an inbuilt function \textit{fminsearch} in MATLAB.
    \item We have considered some test problems from the literature subjected to slight modifications and some freshly introduced problems. For each test considered, we generated 100 initial points randomly and ran the algorithm. In the context of each experiment, we have presented a table with three columns. The resulting error is the value of $\lVert u_k\rVert$ at the end of the final iteration. The subsequent values are gathered for each test instance: 
    \begin{itemize}
        \item \textbf{Initial Points:}   The value represents the first column of the table, which counts the number of initial points taken to solve the proposed Algorithm \ref{algo1}.
        \item \textbf{Iterations:} 
          This value presents the second column with a 6-tuple (Min, Max, Mean, Median, Mode, SD) whose components are the minimum, maximum, mean, median, mode, and standard deviation of the iterations in which the stopping condition is reached.
          \item \textbf{CPU Time:} This value indicates the third column, which is again a 6-tuple (Min, Max, Mean, Median, $\lceil\text{Mode}\rceil$, SD) that shows the minimum, maximum, mean, median, least integer greater or equal to mode, and standard deviation of the CPU time (in seconds) taken by the initial point in reaching the stopping condition.
    \end{itemize}
\end{itemize}

Additionally, the numerical values are presented with precision up to four decimal places to ensure clarity. For every examined problem, the values of $F$ at each iteration for the initial and final points are marked with black and red colors, respectively. We use shapes $\bullet,~\boldsymbol{\star},$ and $\blacktriangle$ to depict the values $F$ for different initial points. Cyan, magenta, and green colors are used to represent the intermediate iterates for different initial points. Initial points are depicted in black, and the termination point is in red. If the initial point is depicted by black \emph{bullet} $\bullet$, then the terminating is depicted by the red \emph{bullet} {\red $\bullet$}, and the intermediate iterates by cyan \emph{bullets} {\cyan $\bullet$} or magenta \emph{bullets} {\magenta $\bullet$} or green \emph{bullets} {\green $\bullet$}. That is, we use the same shape for depicting a complete sequence of iterates generated by Algorithm \ref{algo1}. 

Furthermore, we compare the results of the proposed quasi-Newton's method (abbreviated as QNM) algorithm (Algorithm \ref{algo1}) with the existing steepest descent method (abbreviated as SD) for set optimization presented in \cite{bouza2021steepest}.\\

Now, we discuss different test problems on which our algorithm was tested. The first test problem is freshly introduced.


\begin{example}\label{example1}
Consider the set-valued function $F:\mathbb{R}\rightrightarrows \mathbb{R}^2$ defined as
$$F(x)=\{f^1(x),f^2(x),\ldots,f^{50}(x)\},$$
where for each $i\in[50],~f^i:\mathbb{R}\to\mathbb{R}^2$ is given by 
\[
f^i(x)=\begin{pmatrix}
xe^x+\sin\left(\tfrac{2\pi(i-1))}{50}\right)\\
2x\cos(2x)+\cos\left(\tfrac{2\pi(i-1))}{50}\right)
\end{pmatrix}.
\]
Output of Algorithm \ref{algo1} for different initial points of Example \ref{example1} are depicted in Figure \ref{figure1}.   Figure \ref{figure1a} depicts the sequence $\{F(x_k)\}$ generated by Algorithm \ref{algo1} for the   starting point as $x_0=2.3000$. In Figure \ref{figure1b}, we exhibit the output of Algorithm \ref{algo1} for three initial points. The performance of Algorithm \ref{algo1} for Example \ref{example1} is shown in Table \ref{table1}. The values in Table \ref{table1} show that the proposed method performs better than the existing SD method.

\begin{figure}[ht]
\centering
\mbox{\subfloat[The value of $F$ at each iteration generated by Algorithm \ref{algo1} on Example \ref{example1} for the initial point $x_0=2.3000$]{ \includegraphics[width=0.475\textwidth]{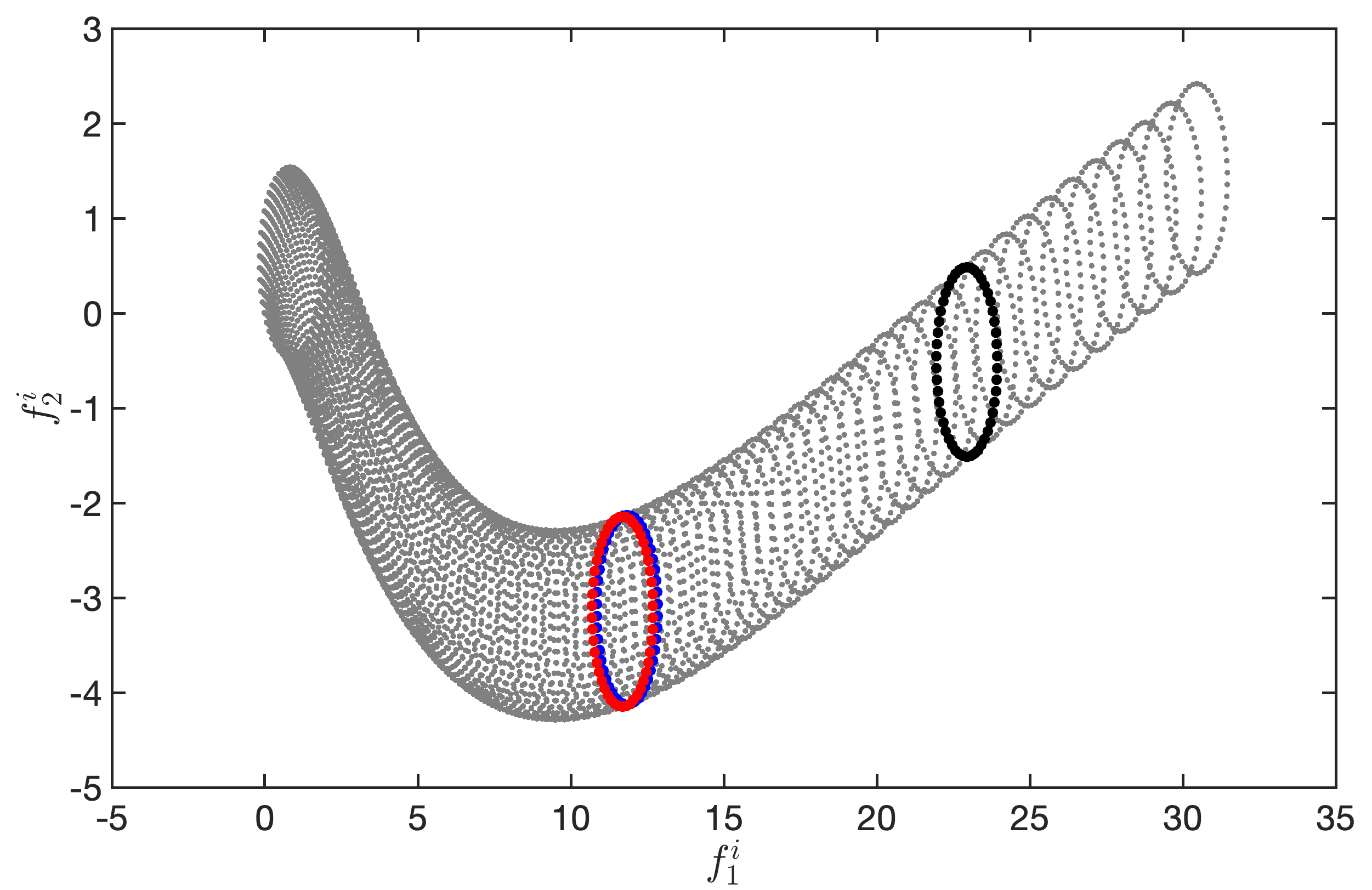}\label{figure1a}}\quad
\subfloat[The value of $F$ for three different initial points at each iteration generated by Algorithm \ref{algo1} on Example \ref{example1}]{\includegraphics[width=0.475\textwidth]{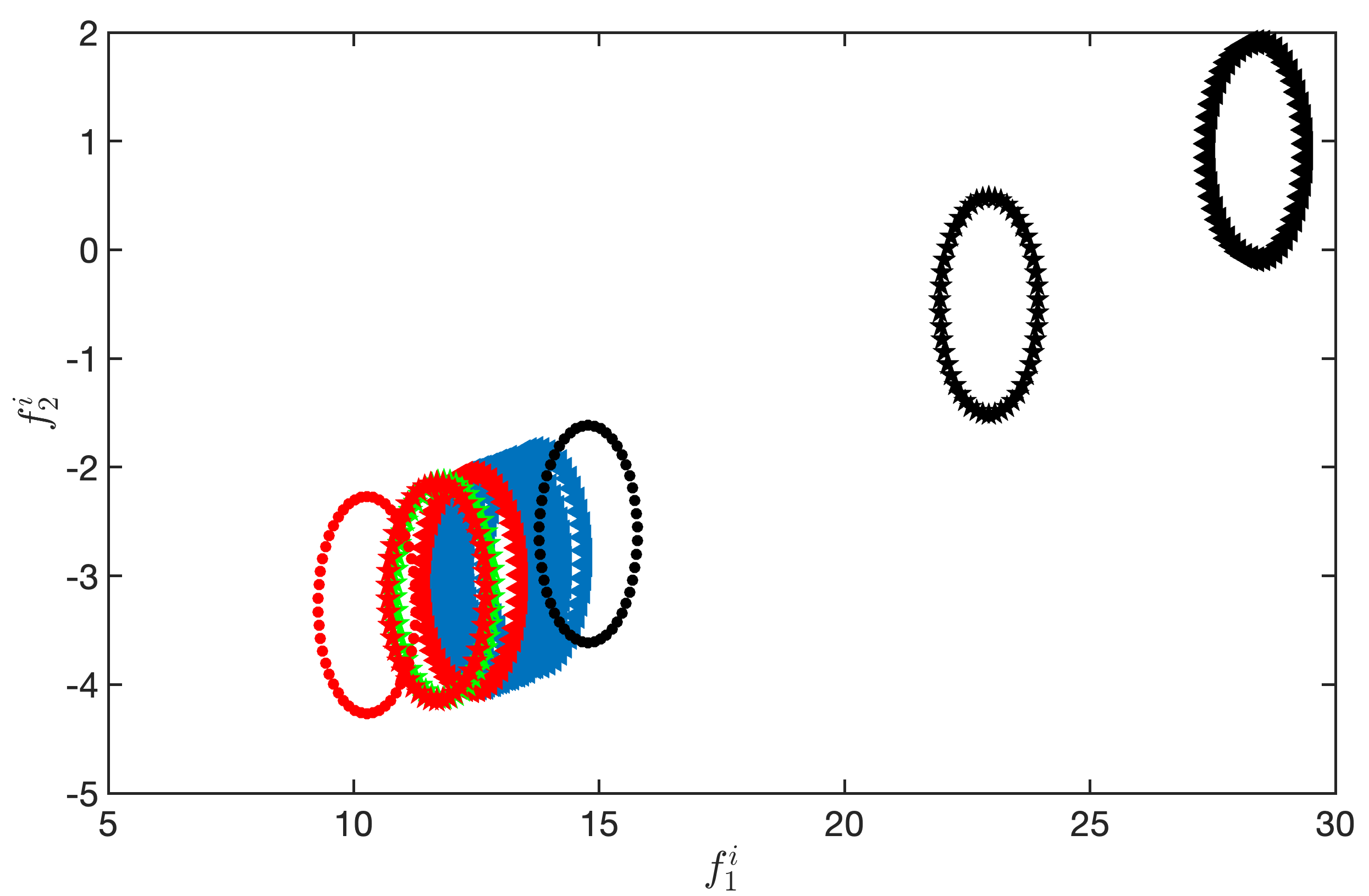}\label{figure1b}}\quad}
\caption{Obtained output on Algorithm \ref{algo1} for Example \ref{example1}}
\label{figure1} 
\end{figure}

\begin{table}[ht]
\centering
\caption{Performance of Algorithm \ref{algo1} on Example \ref{example1}}
\centering 
\scalebox{0.70}{
\begin{tabular}{c c c c} 
\hline 
Number of& Algorithm  &Iterations &CPU time\\
initial points& &(Min, Max, Mean, Median, Mode, SD)  &(Min, Max, Mean, Median, $\lceil\text{Mode}\rceil$, SD)   \\ 
\hline 
$100$ & QNM&($1,~22,~19.1700,~19,~21,~2.4621$) & ($3.2973,~44.5769,~37.5000,~37.6309,~3,~4.7245$) \\ 
    & SD &($1,~33,~32.4600,~32,~32,~0.5009$) & ($2.1389,~32.6268,~31.5719,~31.3467,~30,~0.5165$) \\
\hline 
\end{tabular}}
\label{table1}
\end{table}

Further, for the initial point $x_0=2.3000$, the decreasing behavior in the values of vector-valued functions at each iteration has been exhibited in Table \ref{table_1c}.

\begin{table}[ht]
\centering
\caption{ {Output of Algorithm \ref{algo1} on Example \ref{example1} for the initial point $x_0=2.3000$}}
\centering 
\scalebox{0.73}{
\begin{tabular}{c c c c c } 
\hline
$k$& $x_k^\top$ &$f^{10}(x_k)$ & $f^{25}(x_k)$ & $f^{50}(x_k)$     \\ 
\hline 
 0&  $2.3000$& $(23.8454, -0.0901)$& $(23.0660, -1.5080)$ &$(22.8153, 0.4762)$  \\ 
 1&  $1.8541$& $(12.7453,-2.7029)$& $(11.9658,-4.1208)$ &$(11.7151,-2.1366)$  \\  
 2&  $1.8458$& $(12.5946,-2.7214)$& $(11.8151,-4.1393)$ &$(11.5644,-2.1551)$   \\
3&  $1.8397$& $(12.4845,-2.7343)$& $(11.7050,-4.1522)$ &$(11.4543,-2.1679)$  \\

\hline 
\end{tabular}}
\label{table_1c}
\end{table}
\end{example}

\begin{example}\label{example2}
Consider the set-valued function $F:\mathbb{R}\rightrightarrows \mathbb{R}^3$ defined as
$$F(x)=\{f^1(x),f^2(x),\ldots,f^{30}(x)\},$$
where for each $i\in[30],f^i:\mathbb{R}\to\mathbb{R}^3$ is given by 
\[
f^i(x)=\begin{pmatrix}
0.27\sin\left(\tfrac{2\pi(i-1))}{30}\right)\cos\left(\tfrac{2\pi(i-1))}{30}\right)+x^2\\
\cos(2x)+\tfrac{1}{(1+e^{2x})}+0.27\cos\left(\tfrac{2\pi(i-1))}{30}\right)\\
0.27x^2+\left(\tfrac{i-1}{30}\right)
\end{pmatrix}.
\]
Output of Algorithm \ref{algo1} for different initial points of Example \ref{example2} are depicted in Figure \ref{figure2}. Figure \ref{figure2a} depicts the sequence $\{F(x_k)\}$ generated by Algorithm \ref{algo1} for the  starting point as $x_0=2.1300$. In Figure \ref{figure2b}, we exhibit the output of Algorithm \ref{algo1} for three initial points and depict their corresponding $F$-values. 

The performance of Algorithm \ref{algo1} for Example \ref{example2} is shown in Table \ref{table2}. Moreover, we have compared the results of the QNM with the SD method for set optimization as presented in Table \ref{table2}. The values in Table \ref{table2} show that the proposed method performs better than the existing SD method. 

\begin{figure}[ht]
\centering
\mbox{\subfloat[The value of $F$ at each iteration generated by Algorithm \ref{algo1} on Example \ref{example2} for the initial point $x_0=2.1300$]{ \includegraphics[width=0.475\textwidth]{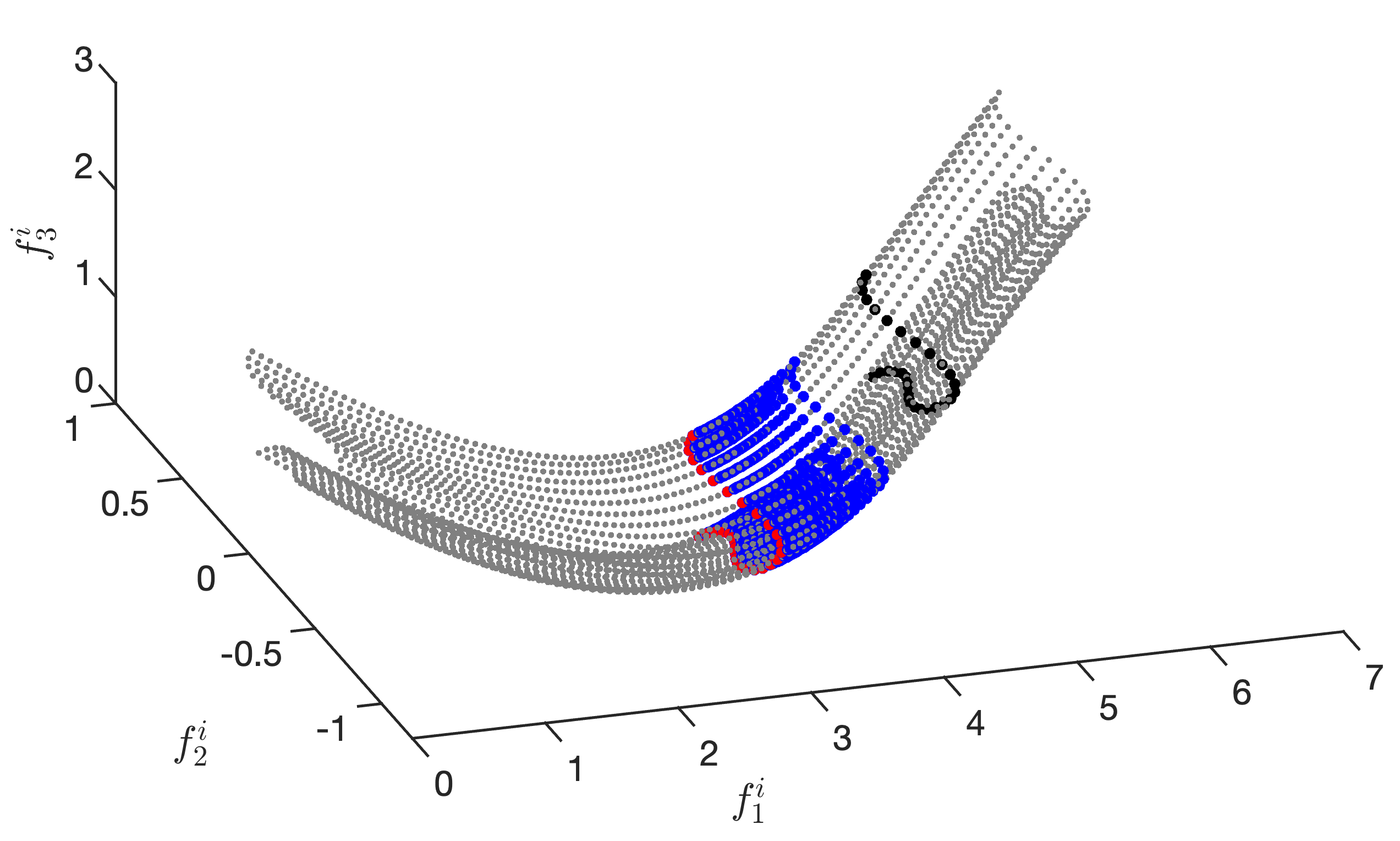}\label{figure2a}}\quad
\subfloat[The value of $F$ for three different initial points at each iteration generated by Algorithm \ref{algo1} on Example \ref{example2}]{\includegraphics[width=0.475\textwidth]{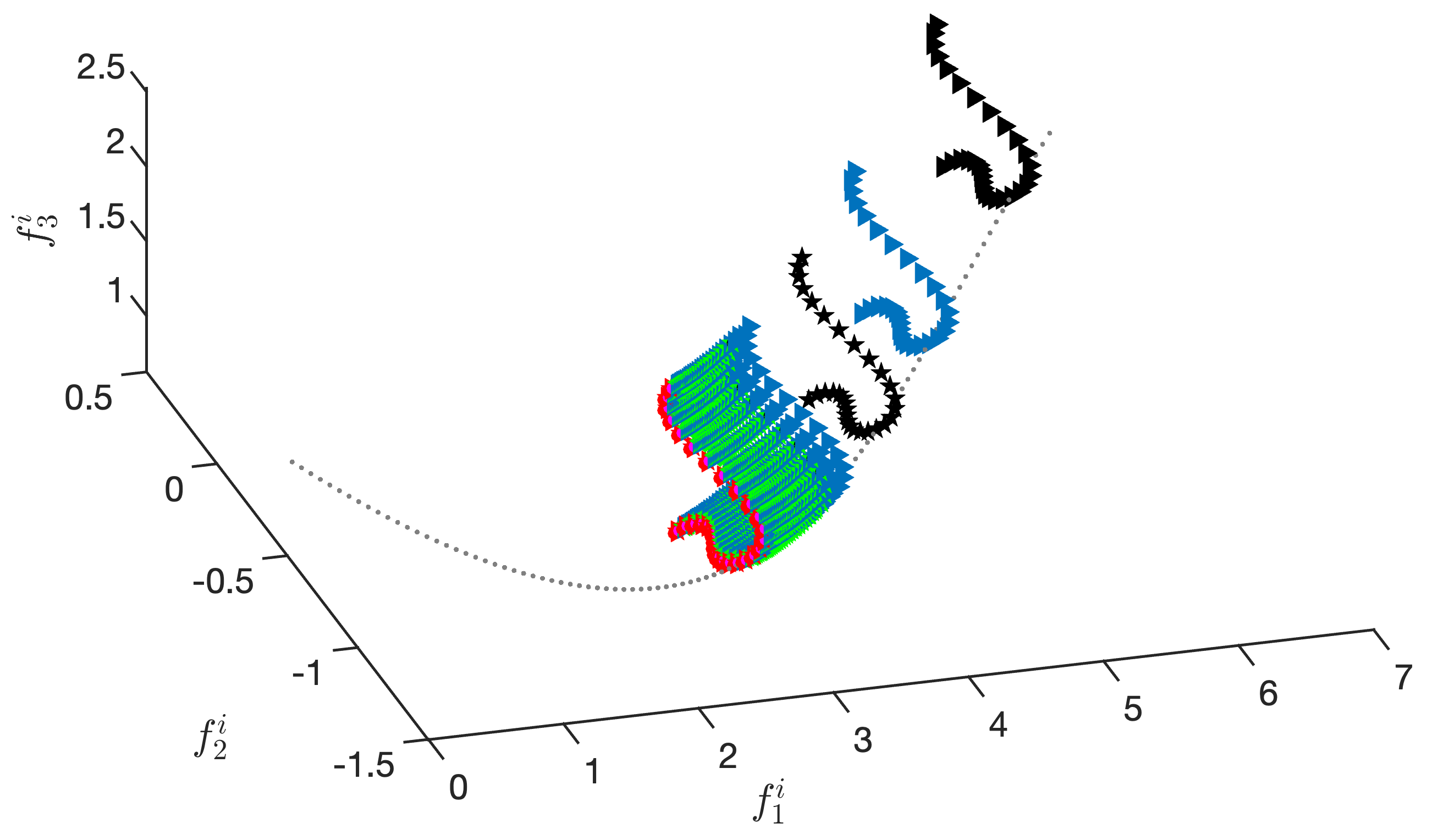}\label{figure2b}}\quad}
\caption{Obtained output of Algorithm \ref{algo1} for Example \ref{example2}}
\label{figure2} 
\end{figure}

\begin{table}[ht]
\centering
\caption{Performance of Algorithm \ref{algo1} on Example \ref{example2}}
\centering 
\scalebox{0.675}{
\begin{tabular}{c c c c} 
\hline
Number of& Algorithm  &Iterations &CPU time\\
initial points& &(Min, Max, Mean, Median, Mode, SD)  &(Min, Max, Mean, Median, $\lceil\text{Mode}\rceil$, SD)   \\ 
\hline 
$100$ & QNM&($1,~10,~5,~4.5800
,~5,~1,~1.9719$) & ($2.5750,~23.7773,~14.9322,~16.2000,~2,~10.2129$) \\ 
    & SD &($1,~11,~10.220,~9,~2,2.4887$) & ($1.2113,~28.1483,~26.1314,~26.1002,~25,~0.2662$) \\
\hline 
\end{tabular}}
\label{table2}
\end{table}
\end{example}


\begin{example}\label{example3}
Consider the function $F:\mathbb{R}\rightrightarrows \mathbb{R}^2$ defined as
$$F(x)=\{f^1(x),f^2(x),\ldots,f^{25}(x)\},$$
where for each $i\in[25],f^i:\mathbb{R}\to\mathbb{R}^2$ is given by 
\[
f^i(x)=\begin{pmatrix}
x_1^2+\cos(x_2)+\cos\left(\tfrac{2\pi(i-1)}{100}\right)\sin\left(\tfrac{2\pi(i-1)}{100}\right)^2+x_2^2\\
2x_1^2+\sin(x_1)+\cos\left(\tfrac{2\pi(i-1)}{100}\right)^2\sin\left(\tfrac{2\pi(i-1)}{100}\right)+2x_2^2
\end{pmatrix}.
\]
\begin{figure}
    \centering
    \subfloat[~The value of $F$ at each iteration generated by Algorithm \ref{algo1} for initial point $x_0=(1.0000,-1.5000)^\top$ of Example \ref{example3}]{\includegraphics[width=0.48\textwidth]{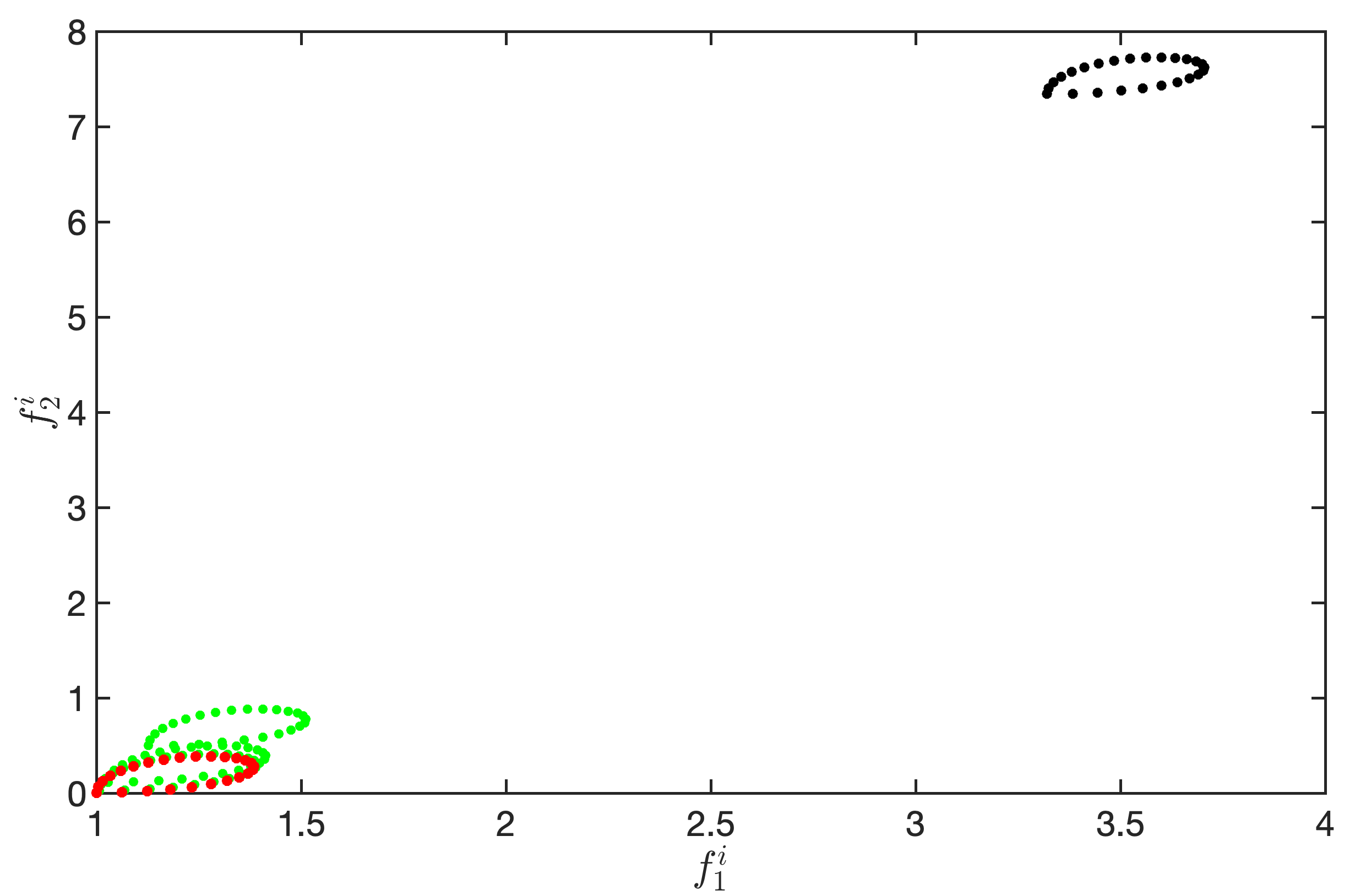}\label{figur3a}} \quad
    \subfloat[~The movement of subsequent $x_k$ generated by Algorithm \ref{algo1} of initial point $x_0=(1.0000,-1.5000)^\top$ of Example \ref{example3}]{\includegraphics[width=0.48\textwidth]{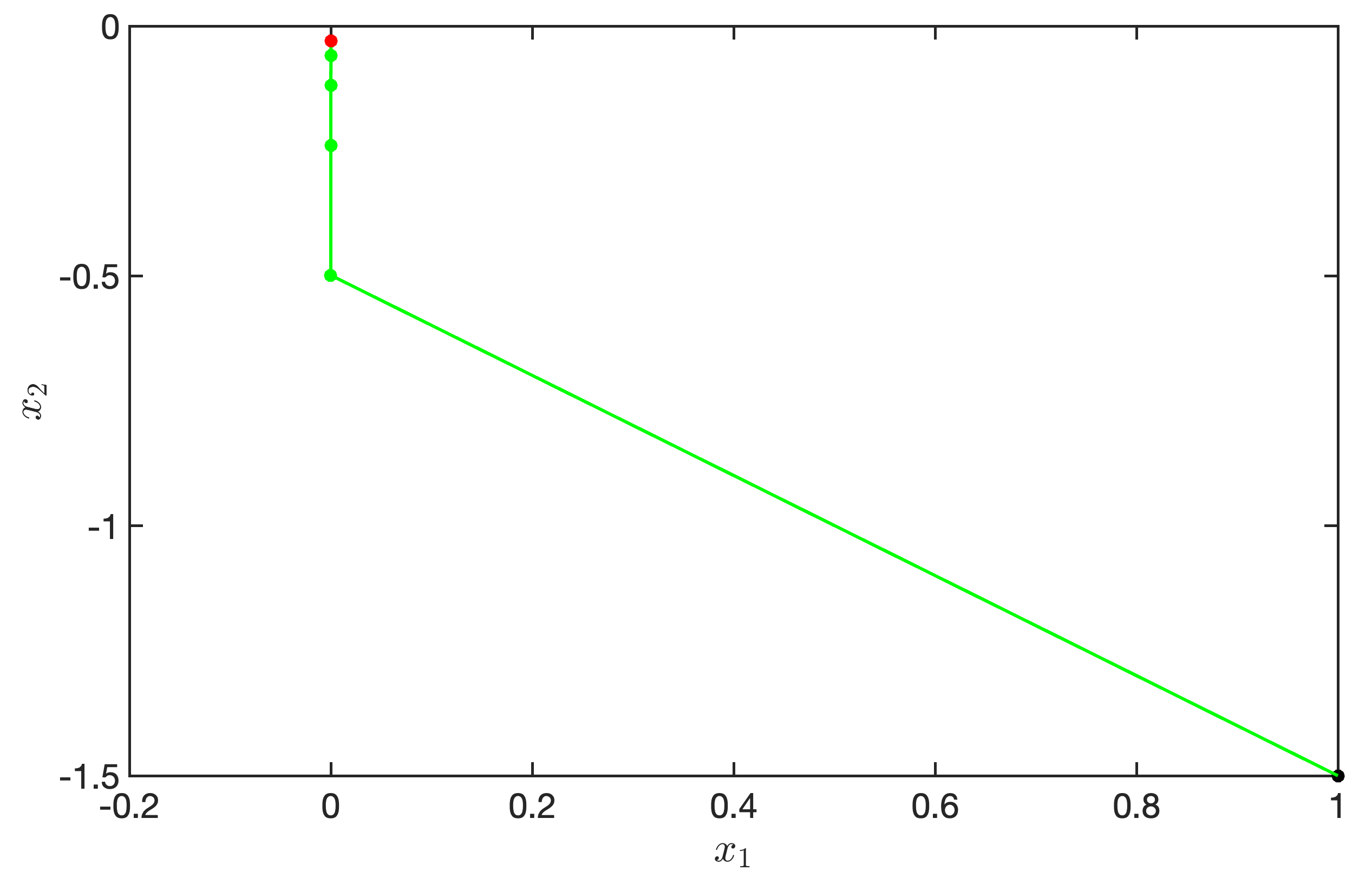}\label{figur3b}} \qquad
    \subfloat[~The value of $F$ at each iteration generated by Algorithm \ref{algo1} for three different randomly chosen initial points of Example \ref{example3}]{\includegraphics[width=0.48\textwidth]{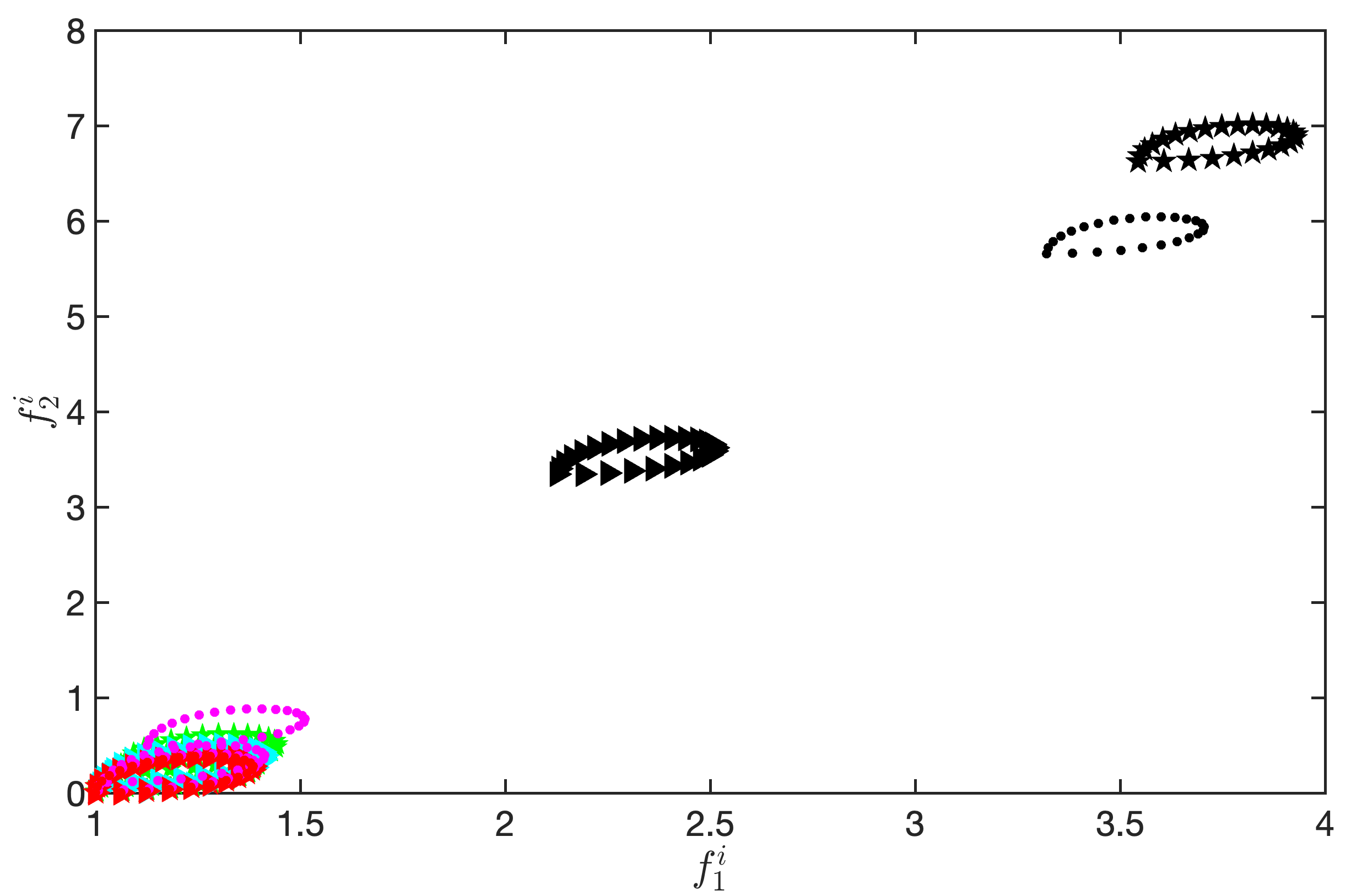}\label{figur3c}}\quad
    \subfloat[~The movement of subsequent $x_k$ generated by Algorithm \ref{algo1} for three different randomly chosen initial points of Example \ref{example3}]{\includegraphics[width=0.48\textwidth]{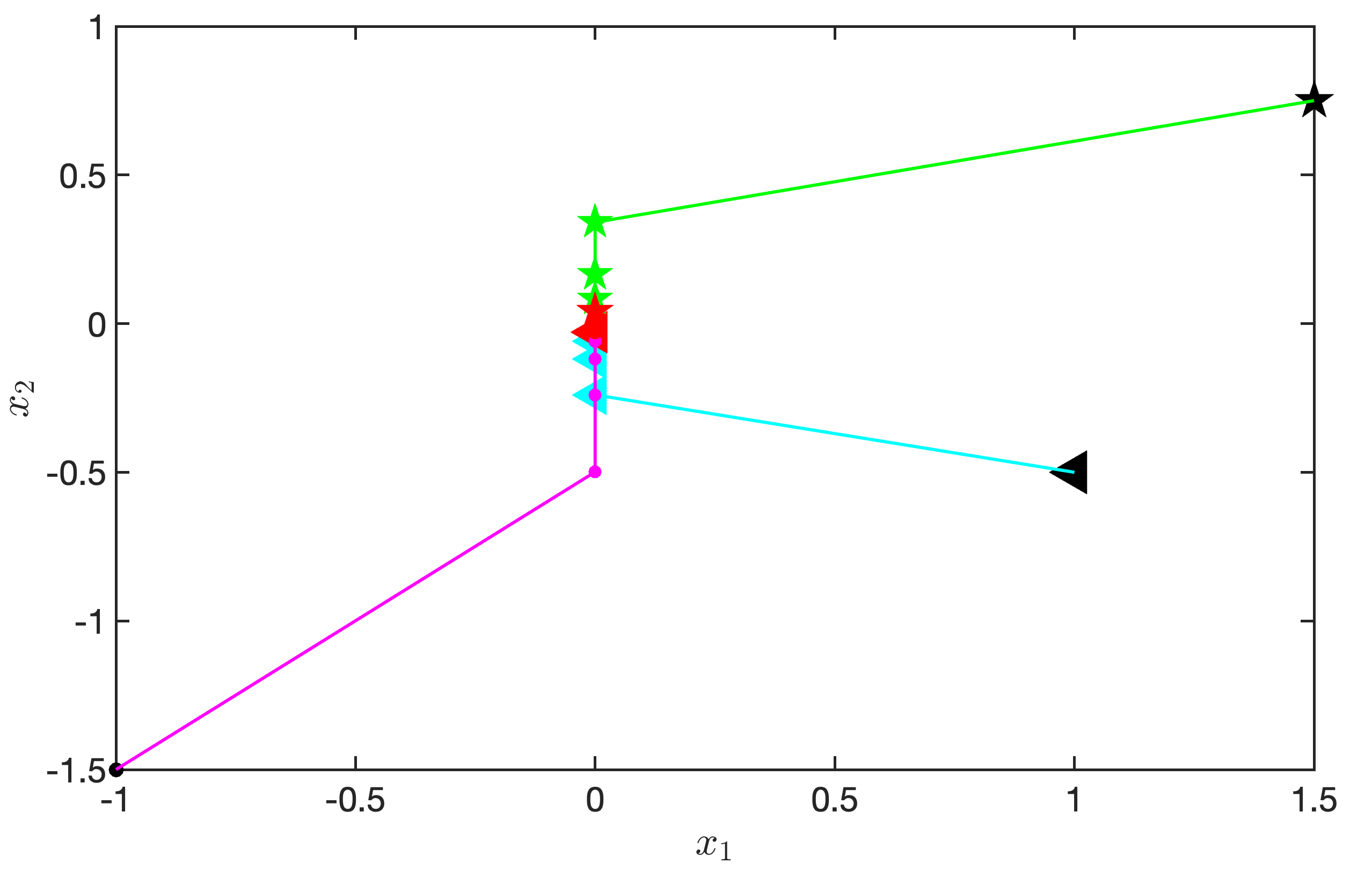}\label{figur3d}}
    \caption{Output of Algorithm \ref{algo1} for Example \ref{example3}}
    \label{figur3}
\end{figure}
In Figure \ref{figur3}, the iterates generated by Algorithm \ref{algo1} for different initial points taken from the set $[-5,5]\times[-5,5]$ are given. The sequence of iterates $\{x_k\}$ and the corresponding $\{F(x_k)\}$ generated by Algorithm \ref{algo1} for an initial point $x_0=(1.0000,-1.5000)^\top$ are given in Figure \ref{figur3b} and Figure \ref{figur3a}, respectively. Moreover, for the other three randomly selected initial points, the sequence of iterates $\{x_k\}$ and the corresponding $\{F(x_k)\}$ generated by Algorithm \ref{algo1} are shown in Figure \ref{figur3d} and Figure \ref{figur3c}, respectively. 

The performance of Algorithm \ref{algo1} for Example \ref{example3} is shown in Table \ref{table3}. Moreover, we have compared the results of the QNM with the SD method for set optimization as presented in Table \ref{table3}. The values in Table \ref{table3} show that the proposed method performs better than the existing SD method.

\begin{table}[ht]
\centering
\caption{Performance of Algorithm \ref{algo1} on Example \ref{example3}}
\centering 
\scalebox{0.68}{
\begin{tabular}{c c c c} 
\hline
Number of& Algorithm  &Iterations &CPU time\\
initial points& &(Min, Max, Mean, Median, Mode, SD)  &(Min, Max, Mean, Median, $\lceil\text{Mode}\rceil$, SD)   \\ 
\hline 
$100$ & QNM&($1,~6,~5.7400,~6,~6,~0.4845$) & ($1.2551,~42.8398,~23.8484,~29.5973,~11,~11.9814$) \\ 
    & SD &($1,~14,~13.9900,~14,14,0.1000$) & ($0.112,~4.2931,~3.2962,~3.2843,~3,~0.1097$) \\
\hline 
\end{tabular}}
\label{table3}
\end{table}

For the initial point $x_0=(1.0000,-1.5000)^\top$, the decreasing behavior in the values of vector-valued functions at each iteration has been exhibited in Table \ref{table_3c}.

\begin{table}[ht]
\centering
\caption{ {Output of Algorithm \ref{algo1} on Example \ref{example3} for the initial point $x_0=(1.0000,-1.5000)^\top$}}
\centering 
\scalebox{0.76}{
\begin{tabular}{c c c c c } 
\hline
$k$& $x_k^\top$ &$f^{5}(x_k)$ & $f^{15}(x_k)$ & $f^{25}(x_k)$     \\ 
\hline 
 0&  $(1.0000,-1.5000)$& $(3.3806, 7.5748)$& $(3.6992,7.6545)$ &$(3.3833, 7.3454)$  \\  
 1&  $( -0.0000,-0.4987)$& $(1.1868, 0.7308)$& $(1.5054, 0.8105)$ &$(1.1895, 0.5014)$   \\
2&  $(0.0000,-0.2392)$& $(1.0886, 0.3477)$& $(1.4072, 0.4275)$ &$(1.0913, 0.1184)$  \\
3&  $(0.0000,-0.1185)$& $(1.0669,0.2614)$& $(1.3855,0.3412)$ &$(1.0696,0.0320)$  \\
4&  $(0.0001,-0.0591)$& $(1.0617,0.2404)$& $(1.3802,0.3201)$ &$(1.0643,0.0110)$  \\
5&  $(0.0001,-0.0295)$& $(1.0603,0.2352)$& $(1.3789, 0.3149)$ &$(1.0630, 0.0058)$  \\
\hline 
\end{tabular}}
\label{table_3c}
\end{table}
\end{example}


\begin{example}\label{example4}
Consider the function $F:\mathbb{R}^2\rightrightarrows \mathbb{R}^3$ defined as
$$F(x)=\{f^1(x),f^2(x),\ldots,f^{10}(x)\},$$
where for each $i\in[10],f^i:\mathbb{R}\to\mathbb{R}^2$ is given by 
\[
f^i(x)=\begin{pmatrix}
e^{x_1}+\sin\left(\tfrac{2\pi(i-1)}{20}\right)+e^{x_2}\\
2e^{x_1}+\cos\left(\tfrac{2\pi(i-1)}{20}\right)+2e^{x_2}\\
x_1^2+\tfrac{(i-1)}{20}+x_2^2
\end{pmatrix}.
\]
\begin{figure}
    \centering
    \subfloat[~The value of $F$ at each iteration generated by Algorithm \ref{algo1} for initial point $x_0=(1.0000,-1.5000)^\top$ of Example \ref{example4}]{\includegraphics[width=0.48\textwidth]{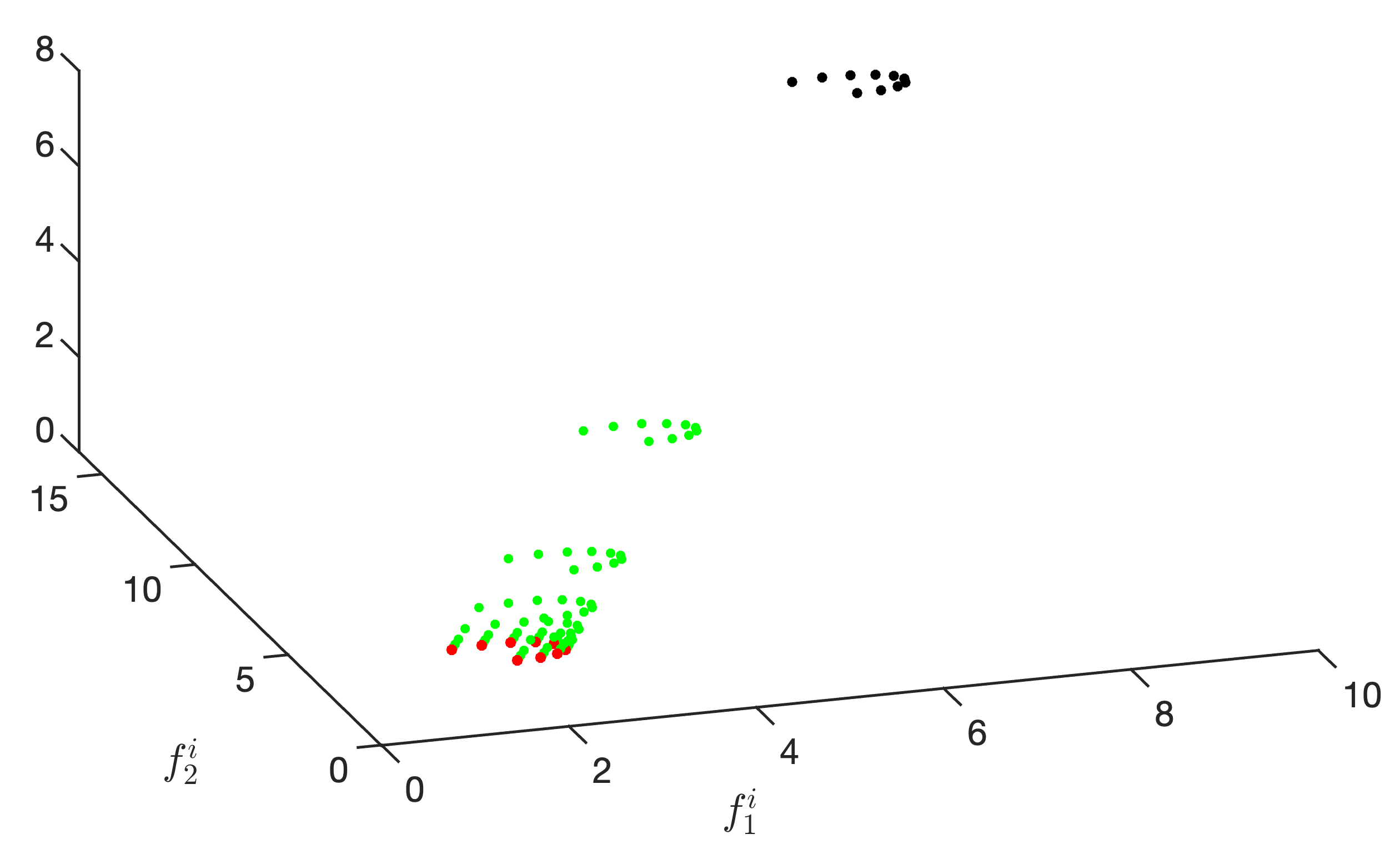}\label{figur4a}} \quad
    \subfloat[~The movement of subsequent $x_k$ generated by Algorithm \ref{algo1} for the initial point $x_0=(1.0000,-1.5000)^\top$ of Example \ref{example4}]{\includegraphics[width=0.48\textwidth]{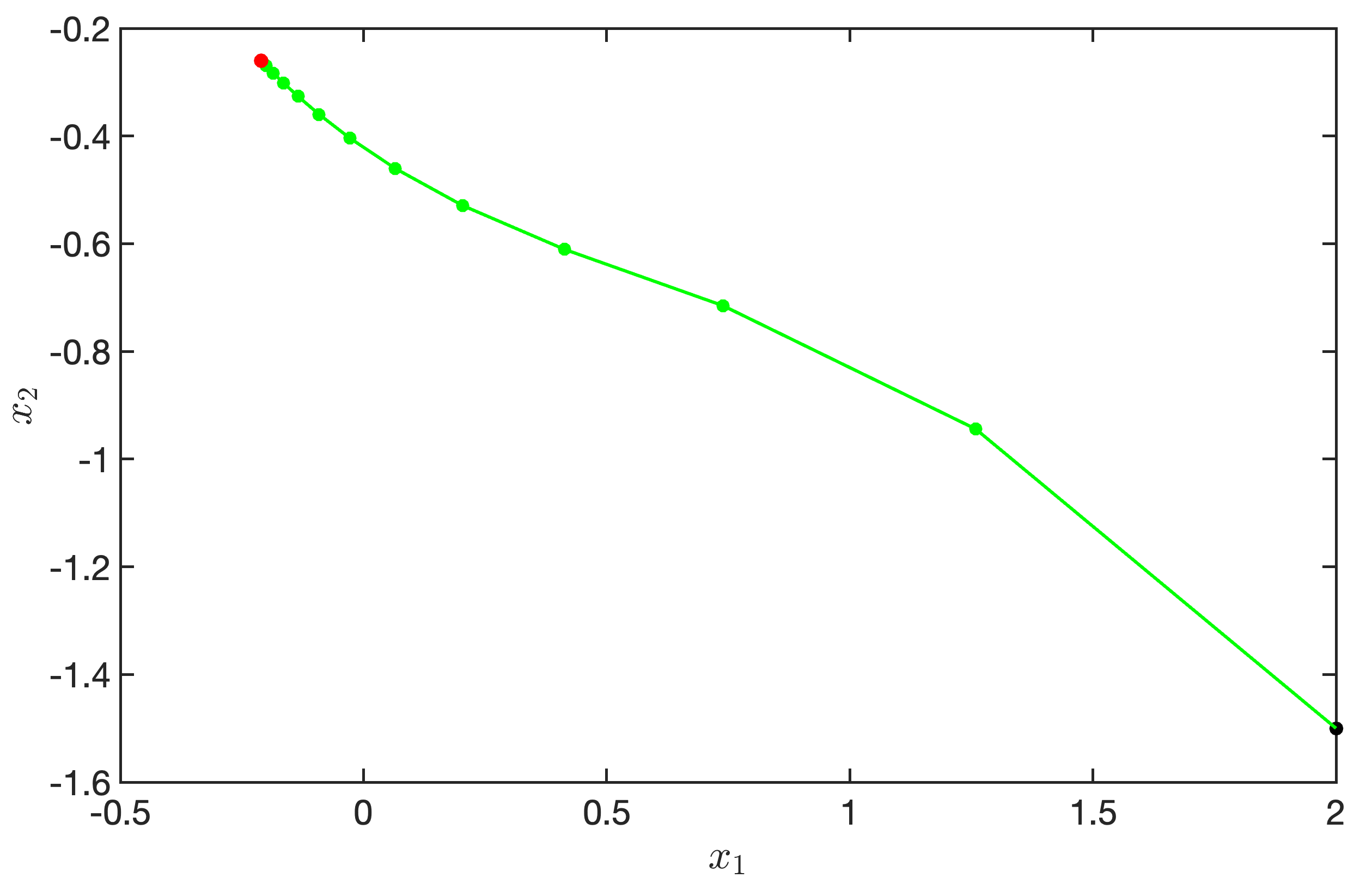}\label{figur4b}} \qquad
    \subfloat[~The value of $F$ at each iteration generated by Algorithm \ref{algo1} for three different randomly chosen initial points of Example \ref{example4}]{\includegraphics[width=0.48\textwidth]{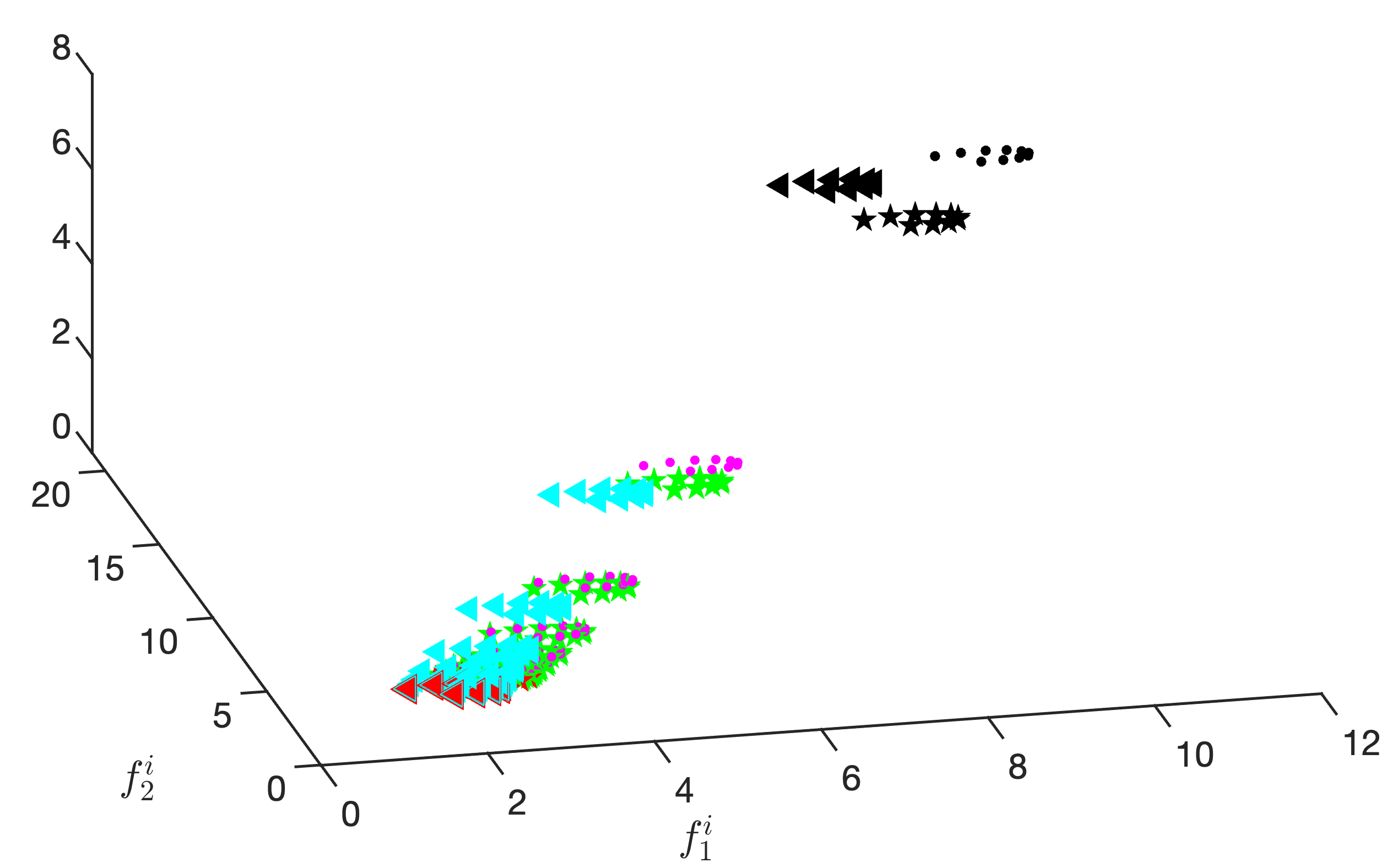}\label{figur4c}}\quad
    \subfloat[~The movement of subsequent $x_k$  generated by Algorithm \ref{algo1} for three different randomly chosen initial points of Example \ref{example4}]{\includegraphics[width=0.48\textwidth]{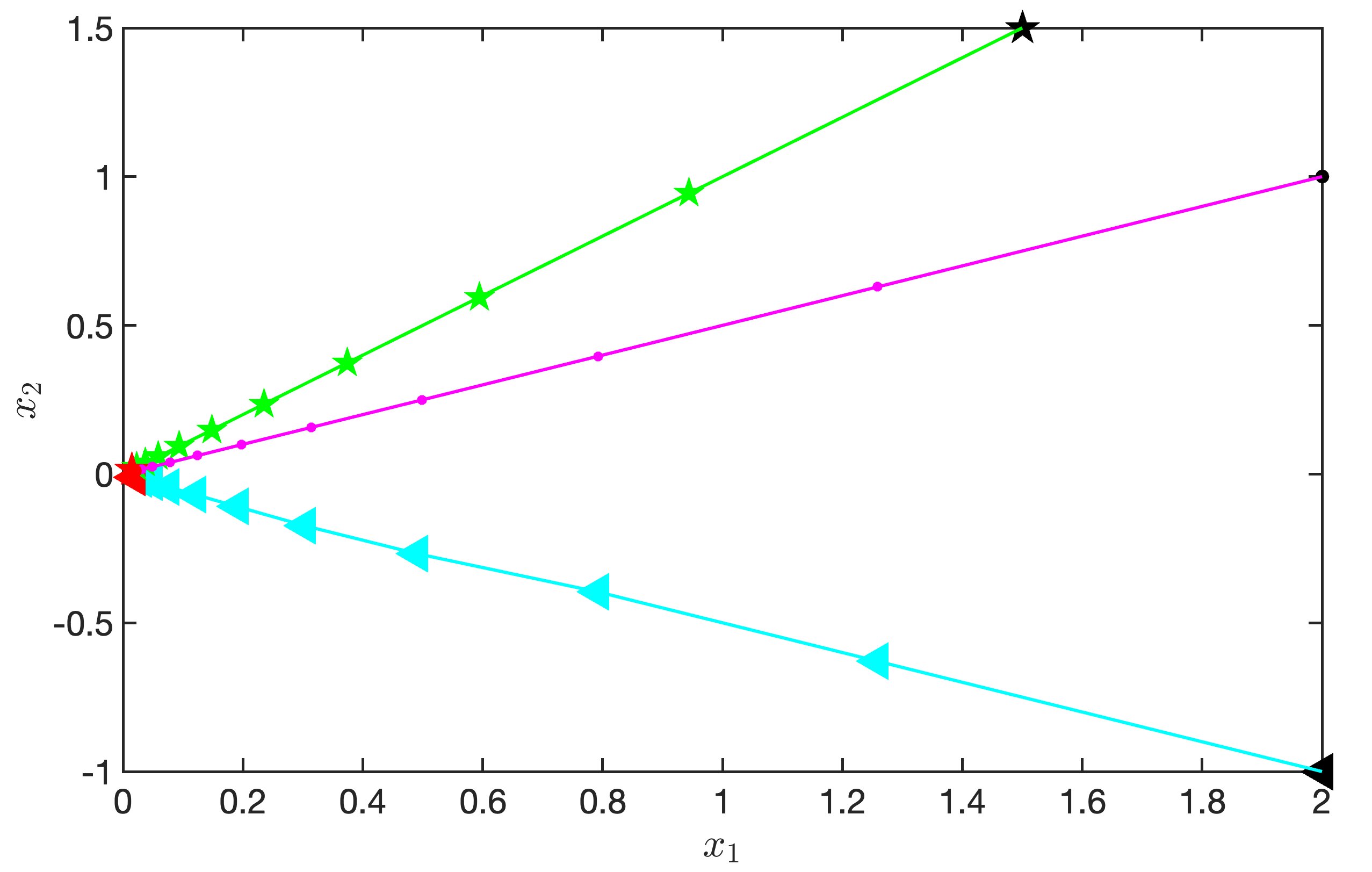}\label{figur4d}}
    \caption{Output of Algorithm \ref{algo1} of Example \ref{example4}}
    \label{figur4}
\end{figure}
In Figure \ref{figur4}, the iterates generated by Algorithm \ref{algo1} for different initial points taken from the set $[-4,3]\times[-4,3]$ are given. The sequence of iterates $\{x_k\}$ and the corresponding $\{F(x_k)\}$ generated by Algorithm \ref{algo1} for a selected initial point $x_0=(1.0000,-1.5000)^\top$ are given in Figure \ref{figur4b} and Figure \ref{figur4a}, respectively. Moreover, for three randomly selected initial points, the sequence of iterates $\{x_k\}$ and the corresponding $\{F(x_k)\}$ generated by Algorithm \ref{algo1} are shown in Figure \ref{figur4d} and Figure \ref{figur4c}, respectively.

The performance of Algorithm \ref{algo1} for Example \ref{example4} is shown in Table \ref{table4}. Moreover, we have compared the results of the QNM with the SD method for set optimization as presented in Table \ref{table4}. The values in Table \ref{table4} show that the proposed method performs better than the existing SD method.

\begin{table}[ht]
\centering
\caption{Performance of Algorithm \ref{algo1} on Example \ref{example4}}
\centering 
\scalebox{0.68}{
\begin{tabular}{c c c c} 
\hline 
Number of& Algorithm  &Iterations &CPU time\\
initial points& &(Min, Max, Mean, Median, Mode, SD)  &(Min, Max, Mean, Median, $\lceil\text{Mode}\rceil$, SD)   \\ 
\hline 
$100$ & QNM&($1,9,~8.0300,~8,~8,~0.2642$) & ($1.1032,~49.9087,~27.6729,~22.9076,~17,~4.4759$) \\ 
    & SD &($1,~10,~9.6300,~10,~10,~0.4852$) & ($33.9218,31.0397
,32.0548,28.6446,28,1.6261$) \\
\hline 
\end{tabular}}
\label{table4}
\end{table}
\end{example}

In the next two examples (Example \ref{example5} and Example \ref{example6}), we consider a cone different from $\mathbb{R}^m_+$ and observe the performance of Algorithm \ref{algo1}. The Example \ref{example5} is a slight modification of Test instance 5.1 discussed in \cite{bouza2021steepest} with respect to a cone $\mathbb{R}^m_+$.

\begin{example}\label{example5}
Consider the function $F:\mathbb{R}\rightrightarrows \mathbb{R}^2$ defined as 
$$F(x)=\{f^1(x),f^2(x),f^3(x),f^4(x)\},$$
where for each $i\in[4],f^i:\mathbb{R}\to\mathbb{R}^2$ is given by 
\[
f^i(x)=\begin{pmatrix}
2x^2+e^x+\frac{(i-3)}{2}\\
\tfrac{x}{2}\cos(x)+\frac{(-i+3)}{2}\sin^2 x
\end{pmatrix}.
\]
The cone is $K$ given by $K=\{(z_1,z_2)^\top\in\mathbb{R}^2:6z_1-2z_2\geq 0,-7z_1+10z_2\geq 0\}$.

\begin{figure}[ht]
\centering
\mbox{\subfloat[The value of $F$ at each iteration generated by Algorithm \ref{algo1} for the initial point $x_0=4.3000$ of Example \ref{example5}]{ \includegraphics[width=0.48\textwidth]{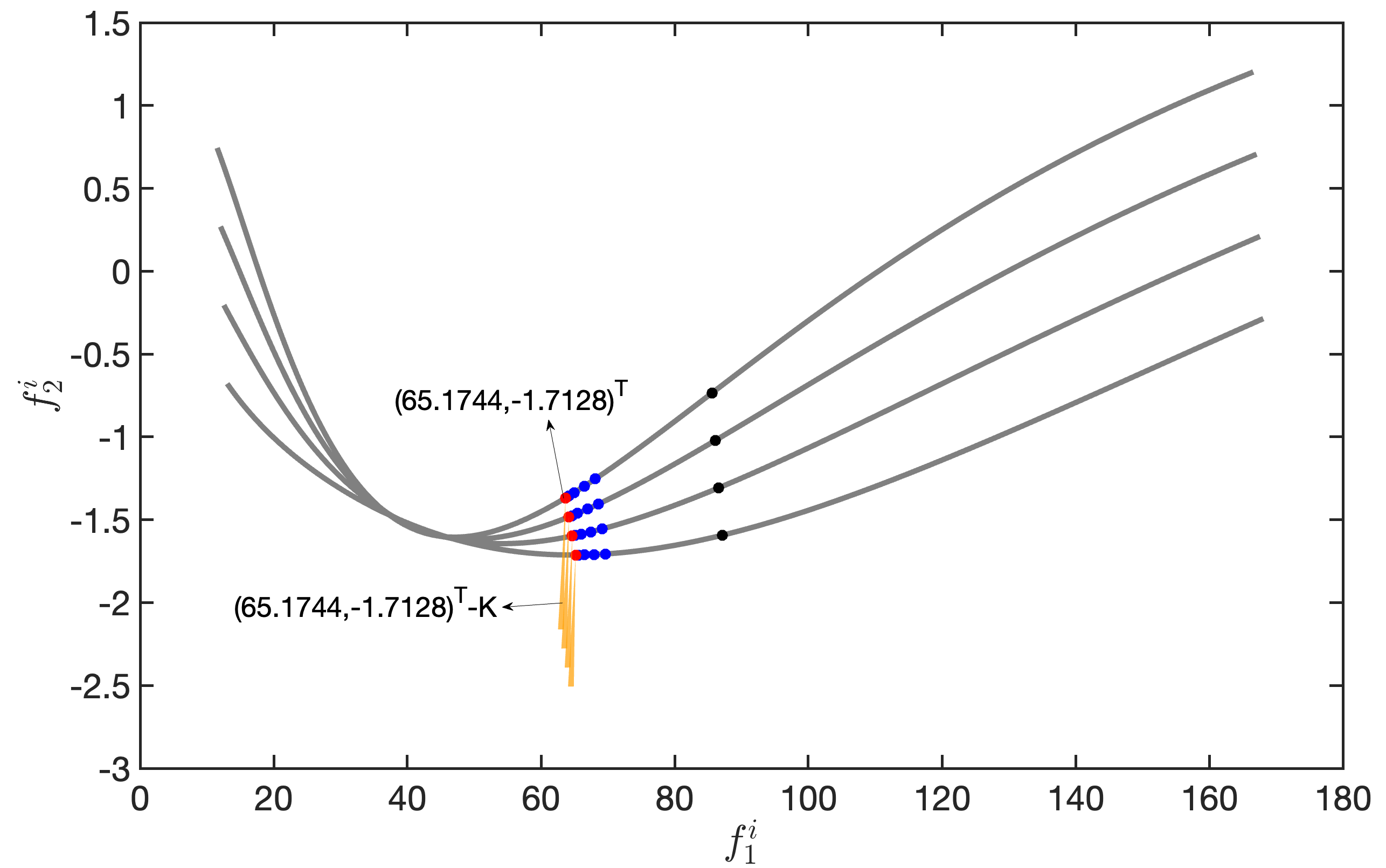}\label{figure5a}}\quad
\subfloat[The value of $F$ for three different initial points at each iteration generated by Algorithm \ref{algo1} on Example \ref{example5}]{\includegraphics[width=0.48\textwidth]{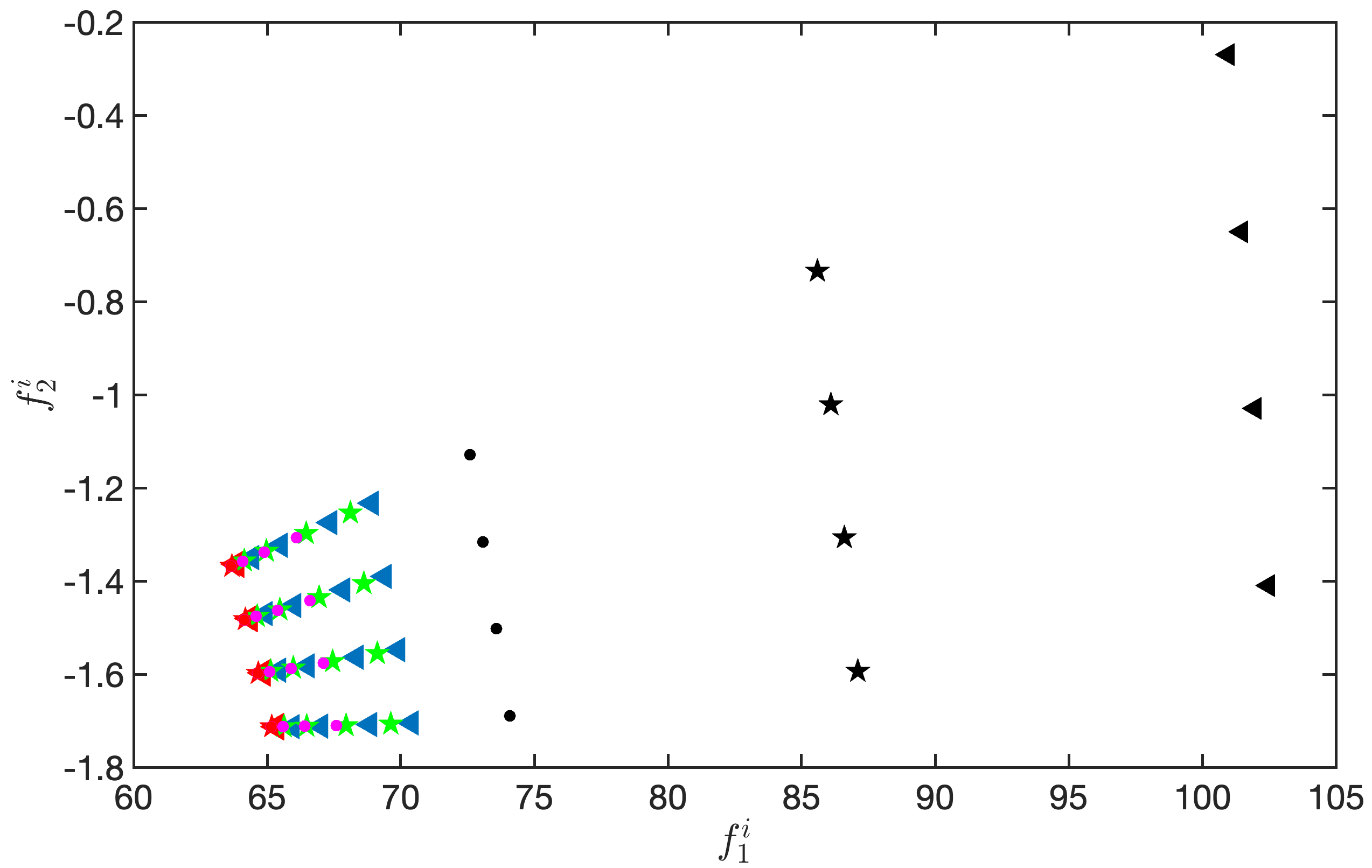} \label{figure5b}}\quad}
\caption{Obtained output of Algorithm \ref{algo1} for Example \ref{example5}}
\label{figure5} 
\end{figure}

The output of Algorithm \ref{algo1} for different initial points of Example \ref{example5} are depicted in Figure \ref{figure5}. The discretized segments represent the objective values that transverse a curve within the interval $[2.3350,4.4010]$. In Figure \ref{figure5a}, we test our algorithm for an initial point $x_0=4.0000$. It can be seen that the points depicted with red color are optimal points of $F$ as the set $(65.1744,-1.7128)^\top-K$ does not contain any element of $F(x)$ other than $(65.1744,-1.7128)$ for all $x\in[2.3350,4.4010]$. In Figure \ref{figure5b}, we test our algorithm for three initial points and depict the output.

The performance of Algorithm \ref{algo1} for Example \ref{example5} is shown in Table \ref{table5}. Moreover, we have compared the results of the QNM with the SD method for set optimization as presented in Table \ref{table5}. The values in Table \ref{table5} show that the proposed method performs better than the existing SD method.

\begin{table}[ht]
\centering
\caption{Performance of Algorithm \ref{algo1} on Example \ref{example5}}
\centering 
\scalebox{0.62}{
\begin{tabular}{c c c c} 
\hline 
Number of& Algorithm  &Iterations &CPU time\\
initial points& &(Min, Max, Mean, Median, Mode, SD)  &(Min, Max, Mean, Median, $\lceil\text{Mode}\rceil$, SD)   \\ 
\hline 
$100$ & QNM&($1,~6,~5.8900,~6,~6,~0.8275$) & ($0.1013,~1.2640e+03,~1.0633e+03,~1.0639e+03,~24,~17.6970$) \\ 
    & SD &($1,~8,~8.8900,~8,~8,~0.8275$) & ($1.2113,~8.7830e+03,~8.6875e+03,~8.6873e+03,~36,~55.8729$)\\
\hline 
\end{tabular}}
\label{table5}
\end{table}

Further, for the initial point $x_0=4.0000$, the decreasing behavior in the values of vector-valued functions at each iteration has been exhibited in Table \ref{table_5c}.

\begin{table}[ht]
\centering
\caption{ {Output of Algorithm \ref{algo1} on Example \ref{example5} for the initial point $x_0=4.0000$}}
\centering 
\scalebox{0.73}{
\begin{tabular}{c c c c c c} 
\hline 
$k$& $x_k$ &$f^{1}(x_k)$ & $f^{2}(x_k)$ & $f^{3}(x_k)$   &$f^{4}(x_k)$  \\ 
\hline 
 0&  $4.0000$& $(85.5982,-0.7345)$& $(86.0982,-1.0209)$ &$(86.5982,-1.3073)$ & $(87.0982,-1.5937)$ \\ 
 1&  $3.7232$& $(68.1191,-1.2538)$& $(68.6191,-1.4047)$ &$(69.1191,-1.5555)$ & $(69.6191,-1.7064)$\\  
 2&  $3.6932$& $(66.4532,-1.2981)$& $(66.9532,-1.4354)$ &$(67.4532,-1.5727)$  & $(67.9532,-1.7100)$\\
3&  $3.6660$& $(64.9752,-1.3360)$& $(65.4752,-1.4613)$ &$(65.9752,-1.5867)$ & $(66.4752,-1.7120)$ \\
4&  $3.6503$& $(64.1331,-1.3569)$& $(64.6331,-1.4755)$ &$(65.1331,-1.5941)$ & $(65.6331,-1.7126)$ \\
5&  $3.6416$& $(63.6744,-1.3681)$& $(64.1744,-1.4830)$ &$(64.6744, -1.5979)$ & $(65.1744,-1.7128)$ \\
6&  $3.6371$& $(63.4378,-1.3738)$& $(63.9378,-1.4868)$ &$(64.4378, -1.5998)$ & $(64.9378,-1.7129)$ \\

\hline 
\end{tabular}}
\label{table_5c}
\end{table}
\end{example}

\begin{example}\label{example6}
 Consider the set-valued function $F:\mathbb{R}^2\rightrightarrows \mathbb{R}^2$ defined as
$$F(x)=\{f^1(x),f^2(x),\ldots,f^{100}(x)\},$$
where for each  $i\in[100],$ the function $f^i:\mathbb{R}^2\to\mathbb{R}^2$ is given as
\[
f^i(x)=\begin{pmatrix}
x_1^2+\sin(x_1)+x_1^2\cos(x_2)+0.25\cos\left(\tfrac{2\pi(i-1)}{100}\right)\sin^2\left(\tfrac{2\pi(i-1)}{100}\right)+e^{x_1+x_2}+x_2^2\\
2x_1^2+x_2^2\cos(x_1)+0.25\cos^2\left(\tfrac{2\pi(i-1)}{100}\right)\sin\left(\tfrac{2\pi(i-1)}{100}\right)+\cos(x_2)+e^{x_1+x_2}+2x_2^2\\
\end{pmatrix}.
\]  
The cone $K$ is $K$ given by $K=\{(z_1,z_2)^\top\in\mathbb{R}^2:2z_1-6z_2\geq 0, -6z_1+7z_2\geq 0\}$.
\begin{figure}[ht]
\centering
\subfloat[The value of $F$ at each iteration generated by Algorithm \ref{algo1} for the initial point $x_0=(0.5000,-0.5000)^\top$ of Example \ref{example6}]{ \includegraphics[width=0.46\textwidth]{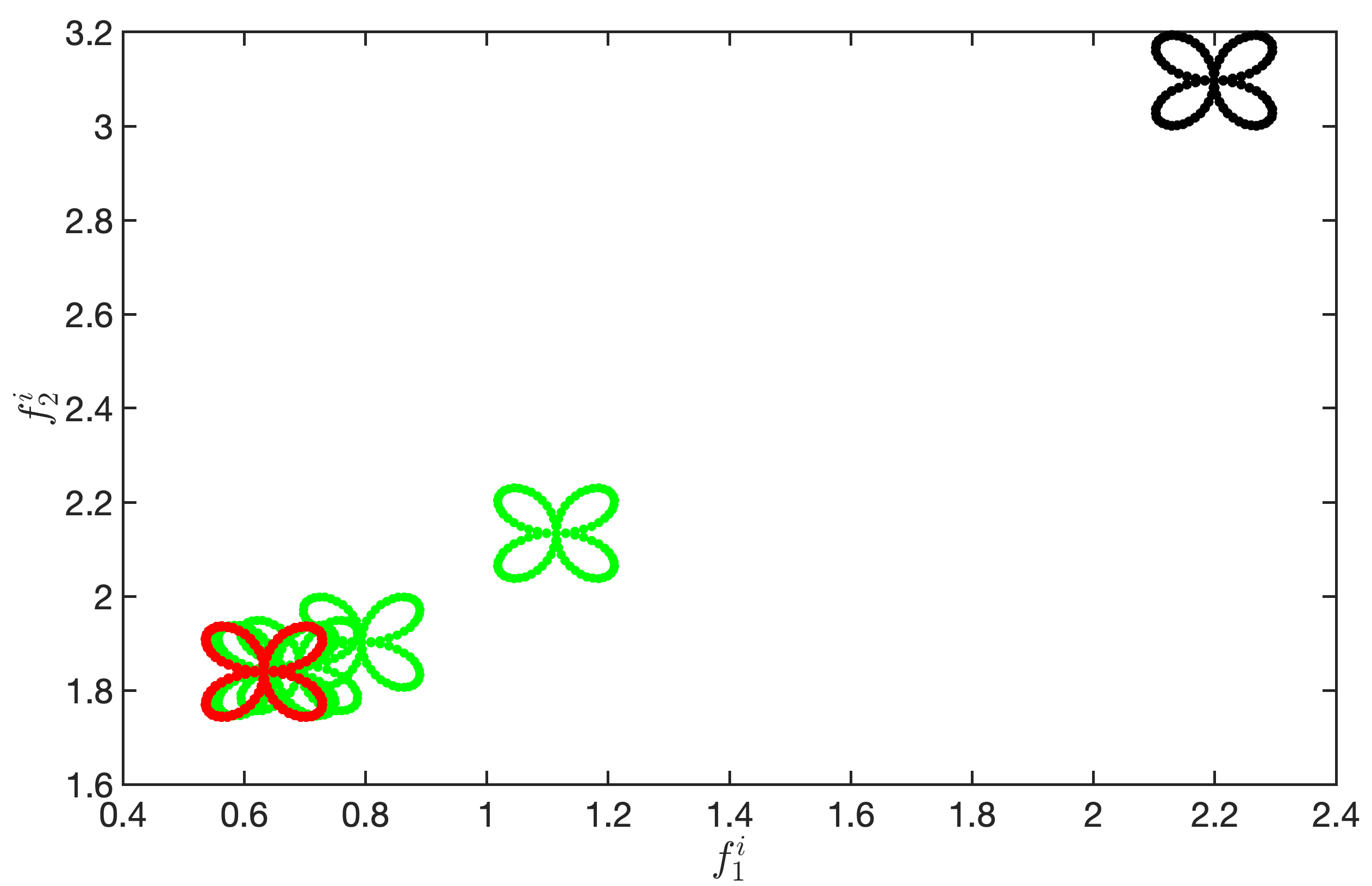}\label{figure6a}}\quad
\subfloat[The movement of subsequent $x_k$ generated by Algorithm \ref{algo1} for the initial 
point $x_0=(0.5000,-0.5000)^\top$ of Example \ref{example6}]{\includegraphics[width=0.46\textwidth]{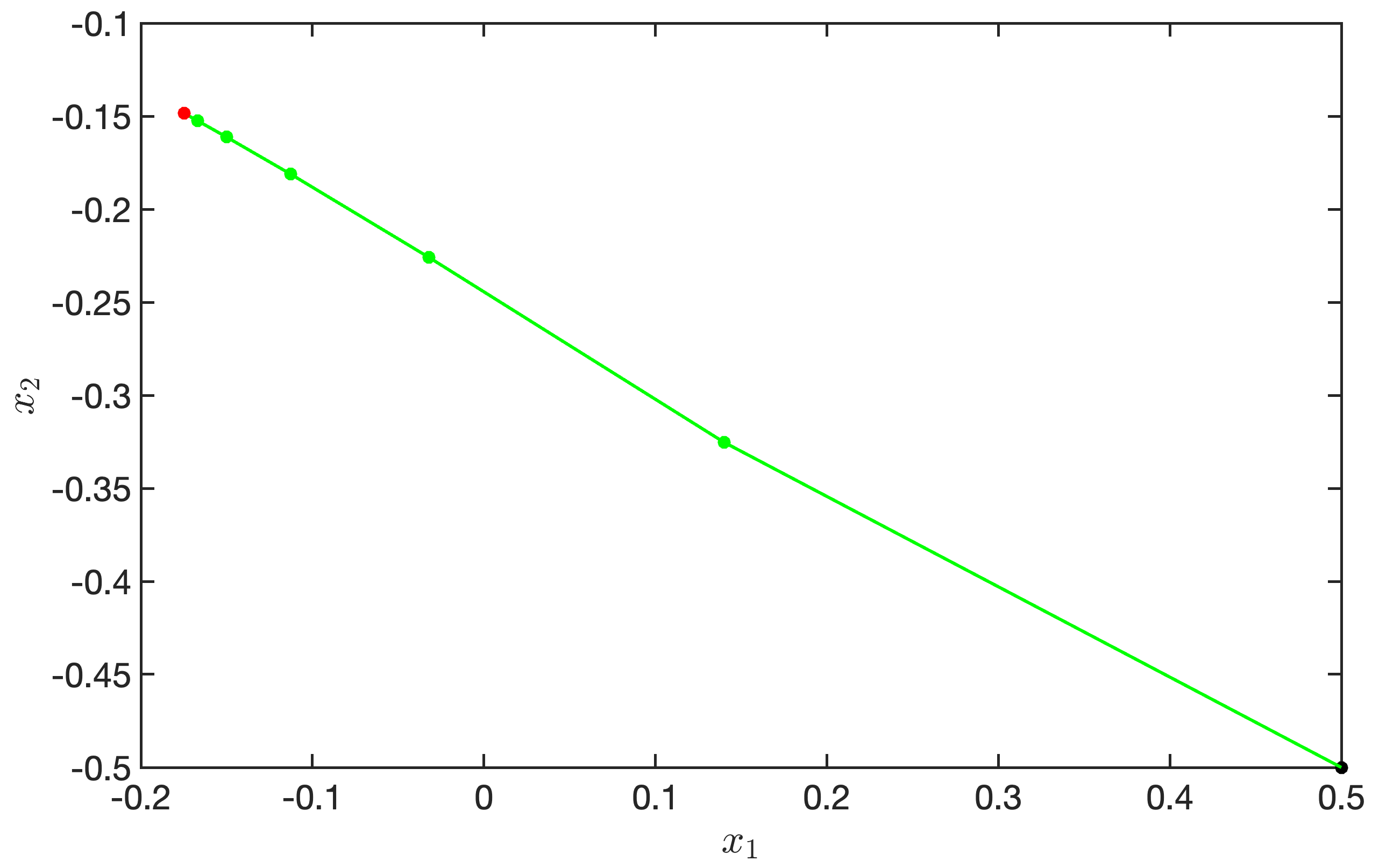} \label{figure6b}}\qquad
\subfloat[The value of $F$ for three different initial points at each iteration generated by Algorithm \ref{algo1} of Example \ref{example6}]{\includegraphics[width=0.46\textwidth]{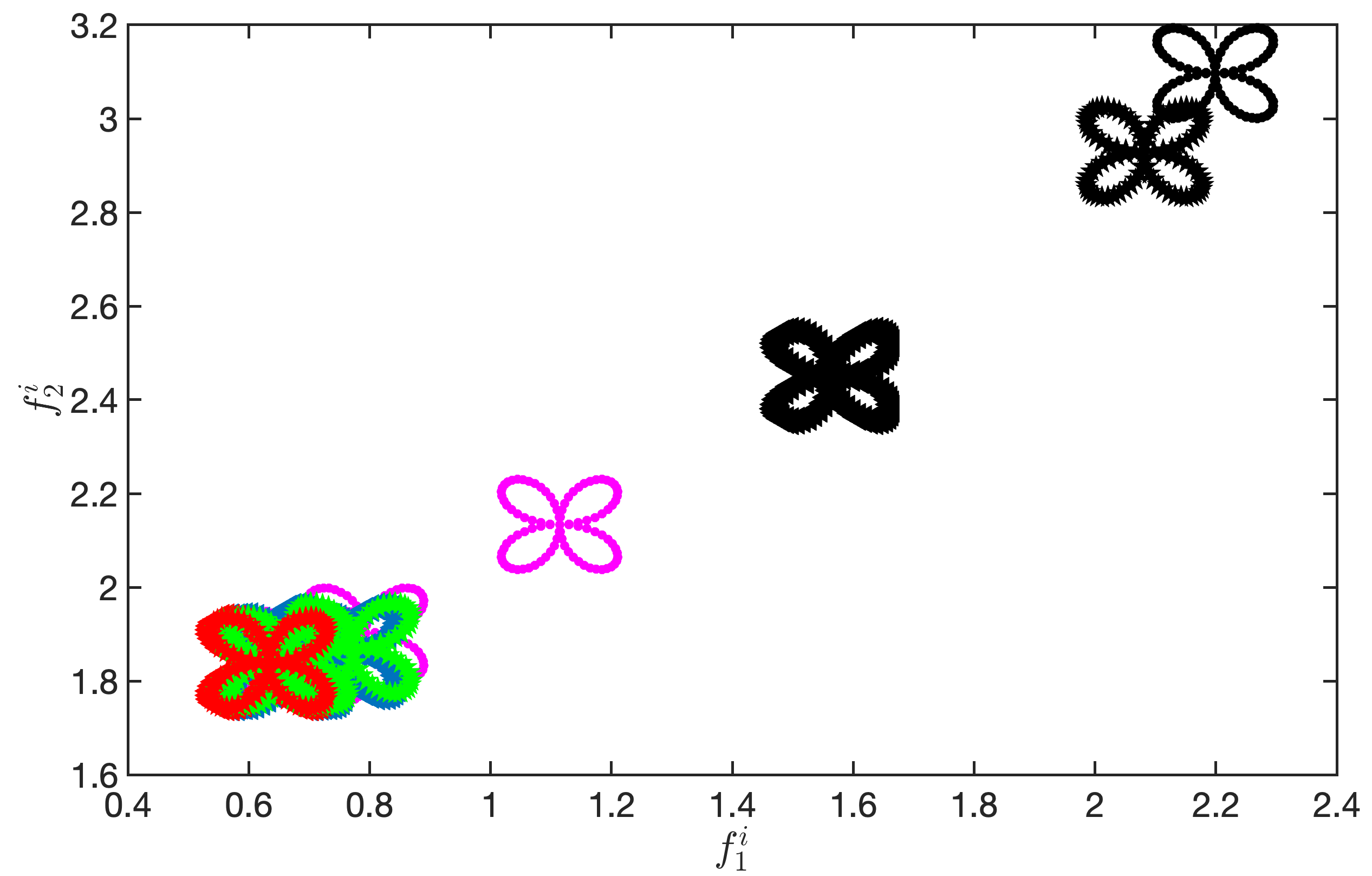} \label{figure6c}}\qquad
\subfloat[The movement of subsequent $x_k$ generated by Algorithm \ref{algo1} for the three different randomly chosen initial 
points of Example \ref{example6}]{\includegraphics[width=0.46\textwidth]{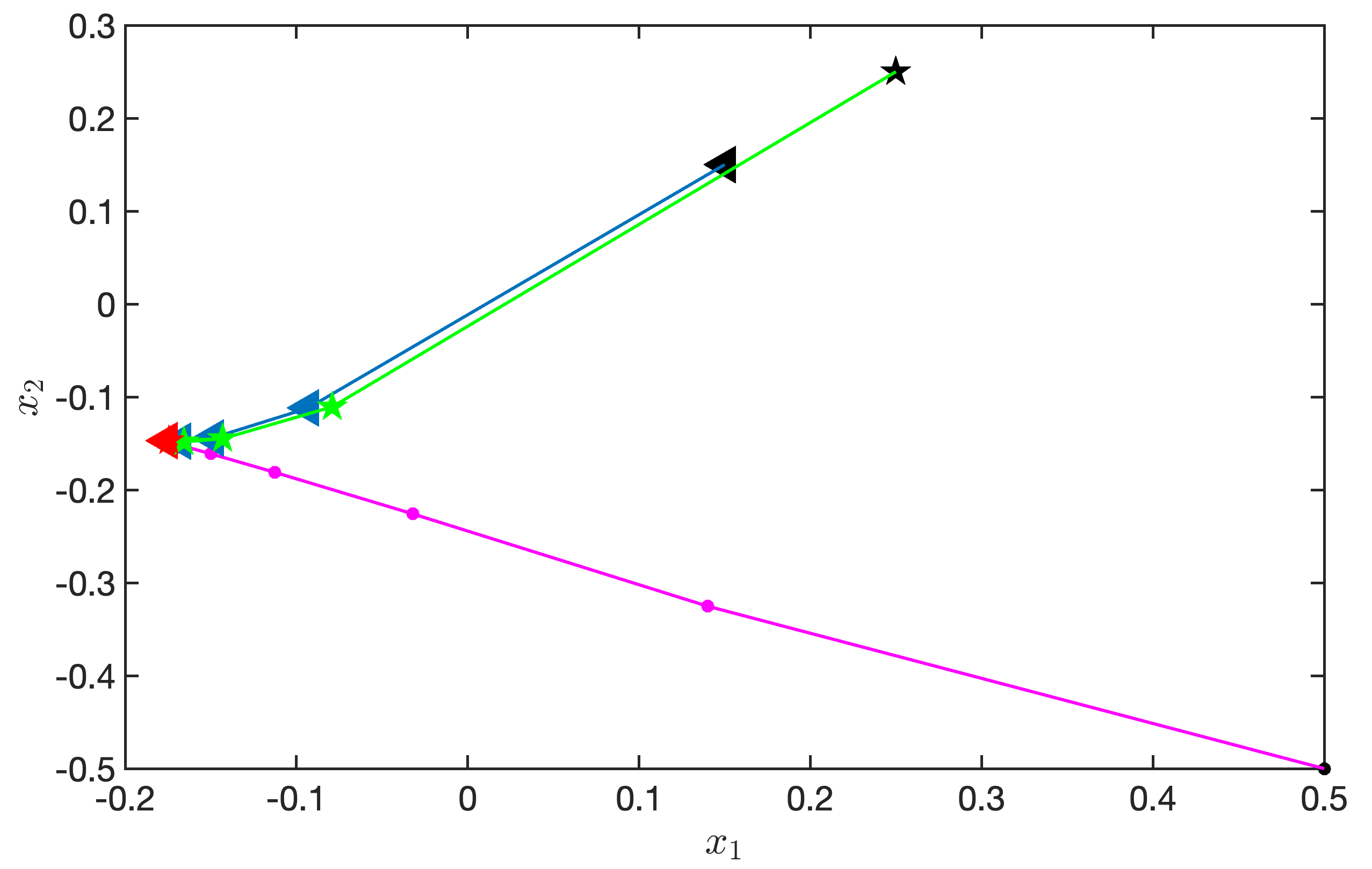} \label{figure6d}}
\caption{Obtained output of Algorithm \ref{algo1} for Example \ref{example6}}
\label{figure6} 
\end{figure}

The output of Algorithm \ref{algo1} for different initial points of Example \ref{example6} are depicted in Figure \ref{figure5}. The discretized segments represent the objective values that transverse a curve within the interval $[-\pi,\pi]\times[-\pi,\pi]$. Figure \ref{figure6a} exhibits the sequence $\{F(x_k)\}$ generated by Algorithm \ref{algo1} for a chosen starting point, and Figure \ref{figure6b} tests Algorithm \ref{algo1} for three different randomly chosen starting points. 

The performance of Algorithm \ref{algo1} for Example \ref{example6} is shown in Table \ref{table6}. Moreover, we have compared the results of the QNM with the SD method for set optimization as presented in Table \ref{table6}. The values in Table \ref{table6} show that the proposed method performs better than the existing SD method. 

\begin{table}[ht]
\centering
\caption{Performance of Algorithm \ref{algo1} on Example \ref{example6}}
\centering 
\scalebox{0.68}{
\begin{tabular}{c c c c} 
\hline 
Number of& Algorithm  &Iterations &CPU time\\
initial points& &(Min, Max, Mean, Median, Mode, SD)  &(Min, Max, Mean, Median, $\lceil\text{Mode}\rceil$, SD)   \\ 
\hline 
$100$ & QNM&($1,~33,~22.7200,~23,~29,~6.4418$) & ($3.2551,~67.4360,~45.8573,~46.2144,~3,~12.9461$) \\ 
    & SD &($1,~33,~22.4500,~22.5000,~25,~6.3776$) & ($1.2551,~63.9492,~43.9482,~44.1527,~3,~12.1495$) \\
\hline 
\end{tabular}}
\label{table6}
\end{table}

\end{example}


In the next example, we discuss the robust counterpart of a vector-valued facility location problem under uncertainty \cite{ide2014relationship}. A complete discussion on this problem is given in \cite{bouza2021steepest}.
\begin{example}\label{example7} 
Consider the function $F:\mathbb{R}^2\rightrightarrows \mathbb{R}^3$ defined as 
$$F(x)=\{f^1(x),f^2(x),\ldots,f^{100}(x)\},$$
where for each $i\in[100],f^i:\mathbb{R}^2\to\mathbb{R}^3$ is given as
\[
f^i(x)=\tfrac{1}{2}
\begin{pmatrix}
\lVert x-l_1-u_i\rVert^2\\
\lVert x-l_2-u_i\rVert^2\\
\lVert x-l_3-u_i\rVert^2
\end{pmatrix},
\]  
where $l_1=\begin{pmatrix}
0\\
0
\end{pmatrix},~l_2=\begin{pmatrix}
8\\
0
\end{pmatrix} ~\text{ and }~l_3=\begin{pmatrix}
0\\
8
\end{pmatrix}$.
We consider a uniform partition set of $10$ points of the interval $[-1,1]$ given by 
$$
\mathcal{U}=\left\{-1,-1+\tfrac{1}{s},-1+\tfrac{2}{s},\ldots,-1+\tfrac{2(s-1)}{s},1 \right\}\text{ with }s=4.5.
$$
The set $\{u_i=(u_{1i},u_{2i})^\top:i\in[100]\}$ is an enumeration of the set $\mathcal{U}\times \mathcal{U}$.\\

We generate the total of $100$ initial points in the square $[-50,50]\times[-50,50]$ as shown in Figure \ref{figure7}. The set $(l_1+u_i)\cup(l_2+u_i)\cup(l_3+u_i)$ is represented by grey points and the locations of $l_1,l_2,l_3$ are depicted in blue color. The values of 
$F(x_k)$ generated by Algorithm \ref{algo1} for three different randomly chosen initial points are given with cyan, magenta, and green colors as shown in Fig. \ref{figure7}.\\

\begin{figure}[ht]
\centering
\mbox{\subfloat[The value of $F$ in argument space at each iteration generated by Algorithm \ref{algo1} for three randomly chosen initial points of Example \ref{example7}]{ \includegraphics[width=0.6\textwidth]{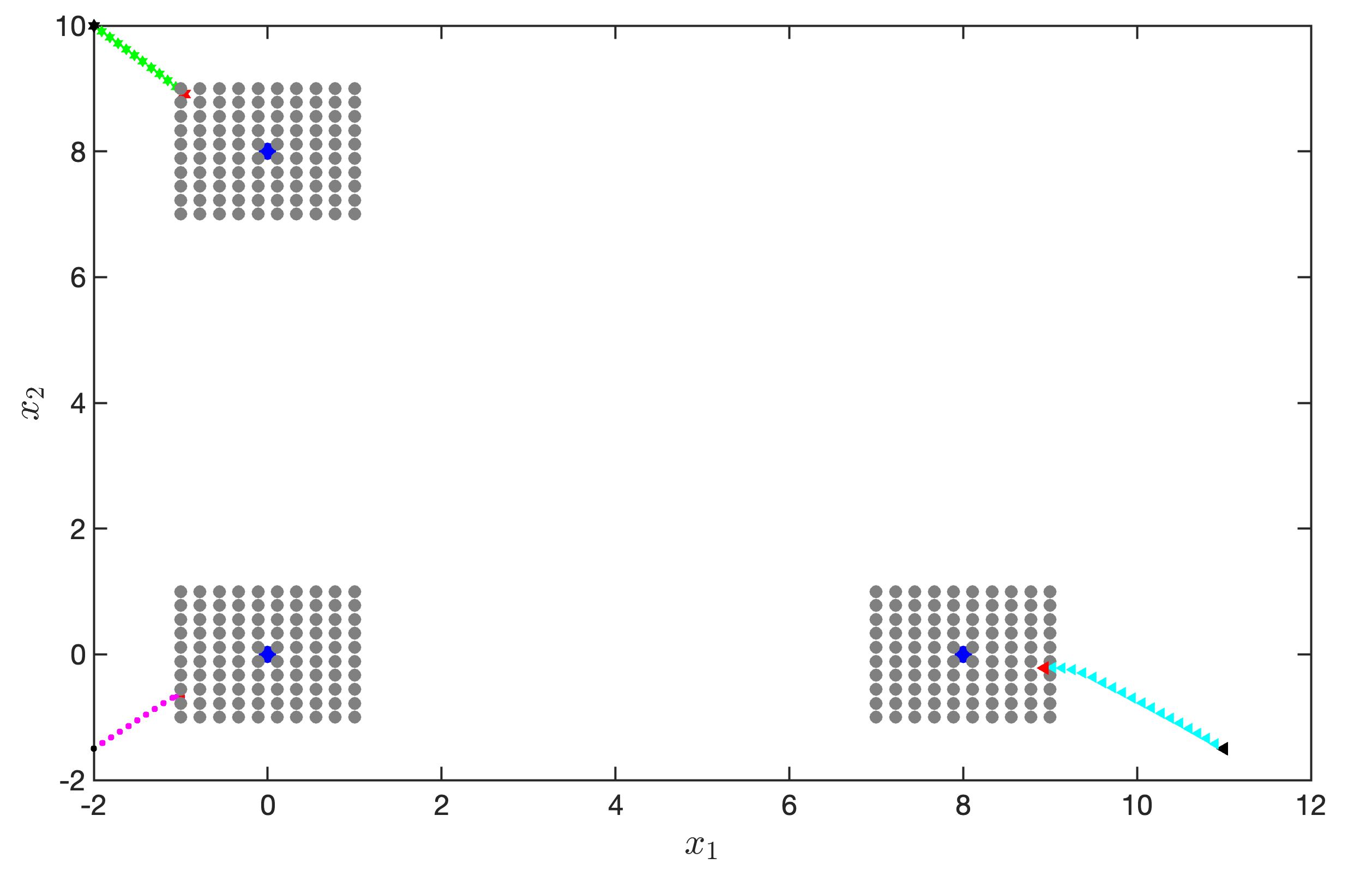}}\quad}
\caption{Obtained output of Algorithm \ref{algo1} on Example \ref{example7}}
\label{figure7} 
\end{figure}

The performance of Algorithm \ref{algo1} on Example \ref{example7} is shown in Table \ref{table7a}. A comparison of the results of QNM with the existing SD method is presented in Table \ref{table7a}.  The values in Table \ref{table7a} show that the proposed method performs better than the existing SD method.

\begin{table}[ht]
\centering
\caption{Performance of Algorithm \ref{algo1} on Example \ref{example7}}
\centering 
\scalebox{0.73}{
\begin{tabular}{c c c c} 
\hline 
Number of &Algorithm &Iterations &CPU time\\
initial points& &(Min, Max, Mean, Median, Mode, SD)  &(Min, Max, Mean, Median, $\lceil\text{Mode}\rceil$, SD)   \\ 
\hline
$100$ &QNM &($1,~24,~11.4503,~6,5123,~1.1421,~6.9534$) & ($1.2813,~19.2874,~4.002,~,~1,~4.0954$) \\
 &SD &($1,~42,~12.1002,~4.2000,3.2873,~13.2432$) & ($2.3412,~21.1771,5.1217,~3.9272,~1,4.2031$) \\
\hline 
\end{tabular}}
\label{table7a}
\end{table}

For the initial point $x_0 = (11,-1.5)^\top$, the decreasing behavior in the values of vector-valued functions at each iteration has been exhibited in Table \ref{table_7b}. 

\begin{table}[ht]
\centering
\caption{Output of Algorithm \ref{algo1} on Example \ref{example7} with initial point $(-11,-1.5)$}
\centering 
\scalebox{0.63}{
\begin{tabular}{c c c c c c} 
\hline
$k$& $x_k^\top$ &$f^{25}(x_k)$ & $f^{50}(x_k)$ & $f^{75}(x_k)$ & $f^{100}(x_k)$   \\ 
\hline 
 0&  $(-11,-1.5)$& $(0.8169,0.6536,0.8581)$& $(0.8133,0.6595,0.8627)$ &$(0.8008,0.6148,0.8481)$ & $(0.7973,0.6310,0.8536)$ \\  
 1&  $(10.9017,-1.4198)$& $(0.8154,0.6494,0.8566)$& $(0.8117,0.6552,0.8612)$ &$(0.7991,0.6093,0.8466)$ & $(0.7955,0.6261,0.8521)$ \\
2&  $(10.8018,-1.3395)$& $(0.8138,0.6451,0.8552)$& $(0.8100, 0.6508, 0.8598)$ &$(0.7974, 0.6034, 0.8450)$ & $(0.7937,  0.6209, 0.8506)$\\
3&  $(10.7001,-1.2589)$& $(0.8123, 0.6406, 0.8536)$& $(0.8084, 0.6462, 0.8583)$ &$(0.7957, 0.5973, 0.8435)$ & $(0.7918, 0.6155, 0.8490)$\\
4&  $(10.5967,-1.1782)$& $(0.8107,0.6359, 0.8521)$& $( 0.8066, 0.6415, 0.8568)$ &$(0.7940, 0.5908, 0.8418)$ & $(0.7899, 0.6099, 0.8475)$\\
5& $(10.4915,-1.0973)$& $(0.8090, 0.6311, 0.8505
)$& $(0.8049, 0.6365, 0.8552
)$ &$(0.7922,0.5839, 0.8402)$ & $(0.7880, 0.6041, 0.8459)$ \\
6&  $(10.3845,-1.0162)$& $(0.8074, 0.6260, 0.8489)$& $(0.8031, 0.6313, 0.8537)$ &$(0.7904, 0.5767,0.8385)$ & $(0.7860, 0.5979, 0.8443)$\\
7& $(10.2756,-0.9349)$ & $(0.8057, 0.6207,0.8473)$  & $(0.8012, 0.6258, 0.8521)$& $(0.7885,0.5689, 0.8368)$ &$(0.7840, 0.5915, 0.8426)$ \\
8&  ($10.1649,-0.8534$)& $(0.8039, 0.6152, 0.8456)$& $(0.7993, 0.6201,  0.8505)$ &$(0.7866,  0.5606, 0.8350)$ & $(0.7819, 0.5848, 0.8409)$ \\
9&  ($10.0565   -0.7725$)& $(0.8022, 0.6096,  0.8439)$& $(0.7975, 0.6143, 0.8488)$ &$(0.7848, 0.5520, 0.8333)$ & $(0.7798, 0.5780, 0.8393)$ \\
10&  ($9.9454,-0.6912$)& $(0.8005, 0.6038, 0.8422)$& $(0.7956, 0.6082, 0.8472)$ &$(0.7828, 0.5427, 0.8315)$ & $(0.7777,0.5707, 0.8375)$ \\
11&  ($9.8330,-0.6101$)& $(0.7987, 0.5977, 0.8405)$& $(0.7936, 0.6017, 0.8455)$ &$(0.7809, 0.5326, 0.8297)$ & $(0.7756,0.5632,  0.8358)$ \\
12&  ($9.7218,-0.5290$)& $(0.7969, 0.5914, 0.8388)$& $(0.7917, 0.5950, 0.8439
)$ &$(0.7789, 0.5219,  0.8279)$ & $(0.7734, 0.5554, 0.8341)$ \\
13&  ($9.6084,-0.4481$)& $(0.7951, 0.5849, 0.8370)$& $(0.7897, 0.5879, 0.8422
)$ &$(0.7769, 0.5101, 0.8260)$ & $(0.7712, 0.5473, 0.8323)$ \\
14&  ($9.4937,-0.3672$)& $(0.7933, 0.5780, 0.8352)$& $(0.7877, 0.5804,0.8404)$ &$(0.7749, 0.4972, 0.8242)$ & $(0.7690, 0.5388, 0.8305)$ \\
15&  ($9.3774   -0.2963$)& $(0.7914, 0.5710, 0.8335)$& $(0.7856, 0.5728,0.8388)$ &$(0.7729,0.4831,0.8223)$ & $(0.7667,0.5309,0.8288)$ \\
16&  ($9.2560,0.2475$)& $(0.7895,0.5634,0.8319)$& $(0.7835,0.5655, 0.8372)$ &$(0.7707,0.4674, 0.8207)$ & $(0.7644, 0.5248, 0.8272)$ \\
17&  ($9.1363, -0.2200$)& $(0.7875,0.5557, 0.8304)$& $(0.7815, 0.5589, 0.8359)$ &$(0.7686, 0.4500,0.8192)$ &$(0.7621, 0.5209, 0.8259)$ \\
18&  ($9.0278, -0.2110$)& $(0.7857, 0.5484, 0.8292)$& $(0.7797, 0.5535, 0.8348)$ &$(0.7666, 0.4322, 0.8180)$ &$(0.7601, 0.5195,0.8248)$ \\
19&  ($8.9339,-0.2151$)& $(0.7842,0.5417, 0.8283)$& $(0.7781, 0.5495, 0.8340)$ &$(0.7649, 0.4147, 0.8171)$ & $(0.7584, 0.5200, 0.8241)$ \\
\hline 
\end{tabular}}
\label{table_7b}
\end{table}
\end{example}


\section{Conclusion}\label{section7}

In this paper, we have studied set optimization problems with respect to the lower set less relation. The objective mapping is given by a finite number of twice continuously differentiable functions. We have proposed a quasi-Newton method (Algorithm \ref{algo1} to generate a sequence of iterates that converge to a point which satisfy a necessary condition for weakly minimal solutions of the problem. In the process of generating the sequence in Algorithm \ref{algo1}, we have approximated the Hessian matrices corresponding to given functions considered in Assumption \ref{assumption} with the help of BFGS methods \cite{broyden1969new,fletcher1970new,goldfarb1970family,shanno1970conditioning}. To generate the sequence, we have used a family of vector optimization problems \eqref{vp_equation}. Then, for a suitably chosen element $a_k$ from the partition set $P_{x_k}$ of the current iterate $x_k$, we have evaluated quasi-Newton direction $u_k$ (Step \ref{step3}) with the help of concepts in \cite{chuong2012steepest,drummond2014quadratically}. The process of generating iterates by Algorithm \ref{algo1} continued until the stopping condition (Step \ref{step4}) was met. We have discussed and ensured the well-definedness (Theorem \ref{critical}) of Algorithm \ref{algo1} with the existence of $(a^k,u_k)$ in Step \ref{step3} and the existence of a step length $t_k$ in Step \ref{step5} (Proposition \ref{armijo}). For deriving the convergence analysis Algorithm \ref{algo1}, we have derived the following results.
\begin{enumerate}[(i)]
\item We have proved a condition of nonstationarity of a point (Proposition \ref{critical}).
\item We analysed the boundedness of the generated sequence of quasi-Newton direction (Proposition \ref{bounded}).
\item We have derived the convergence of the generated sequence of iterates (Theorem \ref{convergence}) under a regularity condition (Definition \ref{regular_def}).
\item We proved the global superlinear convergence to a stationary point (Theorem \ref{superlinear}) of the generated sequence under a regularity condition and uniform continuity of BFGS approximations and with the help of Lemma \ref{bfgs_assumption}.
\end{enumerate}

Finally, we tested the performance of the proposed quasi-Newton method on some existing and freshly introduced numerical test problems in Section \ref{section6}. It is found that the proposed quasi-Newton method outperforms the steepest descent method. \\ 

{As a future direction, the proposed work can be tested for more practical problems similar to those discussed in \cite{ide2014relationship}}. In this paper, we have used the lower set less ordering to compare the given sets.  {Research can be performed on other ordering relations also (given in \cite{jahn2011new}). To study the proposed quasi-Newton method on these relations, the usual derivative concepts like epiderivatives or coderivatives need separate attention. Moreover, in this paper, we have used Armijo's step size condition to find the weakly minimal solution of \eqref{sp_equation}. This work can be extended to different step size conditions, such as strong Wolfe or Armijo-Wolfe conditions, and a comparison of the performance of the method can be observed. Further, we have used the Gerstewitz scalarizing function for treating the set optimization problems \eqref{sp_equation}. Future research can be performed on Hiriat-Urruty functional \cite{hiriart1979tangent}. There are several conventional optimization methods in the literature that can be generated for set-valued optimization problems. A comparison between the performance of these methods can be analysed (see \cite{lai2020q,mishra2020q,prudente2024global,hassan2019modified} and references therein). Future research can focus on devising a quasi-Newton method whose convergent analysis may not require the used assumption of regular point; Use of a strong Wolfe line search instead of Armijo condition may be of great help in this direction, as observed in \cite{kumar2024nonlinear}}.

\section*{Acknowledgement}
The authors are thankful to the anonymous reviewers and editors for their constructive comments to improve the quality of the paper. Debdas Ghosh acknowledges the financial support of the research grants MATRICS (MTR/2021/000696) and Core Research Grant (CRG/2022/001347) by the Science and Engineering Research Board, India. 
Anshika acknowledges a research fellowship from the Science and Engineering Research Board, India, with file number SB/S9/Z-03/2017-II (2022). 
Jen-Chih Yao is thankful for the research funding MOST 111-2115- M-039-001-MY2, Taiwan.

\section*{Disclosure statement}
There were no conflicts of interest reported by the authors of this paper.


\end{document}